\RequirePackage[l2tabu, orthodox]{nag}

\documentclass[oneside,article]{memoir}

  \usepackage{amsmath}

  \usepackage{amsthm,thmtools}

  \usepackage{enumitem}

  \usepackage{tikz,tikz-cd}
  \usetikzlibrary{matrix,arrows,decorations.markings,calc,positioning}

  \usepackage[colorlinks,citecolor=blue]{hyperref}

  \usepackage[capitalize,nosort]{cleveref}

  \usepackage[nobysame,alphabetic,initials]{amsrefs}

  \usepackage{amssymb}
  \usepackage{mathtools}
  \usepackage{stmaryrd}

  \usepackage{relsize}
  \usepackage{stackengine}

  \usepackage{euscript}
  \usepackage{mathrsfs}

  \usepackage{indentfirst}

  \usepackage[inline,nomargin]{fixme}
  \fxsetup{targetlayout=color}

  \usepackage{xspace}

  \usepackage{ifdraft}

  \isopage[10]
  \setlrmargins{*}{*}{1}
  \checkandfixthelayout

  \usepackage{fouridx}

  \numberwithin{equation}{section}

  \counterwithout{section}{chapter}

\newcommand{\defeq}{=_{\operatorname{def}}}

  \newcommand{\from}{\colon}

  \newcommand{\fto}{\twoheadrightarrow}

  \newcommand{\ito}{\hookrightarrow}

  \newcommand{\sto}{\rightarrowtriangle}

  \newcommand{\weto}{\mathrel{\ensurestackMath{\stackon[-2pt]{\xrightarrow{\makebox[.8em]{}}}{\mathsmaller{\mathsmaller\weq}}}}}

  \newcommand{\cto}{\rightarrowtail}

  \renewcommand{\iff}{if and only if}
  \newcommand{\st}{such that}
  \newcommand{\tfae}{the following are equivalent}
  \newcommand{\wrt}{with respect to}

  \newcommand{\lp}{lifting property}
  \newcommand{\llp}{left \lp}
  \newcommand{\rlp}{right \lp}
  
  \newcommand{\ollp}{ordinary \llp}
  
  \newcommand{\elp}{enriched \lp}
  \newcommand{\ellp}{enriched \llp}
  \newcommand{\erlp}{enriched \rlp}

  \newcommand{\smo}{small object argument}
  \newcommand{\wfs}{weak factorisation system}

  \newcommand{\ewfs}{enriched \wfs}

  \newcommand{\di}{complemented inclusion}
  \newcommand{\ldi}{levelwise \di}

  \newcommand{\cartesian}{Cartesian\xspace} 

  \newcommand{\cat}[1]{\mathcal{#1}}

  \newcommand{\cats}[1]{\mathsf{s}\cat{#1}}

\makeatletter
\newcommand{\splus}{\text{\tikz[baseline=(s.base)]{
        \node[outer sep=0pt,inner sep=0pt] (s) {$\mathsf{s}$};
        \path[use as bounding box] (s.north west) rectangle ($(s.south east)+(\f@size * 0.0125,0)$); 
        \node[scale=0.5,xshift=(0.4pt * \f@size),yshift=(-0.2pt * \f@size)] at (s.south east) (p) {$\boldsymbol{+}$}; 
    }}}
\makeatother

   \newcommand{\catss}[1]{{\splus\cat{#1}}}

  \newcommand{\ncat}[1]{\mathsf{#1}}

  \newcommand{\Cat}{\ncat{Cat}}

  \newcommand{\Set}{\ncat{Set}}

  \newcommand{\sSet}{\ncat{sSet}}

   \newcommand{\ssSet}{{\splus\ncat{Set}}}

  \newcommand{\Simp}{\Delta}

  \newcommand{\sSimp}{\Delta_+}

  \newcommand{\cof}{\mathsf{cof}}
  \newcommand{\fib}{\mathsf{fib}}

  \newcommand{\Fam}{\operatorname{\ncat{Fam}}}

  \newcommand{\sFam}{\operatorname{\ncat{sFam}}}

  \newcommand{\ssFam}{\operatorname{{\splus\ncat{Fam}}}}

  \newcommand{\con}{\mathsf{con}}

  \newcommand{\psh}{\operatorname{\ncat{Psh}}}

  \newcommand{\spsh}{\operatorname{\ncat{sPsh}}}

  \newcommand{\yon}{\mathrm{y}}

  \newcommand{\End}{\operatorname{\ncat{End}}}

  \newcommand{\Ho}{\operatorname{Ho}}

   \newcommand{\Hoi}{\operatorname{Ho_\infty}}

  \renewcommand{\emptyset}{{\mathord\varnothing}}

  \newcommand{\set}[2]{\left\{#1\mathrel{}\middle|\mathrel{}#2\right\}}

  \newcommand{\union}{\cup}

  \newcommand{\fset}[1]{\underline{#1}}

  \newcommand{\weq}{\mathrel\sim}

  \newcommand{\iso}{\mathrel{\cong}}

  \newcommand{\adj}{\dashv}

  \newcommand{\colim}{\operatorname*{colim}}

  \let\bigcoprod\coprod

  \renewcommand{\coprod}{\sqcup}

  \newcommand{\fibprod}{\operatorname*{\times}}

  \newcommand{\myend}{\textstyle\int}
  \newcommand{\coend}{\textstyle\int}

  \newcommand{\cotensor}{\mathop\pitchfork}

  \newcommand{\pull}{\times}

  \newcommand{\push}{\sqcup}

  \newcommand{\op}{{\mathord\mathrm{op}}}

  \newcommand{\slice}{\mathbin\downarrow}

  \newcommand{\fslice}{\mathbin\twoheaddownarrow}

  \newcommand{\Hom}{\operatorname{Hom}}

  \newcommand{\Prob}{\operatorname{Prob}}

  \newcommand{\pprod}{\mathbin{\hat{\times}}}

  \newcommand{\ev}{\operatorname{ev}}

\newcommand{\app}{\operatorname{app}}

  \newcommand{\bd}{\partial}

  \newcommand{\simp}[1]{\mathord\Delta[#1]}

  \newcommand{\bdsimp}[1]{\mathord{\bd\simp{#1}}}

  \makeatletter
  \def\horn#1{\expandafter\horn@i#1,,\@nil}
  \def\horn@i#1,#2,#3\@nil{\mathord\Lambda^{#2}[#1]}
  \makeatother

  \newcommand{\ssimp}[1]{\mathord\Delta_+[#1]}

  \newcommand{\bdssimp}[1]{\mathord{\bd\ssimp{#1}}}

  \makeatletter
  \def\shorn#1{\expandafter\shorn@i#1,,\@nil}
  \def\shorn@i#1,#2,#3\@nil{\mathord\Lambda_+^{#2}[#1]}
  \makeatother

  \newcommand{\Ex}{\operatorname{Ex}}

  \newcommand{\Sk}{\operatorname{Sk}}

  \newcommand{\gprod}{\boxtimes}

  \newcommand{\fml}[1]{\mathscr{#1}}

  \newcommand{\id}{\operatorname{id}}

  \newcommand{\uvar}{\mathord{\relbar}}
  \newcommand{\Uvar}{\mathord{\Relbar}}

  \renewcommand{\epsilon}{\varepsilon}

  \renewcommand{\phi}{\varphi}

  \renewcommand{\hat}{\widehat}

  \newcommand{\hatop}[1]{\mathbin{\hat{#1}}}

  \newcommand{\eqdef}{\mathbin{=_{\mathrm{def}}}}

\newcommand{\tensorSetE}{\mathbin{\cdot}}
\newcommand{\tensorsSetsE}{\mathbin{\cdot}}
\newcommand{\tensorEsE}{\times}

\newcommand{\cotensorsSetsE}{\mathbin{\pitchfork}}

\newcommand{\cotensorsSetsEslice}{\mathbin{\pitchfork}}

\newcommand{\pbcotensorsSetsE}{\mathbin{\hat{\cotensorsSetsE}}}

\newcommand{\corepresentable}[2][]{\Delta^\op[{#2}]_{#1}}
\newcommand{\corepresentabledeg}[2][]{\Delta_-^\op[{#2}]_{#1}}
\newcommand{\coboundary}[2][]{\partial \Delta^\op[{#2}]_{#1}}
\newcommand{\coboundarydeg}[2][]{\partial \Delta_-^\op[{#2}]_{#1}}

\declaretheorem[style=plain,within=section]{corollary}
\declaretheorem[style=plain,numberlike=corollary]{lemma}
\declaretheorem[style=plain,numberlike=corollary]{proposition}
\declaretheorem[style=plain,numberlike=corollary]{theorem}
\declaretheorem[style=definition,numberlike=corollary]{remark}
\declaretheorem[style=plain,numberlike=corollary]{conjecture}

\declaretheorem[style=plain,name=Corollary,within=section]{corollary-s}
\declaretheorem[style=plain,name=Lemma,numberlike=corollary-s]{lemma-s}
\declaretheorem[style=plain,name=Proposition,numberlike=corollary-s]{proposition-s}

\declaretheorem[style=definition,name=Remark,numberlike=corollary-s]{remark-s}

\Crefname{corollary}{Corollary}{Corollaries}
\Crefname{lemma}{Lemma}{Lemmas}
\Crefname{proposition}{Proposition}{Propositions}
\Crefname{theorem}{Theorem}{Theorems}
\Crefname{conjecture}{Conjecture}{Conjectures}
\Crefname{remark}{Remark}{Remarks}

\Crefname{corollary-s}{Corollary}{Corollaries}
\Crefname{lemma-s}{Lemma}{Lemmas}
\Crefname{proposition-s}{Proposition}{Propositions}
\Crefname{remark-s}{Remark}{Remarks}

  \newenvironment{tikzeq}[1]
  {
    \begingroup
    \begin{equation}\label{#1}
    \begin{tikzpicture}[baseline=(current bounding box.center)]
  }
  {
    \end{tikzpicture}
    \end{equation}
    \endgroup
    \ignorespacesafterend
  }

  \newenvironment{tikzeq*}
  {
    \begingroup
    \begin{equation*}
    \begin{tikzpicture}[baseline=(current bounding box.center)]
  }
  {
    \end{tikzpicture}
    \end{equation*}
    \endgroup
    \ignorespacesafterend
  }

  \tikzset
  {
    diagram/.style=
    {
      matrix of math nodes,
      column sep={4.3em,between origins},
      row sep={4em,between origins},
      text height=1.5ex,
      text depth=.25ex
    },
    over/.style={preaction={draw=white,-,line width=6pt}},
    every to/.style={font=\footnotesize},
    inj/.style={right hook->},
    surj/.style={-{Latex[open]}},
    cof/.style={>->},
    fib/.style={->>},
    strike/.style={decoration={markings,mark=at position 0.5 with {\arrow{|}}},postaction={decorate}},
    ano/.style={cof,strike},
    tfib/.style={fib,strike},
    tcof/.style={cof,strike},
  }

  \newcommand{\pb}[2]{\node at ($(#1)!0.25!(#2)$) {\tikz{\draw (3mm,0)--++(-90:3mm)--++(180:3mm)}}}

  \newcommand{\pbs}[2]{\node at ($(#1)!0.3!(#2)$) {\tikz{\draw (3mm,0)--++(-135:3mm)--++(180:3mm)}}}

  \newcommand{\pbdr}[2]{\node at ($(#1)!0.25!(#2)$) {\tikz{\draw (3mm,180)--++(90:3mm)--++(0:3mm)}}}
  \newcommand{\pbur}[2]{\node at ($(#1)!0.25!(#2)$) {\tikz{\draw (3mm,-90)--++(180:3mm)--++(90:3mm)}}}

  \newcommand{\MSC}[1]{%
  \let\thempfn\relax
  \footnotetext[0]{2020 Mathematics Subject Classification: #1.}
}

  \Crefname{subsection}{Subsection}{Subsections}

  \setlist[enumerate]{label=(\arabic*),itemsep=0ex}
  \Crefname{enumi}{Part}{Parts}
  \crefname{enumi}{part}{parts}

  \newlist{parts}{enumerate}{1}
  \setlist[parts]{label=\textup{(\roman*)},ref=\textup{(\roman*)}}
  \Crefname{partsi}{Part}{Parts}
  \crefname{partsi}{part}{parts}

  \newlist{conditions}{enumerate}{1}
  \setlist[conditions]{label=\textup{(\roman*)},ref=\textup{(\roman*)}}
  \Crefname{conditionsi}{Condition}{Conditions}
  \crefname{conditionsi}{condition}{Conditions}

  \newcommand{\axm}[1]{(#1\arabic*)}
  \newlist{axioms}{enumerate}{1}
  \Crefname{axiomsi}{Axiom}{Axioms}
  \crefname{axiomsi}{axiom}{axioms}

  \newlist{fibcat-axioms}{enumerate}{1}
  \setlist[fibcat-axioms]{label=\axm{F},ref=\axm{F},resume}
  \Crefname{fibcat-axiomsi}{Axiom}{Axioms}
  \crefname{fibcat-axiomsi}{axiom}{axioms}

  \makeatletter
  \newcommand{\customlabel}[2]{#2\def\@currentlabel{#2}\label{#1}}
  \makeatother

  \DeclareFontFamily{U}{mathx}{\hyphenchar\font45}

  \DeclareFontShape{U}{mathx}{m}{n}{
    <5> <6> <7> <8> <9> <10>
    <10.95> <12> <14.4> <17.28> <20.74> <24.88>
    mathx10}{}

  \DeclareSymbolFont{mathx}{U}{mathx}{m}{n}

  \DeclareMathAccent{\widebar}{0}{mathx}{"73}

  \DeclareFontFamily{U}{MnSymbolA}{}

  \DeclareFontShape{U}{MnSymbolA}{m}{n}{
    <-6> MnSymbolA5
    <6-7> MnSymbolA6
    <7-8> MnSymbolA7
    <8-9> MnSymbolA8
    <9-10> MnSymbolA9
    <10-12> MnSymbolA10
    <12-> MnSymbolA12}{}

  \DeclareSymbolFont{MnSyA}{U}{MnSymbolA}{m}{n}

  \DeclareMathSymbol{\twoheaddownarrow}{\mathrel}{MnSyA}{27}

\theoremstyle{definition}
\newtheorem{definition}[corollary]{Definition}
\newtheorem{example}[corollary]{Example}

\DeclarePairedDelimiter\braces\lbrace\rbrace

\makeatletter
\newcommand{\DeclareAbbrevation}[2]{\newcommand{#1}{\@ifnextchar{.}{#2}{#2.\@\xspace}}}
\makeatother

\DeclareAbbrevation{\ie}{i.e}
\DeclareAbbrevation{\eg}{e.g}
\DeclareAbbrevation{\cf}{cf}
\DeclareAbbrevation{\etc}{etc}
\DeclareAbbrevation{\resp}{resp}
\DeclareAbbrevation{\etal}{et al}
\DeclareAbbrevation{\ibid}{ibid}
\DeclareAbbrevation{\ca}{ca}
\DeclareAbbrevation{\vs}{vs}

\author{Nicola Gambino \and Simon Henry \and Christian Sattler \and Karol Szumi\l{}o}
\title{The effective model structure and $\infty$-groupoid objects}
\date{\today\thanks{This version of the paper reflects the one published in
{\em Forum of Mathematics, Sigma} (2022),  Vol. 10:e34 1--59, submitted 9 March 2021,
revised 19 January 2022, accepted 7 February 2022, available
\href{https://www.cambridge.org/core/services/aop-cambridge-core/content/view/D0AF56CC8BF93DC546007D5C6352E0E0/S2050509422000135a.pdf/the-effective-model-structure-and-dollarinfty-dollar-groupoid-objects.pdf}{here}.
 Two additional minor typos have been fixed.}}

\begin{document}

  \maketitle

\begin{abstract}
For a category $\cat E$ with finite limits and well-behaved countable coproducts, we construct
a model structure, called the effective model structure, on the category of simplicial objects
in~$\cat E$, generalising the Kan--Quillen model structure on simplicial sets. We then prove that the effective
model structure is left and right proper and satisfies descent in the sense of Rezk. As a consequence,
we obtain that the associated $\infty$-category has finite limits, colimits satisfying descent,
and is locally \cartesian closed when $\cat E$ is, but is not a higher topos in general.
We also characterise the $\infty$-category presented by the effective model structure, showing that it is the full sub-category of presheaves on $\cat E$ spanned by Kan complexes in $\cat E$, a result that suggests a close analogy with the theory of exact completions.
\MSC{18N40 (primary), 18N60, 55U10}
\end{abstract}

\section*{Introduction}

\noindent
\textbf{Context and motivation.}
Over the past two decades, there has been an explosion of interest in the connections between model categories and higher categories~\cites{cisinski2020higher,Gepner-Kock,Joyal-Tierney,Lurie,Rezk-Segal,Szumilo}. This
line of research led to the reformulation of significant parts of modern homotopy theory in terms of higher category theory, the development of  higher topos theory~\cites{Toen-Vezzosi,Lurie} and is of great importance for Homotopy Type Theory and the Univalent Foundations programme~\cites{Awodey-Warren,berg2018univalent,Gepner-Kock,Kapulkin-Lumsdaine,Shulman}. Central to these developments are model structures on categories of simplicial objects, \ie, functor categories of the form~$\cats E = [\Delta^\op,
\cat E]$, where $\cat E$ is a category, as considered in~\cite{Quillen}*{Section~II.4}, \cite{Goerss-Jardine}*{Chapter~II}, \cite{christensen2002quillen}*{Theorem~6.3} and \cite{Hormann}. In particular, the category of simplicial sets equipped with the Kan--Quillen model structure~\cite{Quillen} can be understood
as a presentation of the $\infty$-category of spaces, while categories of simplicial presheaves and
sheaves (\ie, simplicial objects in a Grothendieck topos) equipped with the Rezk model structure~\cite{Rezk} and the Joyal--Jardine model structure~\cites{Brown,Joyal,Jardine} can be seen as presentations of $\infty$-toposes and their hypercompletions, respectively~\cites{DHI,Lurie}.

The main contribution of this paper is to construct a new model structure, which we call the {\em effective model structure}, on categories of simplicial objects $\cats E$, assuming that $\cat E$ is merely a
countably lextensive category, \ie, a category with finite limits and countable coproducts, where the latter are required to be van Kampen colimits~\cites{CLW,Rezk}. The effective model structure is defined
so that when~$\cat{E} = \Set$, we recover the Kan--Quillen model structure on simplicial sets~\cite{Quillen}.
We also prove several results on the effective model structure and its associated $\infty$-category, which we
discuss below.

The initial motivation for this work was the desire to establish whether our earlier work on the constructive Kan--Quillen model structure~\cites{H,Gambino-Sattler,GSS,Sattler} could be developed further so as to obtain a new model structure on categories of simplicial sheaves. Indeed, in \cites{H,GSS} we worked with simplicial sets without using the law of excluded middle and the axiom of choice, thus opening the possibility of replacing them with simplicial objects in a Grothendieck topos. As we explored this idea, we realised that the resulting argument  admitted not only a clean presentation in terms of enriched weak factorisation systems~\cite{RiehlE:catht}*{Chapter~13}, but also a vast generalisation.

In fact, the existence of the effective model structure may be a surprise to some readers, since
assuming~$\cat E$ to be countably lextensive is significantly weaker than assuming it to be a Grothendieck topos and covers many more examples (such as the category of countable sets and the category of schemes). In particular, our
arguments do not require the existence of all small colimits, (local) \cartesian closure and local presentability, which are ubiquitous in the known constructions of model structures.

One reason for the interest in the effective model structure is that, when~$\cat{E}$ is a Grothendieck topos, the effective model structure on~$\cats E$ differs from the known model structures on simplicial sheaves and provides the first example of a peculiar combination of higher categorical structure. Indeed, the
associated $\infty$-category  has finite limits, colimits that satisfy descent and is locally \cartesian closed,
but is neither a higher Grothendieck topos~\cite{Lurie} nor a higher elementary topos~in the sense
of~\cites{Shulman-ncatcafe,Rasekh}, since its 0-truncation does not always have a subobject classifier (see~\cref{ex:arrow_sets2}). In this case, the effective model structure satisfies most of  the axioms for a model topos~\cite{Rezk}, but is not combinatorial. One key point here is that the effective model structure is not cofibrantly generated in the usual sense, but only in an enriched sense. The relation between the effective model structure and other model structures on categories of simplicial objects is discussed
further in \cref{rmk:comparison}.

This situation can be understood by analogy with the theory of exact completions in ordinary category theory~\cite{Carboni-Vitale}. There, it is known that the exact completion of a (Grothendieck) topos need not be a (Grothendieck) topos~\cite{Menni}. Indeed, we believe that the effective model structure will provide a starting point for the development of a homotopical counterpart of the theory of exact completions. As a first step in this direction, we prove that the $\infty$-category associated to the effective model structure on $\cats{E}$ is the full subcategory of the $\infty$-category of presheaves on~$\cat{E}$ spanned by Kan complexes in~$\cat E$, mirroring a corresponding description of the exact completion of~$\cat{E}$ in~\cite{Hu-Tholen}. We also make a conjecture (\cref{thm:conj-exlex}) on the relation between the effective model structure and $\infty$-categorical exact completions, which we leave for future work. In the long term, we hope that our work could be useful for the definition of a higher categorical version of the effective topos~\cite{Hyland}, which can be described as an exact completion~\cite{Carboni}.

Finally, our results may be of interest also in Homotopy Type Theory, since they help to clarify how the simplicial model of Univalent Foundations~\cite{Kapulkin-Lumsdaine}, in which types are interpreted as Kan complexes, is related to the setoid model of type theory~\cite{Hofmann}, in which types are interpreted as types equipped with an equivalence relation, by showing how not only the latter~\cite{Emmenegger-Palmgren} but also
the former is related to the theory of exact completions. Furthermore, we expect that
the effective model structure may lead to new models of Homotopy Type Theory, another topic that we leave for future research.

\smallskip

\noindent
\textbf{Main results.} In order to outline our main results, let us briefly describe the effective model structure, whose fibrant objects are to be thought of as Kan complexes, or $\infty$-groupoids, in~$\cat{E}$. In order to describe the fibrations of the effective model structure,  recall that, for $E \in \cat{E}$, we have a functor
\begin{equation}
\tag{$\ast$}
\Hom_\sSet(E, -)  \from \cats{E} \to \sSet
\end{equation}
sending $X \in \cats{E}$ to the simplicial set defined by $\Hom_\sSet(E,X)_n = \Hom(E, X_n)$, for $[n] \in \Delta$.  We can then define a map in $\cats{E}$ to be a fibration in $\cats{E}$ if its image under the functor in~($\ast$) is a Kan fibration in $\sSet$ for every $E \in \cat{E}$. Trivial fibrations are defined
analogously. Our main results are the following:
\begin{itemize}
\item \cref{model}, asserting the existence of the effective model structure, whose fibrations and trivial fibrations are defined as above;
\item \cref{thm:right-proper} and \cref{sE-left-proper}, asserting that the effective model structure is
right and left proper, respectively, and \cref{ms-descent-pushout}, showing that homotopy colimits in $\cats E$
satisfy descent;
\item  \cref{thm-descent-infty}  asserting that the $\infty$-category $\Ho_\infty(\cats{E})$ associated to the effective model structure has finite limits and $\alpha$-small colimits satisfying descent when $\cat E$ is $\alpha$-lextensive, and \cref{thm-lccc-infty}, showing that $\Ho_\infty(\cats{E})$ is also locally \cartesian closed when~$\cat E$ is so.
\item \cref{th:Joyal_Szumilo_conjecture}, characterising the $\infty$-category associated to the effective model structure.
\end{itemize}

Along the way, we prove several other results of independent interest. For example, we characterise completely the cofibrations of the effective model structure, which do not coincide with all monomorphisms (\cref{lem:cof_charac}) and we compare the effective model structure
with model structures studied in relation to Elmendorf's Theorem (\cref{th:Elmendorf}).

\smallskip

\noindent
\textbf{Novel aspects.} This paper differs significantly from our work in~\cites{H,GSS,Sattler} in both
scope and technical aspects. Regarding scope, apart from generalising the existence of the
model structure from the case $\cat E = \Set$ to that of a general countably lextensive category $\cat E$,
here we discuss a number of topics that are not even mentioned  for the case $\cat E = \Set$ in our earlier work, such as the structure and characterisation of the $\infty$-category associated to the effective model structure, the discussion of descent and the connections with Elmendorf's theorem.

Regarding the technical aspects, even if the general strategy for proving the existence of the effective model structure is inspired from the case $\cat E = \Set$ in~\cites{GSS}, several new ideas are necessary to implement it to the general case, as we explain below. This strategy involves in three steps. First, we introduce the notions of a (trivial) fibration in $\cats E$ as above and establish the existence of a fibration category structure on the category of Kan complexes (assuming in fact only that $\cat E$ has finite limits). Secondly, we construct the two weak factorisation systems of the model structure, one given by cofibrations and trivial fibrations and one given by trivial cofibrations and fibrations. Thirdly, we show that weak equivalences (as determined by the two weak factorisation systems) satisfy 2-out-of-3 by proving the so-called Equivalence Extension Property (\cref{equivalence-extension}).

In order to realise this plan, we prove several results that are not necessary for $\cat E =  \Set$. We mention only the key ones.
First, we develop a new version of the enriched small object argument (\cref{esmo}), which does not require existence of all colimits. In order to achieve it, we analyse the colimits required for our applications and prove that they exist in a countably lextensive category, exploiting crucially that some of the maps involved are complemented monomorphisms. Secondly, we show that the fibration category structure, where fibrations are defined as above, agrees with the weak factorisation systems, defined in terms  of enriched lifting properties (\cref{fibration-levelwise}). Thirdly, we obtain a characterisation of cofibrations in categories of simplicial objects (\cref{lem:cof_charac}), which requires a  new, purely categorical, argument that is
entirely different to the one used in~\cites{H,GSS,Sattler}. Finally, new ideas are required  in the proof of the Equivalence Extension
Property~(\cref{equivalence-extension}). For this, we need to construct explicitly dependent products (\ie, pushforward) functors along cofibrations~(\cref{th:Dependent_prod_along_cof}), which are not guaranteed to exist since $\cat E$ is not assumed to be locally \cartesian closed. The existence of these pushforward functors may be considered as a pleasant surprise since they are essential for our argument and no exponentials are assumed to be present in $\cat E$.


The existence of the effective model structure is independent from that of the constructive Kan--Quillen model structure on simplicial sets~\cites{GSS,H}. Actually, the use of enriched category theory here, especially for expressing stronger versions of the lifting properties usually phrased in terms of mere existence of diagonal fillers, makes  explicit some of the informal conventions adopted in~\cites{GSS,H} when treating the case~$\cat E = \Set$. Also, the proofs in~\cites{GSS,H} make use of structure on $\Set$ that is not available in a countably lextensive category and therefore cannot be interpreted as taking place in the so-called internal logic of~$\cat E$~\cite{Elephant}*{Section~D1.3}. Even when $\cat E$ is a Grothendieck
topos, carrying out the proofs in the internal language of $\cat E$~\cite{maclane-moerdijk}*{Chapter~6}
would not make explicit the structure under consideration, thus making it more difficult for the results
to be accessible and applicable.

\smallskip

\noindent
\textbf{Outline of the paper.} The paper is organised in four parts.
The first, including only \cref{sec:fib_cat}, establishes the fibration category structure.
The second, including \cref{extensive-categories,sec:enrwfs,sec:cofibrations}, introduces
the two weak factorisation systems, having first developed an appropriate version of
the small object argument.
The third, including \cref{sec:prop-of-cof,sec:dependent-products,sec:frobenius,sec:equivalence-extension,sec:model-structure}, establishes the existence of the effective model structure, by constructing pushforward functors
and establishing the Frobenius and Equivalence Extension Property.
The fourth, including \cref{sec:descent,sec:Elmendorf,sec:Semi-simplicial,sec:the-infty-category}, proves
the key properties of the effective model structure, namely descent and properness, their $\infty$-categorical
counterparts,  and characterises
its associated $\infty$-category. Throughout the paper, we omit the proofs that can be carried out with minor modifications from~\cites{H,GSS},
but include the ones that require new ideas.

\smallskip

\noindent
\textbf{Remark.} The material in this paper is developed within ZFC set theory. Some of the material,
however, can be developed also in a constructive setting (see footnotes and \ref{app:constructivity} for details).

\smallskip

\noindent
\textbf{Acknowledgements.} We are grateful to Andr\'e Joyal for  questions and discussions, which in particular led to \cref{th:Joyal_Szumilo_conjecture}, and to Dan Christensen and Denis-Charles Cisinski for comments on an earlier version of the paper.  Nicola Gambino and Karol Szumi{\l}o gratefully acknowledge that this material is based upon work supported by the US Air Force Office for Scientific Research under awards number FA8655-13-1-3038 and FA9550-21-1-0007. Nicola Gambino was also supported by EPSRC under grant EP/V002325/1.
Simon Henry was partially supported by an NSERC Discovery Grant.
Christian Sattler was supported by Swedish Research Council grant 2019-03765.

%


  \section{Kan fibrations}\label{sec:fib_cat}

This section develops some simplicial homotopy theory in a category $\cat{E}$ with finite limits. The category of simplicial objects in $\cat{E}$ is defined by letting
\[
\cats{E} \defeq [\Simp^\op, \cat{E}] \, .
\]
In \cref{def:fibration_algebraic} we introduce the notion of a fibration in $\cats{E}$ with which we shall work throughout the paper. This notion is defined using the
enrichment of  $\cats{E}$ in $\sSet$ and generalises that of a Kan fibration in $\sSet$. The main result of this section, \cref{fibcat-Kan}, establishes a structure of a fibration category on the category of fibrant objects in $\cats{E}$. For applications throughout the paper, we also establish a fiberwise version of this fibration category in \cref{fibcat-fiberwise}. We also introduce the notion of a \emph{pointwise weak equivalence} (\cref{pwe}), which provides the weak equivalences of these fibration categories. In the subsequent sections we will extend these results to obtain the effective model structure on~$\cats E$, under the stronger assumption that~$\cat{E}$ is countably lextensive.
The weak equivalences of the effective model structure will not be the pointwise weak equivalences in general, although the two notions will coincide  for maps between fibrant objects.

Let us recall how the category $\cats{E}$ is enriched over $\sSet$ \wrt{} the \cartesian monoidal structure.
For a finite simplicial set $K$ and $X \in \cats{E}$, we define $K \cotensor X \in \cats E$ via the end formula
\begin{equation}
\label{equ:cotensor}
(K \cotensor X)_m \eqdef \coend_{[n] \in \Simp} X_n^{(K \times \simp{m})_n} \text{.}
\end{equation}
For $X, Y \in \cats{E}$, the simplicial hom-object  is then defined by letting\footnote{Here and in the following we use subscripts to indicate to which category
the hom-objects under consideration belong.}
\begin{equation}
\label{equ:sset-enrichment}
  \Hom_{\sSet}(X, Y)_m \eqdef \Hom_\Set( X, \simp{m} \cotensor Y ) \text{.}
\end{equation}
This makes $\cats E$ into a $\sSet$-enriched category so that the formula in~\eqref{equ:cotensor} gives the cotensor (over finite simplicial sets) with respect
to the enrichment.
Without further assumptions on $\cat E$, $\cats E$ does not admit all cotensors or tensors over simplicial sets.
We often identify an object $E \in \cat E$ with the constant simplicial object with value $E$. For example,
for $E \in \cat E$ and $Y \in \cats E$ we write $\Hom_\sSet(E, Y)$. Note that
\begin{align*}
  \Hom_\sSet(E, Y)_m & = \Hom_\Set(E, Y_m) \text{,} \\
 \Hom_\sSet(E, K \cotensor Y) & = K \cotensor \Hom_\sSet(E, Y) \text{.}
\end{align*}

The $\sSet$-enrichment allows us to define a notion of a homotopy between morphisms of $\cats{E}$.
Given maps $f_0, f_1 \from X \to Y$ in $\cats{E}$ (or one of its slice categories), a \emph{homotopy} $H$ from $f_0$ to $f_1$, written $H \from f_0 \sim f_1$,
is a map
\begin{equation}
\label{equ:homotopy}
H \from X \to \simp{1} \cotensor Y
\end{equation}
that restricts to $f_0$ on $\braces{0} \to \simp{1}$ and to $f_1$ on $\braces{1} \to \simp{1}$.
It is \emph{constant} if it factors through the canonical map $\simp{0} \cotensor Y \to \simp{1} \cotensor Y$, in which case $f_0 = f_1$.
Note that we can regard $H$ as a map $\simp{1} \to \Hom_{\sSet}(X, Y)$.
This generalises the usual notion of homotopy in simplicial sets.
For each $E \in \cat E$, the functor $\Hom_\sSet(E, \uvar)$ preserves homotopies because it preserves the cotensor with $\simp{1}$.




We need some definitions to introduce the notions of a Kan fibration and trivial Kan fibration in $\cats{E}$.
For a finite simplicial set $K$, we define the \emph{evaluation functor} $\ev_K \from \cats{E} \to \cat{E}$  via the end formula
\begin{equation}
\label{equ:evaluation}
  \ev_K(X) = X(K) \eqdef \coend_{[n] \in \Simp} X_n^{K_n} \text{.}
\end{equation}
We will usually write $X(K)$ rather than $\ev_K(X)$ for brevity.
However, in some situations the notation $\ev_K(X)$ will be more convenient, see the definition of pullback evaluation below.
The end above exists since, by the finiteness of $K$, it can be constructed from finite limits.
For example, $X(\simp{n}) = X_n$ and $X(\horn{2,k}) = X_1 \times_{X_0} X_1$.  Also note that
$X(K) = (K \cotensor X)_0$ and $X(K \times \simp{m}) = (K \cotensor X)_m$.

\begin{remark} \label{evaluation}
There are two alternative ways of viewing the evaluation functor. First, since $\cat E$ has finite limits, we can consider $X(K)$
as the value on $K$ of the right Kan extension of $X \colon \Delta^\op \to \cat{E}$ along the  inclusion
of $\Delta$ into the category of finite simplicial sets. Secondly, seeing $\cat{E}$ as a $\Set$-enriched category,
we can view $X(K)$ as a weighted limit, namely the limit of $X$, viewed as a diagram in $\cat{E}$, weighted
by $K$, viewed as a diagram in $\Set$.
Both of these observations show that~$X(K)$ is contravariantly functorial in $K$.
\end{remark}

We write $\hat{\ev}$ for the \emph{pullback evaluation} functor, which is the result  of applying
the so-called Leibniz construction~\cite{Riehl-Verity} to the two-variable functor $\ev$, \ie, the functor
sending a map $i \from A \to B$ between finite simplicial sets and a morphism $f \from X \to Y$ of $\cats{E}$ to
\begin{align}
\label{equ:pullback-evaluation}
  \hat{\ev}_i(f) \from \ev_{A}(X) \to \ev_{B}(X) \times_{\ev_{B}(Y)} \ev_{A}(Y) \text{ in } \cat{E} \\
  \nonumber \text{also written as } \hat{\ev}_i(f) \from X(A) \to X(B) \times_{Y(B)} Y(A) \text{.}
\end{align}

\begin{remark} \label{thm:leibniz-terminology}
We adopt the convention of prefixing with `pullback' (or `pushout') the name of a two-variable functor to indicate the
result of applying the Leibniz construction to it. So for example, we shall say pushout product for what is
also referred to as Leibniz product or corner product.
\end{remark}

We use standard notation for the sets of boundary inclusions and horn inclusions,
\begin{equation}
\label{equ:boundary-horns}
  I_\sSet = \{ \bdsimp{n} \to \simp{n} \mid n \ge 0 \} \text{ and } J_\sSet = \{ \horn{n,k} \to \simp{n} \mid n \ge k \ge 0, n > 0 \} \text{.}
\end{equation}


\begin{definition} \label{def:fibration_algebraic}
We say that a morphism in $\cats{E}$ is
\begin{itemize}
\item a \emph{trivial Kan fibration} if its pullback evaluations with all maps in $I_\sSet$ are split epimorphisms;
\item a \emph{Kan fibration} if its pullback evaluations with all maps in $J_\sSet$ are split epimorphisms.
\end{itemize}
\end{definition}
Explicitly, a map $f \colon X \to Y$ in $\cats E$ is a Kan fibration if the morphism
\begin{equation*}
  X(\simp{n}) \to X(\horn{n,k}) \times_{Y(\horn{n,k})} Y(\simp{n})
\end{equation*}
in $\cat{E}$ has a section, for all $n \ge k \ge 0$ and $n > 0$.
For $Y = 1$, this means that the morphism
\begin{equation*}
  X(\simp{n}) \to X(\horn{n,k})
\end{equation*}
has a section,  for all $n \ge k \ge 0$ and $n > 0$,
in which case we say that $X$ is a \emph{Kan complex}.
Note that for $\cat{E} = \Set$, these definitions reduce to the standard notions of a Kan fibration, trivial Kan fibration
and a Kan complex in simplicial sets.
In the following, we shall frequently write \emph{fibration}, \emph{trivial fibration} and \emph{fibrant object},
as we do not consider other notions of fibrations.

Although we have not yet introduced cofibrations and trivial cofibrations in $\cats{E}$, we
can use the standard classes of cofibrations and trivial cofibrations in $\sSet$, which are
the saturations of the generating sets $I_\sSet$ and $J_\sSet$, respectively.

The next proposition characterises fibrations and trivial fibrations by reducing them to the corresponding notions in $\sSet$
in terms of the $\sSet$-enrichment of $\cats{E}$, defined in~\eqref{equ:sset-enrichment}.

\begin{proposition}\label{prop:fibration_pointwise}
  Let $f \from X \to Y$ be a map in $\cats E$.
Then $f$  is a (trivial) fibration if and only if, for all $E \in \cat E$, the map
\begin{equation*}
  \Hom_\sSet(E, f) \from \Hom_\sSet(E,X) \to \Hom_\sSet(E,Y)
\end{equation*}
is a (trivial) fibration in $\sSet$.
\end{proposition}

\begin{proof}
  Note that the functors $X(\uvar) \from \sSet^\op \to \cat{E}$ and $\Hom_\sSet(\uvar, X) \from \cat{E}^\op \to \sSet$ are contravariantly adjoint.
  Thus for all maps $i \from A \to B$ between finite simplicial sets there is a bijective correspondence between the lifting problems
  \begin{tikzeq*}
  \matrix[diagram,column sep={between origins,8em}]
  {
    |(A)| A & |(EX)| \Hom_\sSet(E, X) &         & |(XB)| X(B)                   \\
    |(B)| B & |(EY)| \Hom_\sSet(E, Y) & |(E)| E & |(P)|  X(A) \pull_{Y(A)} Y(B) \\
  };

  \draw[->] (A)  to (B);
  \draw[->] (EX) to (EY);
  \draw[->] (A)  to (EX);
  \draw[->] (B)  to (EY);

  \draw[->,dashed] (B) to (EX);

  \draw[->] (E)  to (P);

  \draw[->] (XB) to node[right] {$\hat{\ev}_i(f)$} (P);

  \draw[->,dashed] (E) to (XB);
  \end{tikzeq*}
  the latter of which is equivalent to the morphism on the right being a split epimorphism (by setting $E = X(A) \pull_{Y(A)} Y(B)$).
\end{proof}


If $i \from A \to B$ is a map of finite simplicial sets and $p \from X \to Y$ is a morphism of $\cats{E}$, then we define
the \emph{pullback cotensor} of $i$ and $p$ (\cf \cref{thm:leibniz-terminology}) as the induced morphism
\begin{equation*}
  i \pbcotensorsSetsE p \from B \cotensorsSetsE X \to (A \cotensorsSetsE X) \pull_{A \cotensorsSetsE Y} (B \cotensorsSetsE X) \text{.}
\end{equation*}

\begin{lemma}\label{lem:basic_pointwise_fibrations} \leavevmode
  \begin{parts}
  \item\label{item:pbcot}
  The pullback cotensor in $\cats{E}$ of  a cofibration between finite simplicial sets and a fibration is a fibration.
  If the given cofibration or fibration  is trivial, then the result is a trivial fibration.
  \item\label{item:retract_pb_com} Fibrations and trivial fibrations in $\cats E$ are closed under composition, pullback, and retract.
  \item\label{item:tfib_right_cancel} Let $f \from X \to Y$ and $g \from Y \to Z$ be morphisms of $\cats{E}$. If $f \from X \to Y$ and $g f \from X \to Z$ are trivial fibrations, then so is $g \from Y \to Z$.
  \end{parts}
\end{lemma}

\begin{proof}
  All the statements are proved in the same way: they hold for simplicial sets
  (see, \eg,~\cite{Quillen}*{Theorem II.3.3}) and transfer to $\cats{E}$
  using \cref{prop:fibration_pointwise}.\footnote{Constructively, \cref{item:pbcot} is true in $\sSet$ by~\cite{GSS}*{Corollary~1.3.4}, \cref{item:retract_pb_com} is evident
    and \cref{item:tfib_right_cancel} is~\cite{GSS}*{Lemma~1.3.6}.}
  Note that transferring \ref{item:pbcot} from $\sSet$ to $\cats{E}$ relies on the fact that $\Hom_\sSet(E, \uvar)$ preserves pullbacks and cotensors
  and hence pullback cotensors.
\end{proof}


\begin{definition}\label{pwe}
Let $f \from X \to Y$ in $\cats E$. We say that $f$ is a \emph{pointwise weak equivalence} if
\begin{align*}
  \Hom_\sSet(E,f) \from \Hom_\sSet(E,X) \to \Hom_\sSet(E,Y)
\end{align*}
is a weak equivalence in $\sSet$ for all $E \in \cat{E}$.
\end{definition}

For the next theorem, we use the definition of a fibration category as stated in~\cite{GSS}*{Section~1.6}.

\begin{theorem}\label{fibcat-Kan}
Let $\cat{E}$ be category with finite limits.
Then pointwise weak equivalences, Kan fibrations and trivial Kan fibrations
equip the category of Kan complexes in~$\cats{E}$ with the
structure of a fibration category.
\end{theorem}

\begin{proof}
  Trivial fibrations are exactly the fibrations that are weak equivalences because this holds in $\sSet$.
  We need to verify the following axioms.
  \begin{fibcat-axioms}
  \item $\cats{E}$ has a terminal object and all objects are fibrant, which follows directly from the definitions.
  \item Pullbacks along fibrations exist because $\cat{E}$ (and hence $\cats{E}$) has all finite limits.
    Moreover, fibrations and acyclic fibrations are closed under pullback by point \ref{item:retract_pb_com} of \cref{lem:basic_pointwise_fibrations}.
  \item Every morphism factors as a weak equivalence followed by a fibration.
    By~\cite{Brown}*{p.~421, Factorization lemma} it suffices to construct a path object, i.e.,
    a factorisation of the diagonal $X \to X \times X$.
    Such factorisation is given by the cotensor $X \to  \simp{1} \cotensor X \to X \times X$.
    Applying $\Hom_\sSet(E,\uvar)$ to this factorisation gives
    \begin{equation*}
      \Hom_\sSet(E, X) \to \simp{1} \cotensor \Hom_\sSet(E, X) \to \Hom_\sSet(E, X) \times \Hom_\sSet(E, X)
    \end{equation*}
    which is a well known factorisation of the diagonal of $\Hom_\sSet(E, X)$ into a weak equivalence followed by a fibration in $\sSet$
    (since $\Hom_\sSet(E, X)$ is a Kan complex by \cref{prop:fibration_pointwise}).
    See, \eg,~\cite{Goerss-Jardine}*{p.~43}.
    Hence $X \to \simp{1} \cotensor X \to X \times X$ is also such factorisation in $\cats{E}$.
  \item Weak equivalences satisfy 2-out-of-6, which follows since this property holds in $\sSet$. \qedhere
  \end{fibcat-axioms}
\end{proof}

In view of our development in \cref{sec:equivalence-extension}, we generalise \cref{fibcat-Kan} to the case of a slice of $\cats E$ over a simplicial object $X$,
which we write $\cats{E} \slice X$. We then define  $\cats{E} \fslice X$ to be the full subcategory of $\cats{E} \slice X$ spanned by the fibrations over $X$.

First of all,  let us recall that the enrichment of $\cats{E}$ in simplicial sets, including the cotensor with finite simplicial sets, descends to its slices.
For $(A, f), (B, g) \in \cats{E} \slice X$, the hom-object $\Hom_\sSet((A, f), (B, g))$ is
the pullback of $\Hom_{\sSet}(A, B)$ along the map $f \from 1 \to \Hom_{\sSet}(A, X)$.
The cotensor of $(A, f) \in \cats{E} \slice X$ by a finite simplicial set~$K$ is
the pullback of $K \cotensor A$ along the map $X \to K \cotensor X$ (using the fact that the monoidal unit in $\sSet$ is the terminal object).
As before, for each $E$, the functor $\Hom_\sSet(E, \uvar) \from \cats{E} \slice X \to \sSet \slice \Hom_\sSet(E, X)$ preserves these cotensors.

\begin{lemma} \label{pullback-cotensor-slice}
  Let $X \in \cats{E}$.
  The pullback cotensor properties in \cref{item:pbcot} of \cref{lem:basic_pointwise_fibrations}
  hold in $\cats{E} \slice X$ as well.
\end{lemma}

\begin{proof}
  This follows from their validity in~$\cats E$, \ie, \cref{item:pbcot} of \cref{lem:basic_pointwise_fibrations} and the stability of fibrations and trivial fibration under pullback, \ie, \cref{item:retract_pb_com} of \cref{lem:basic_pointwise_fibrations}.
\end{proof}

\begin{theorem}\label{fibcat-fiberwise} Let $X \in \cats{E}$. Then pointwise weak equivalences, fibrations and trivial fibrations equip the category $\cats{E} \fslice X$ with
the structure of a fibration category.
\end{theorem}

\begin{proof}
  All axioms are verified by the same argument as in the proof of \cref{fibcat-Kan}.
  For (F3), we use \cref{pullback-cotensor-slice} which is a fiberwise version of \cref{item:pbcot} of \cref{lem:basic_pointwise_fibrations}
  used in the proof of \cref{fibcat-Kan}.
\end{proof}

 We conclude this section with a basic observation on homotopy equivalences.

\begin{proposition} \label{h-equiv-is-we}
  Homotopy equivalences in $\cats{E}$ (and in particular, in $\cats{E} \slice X$ for all $X \in \cats{E}$) are pointwise weak equivalences.
\end{proposition}

\begin{proof}
  The functors $\Hom_\sSet(E, \uvar)$ preserve homotopies and hence also homotopy equivalences.
  Thus the conclusion follows from the fact that homotopy equivalences are weak equivalences in~$\sSet$.
\end{proof}

  \section{Lextensive categories and complemented inclusions}\label{extensive-categories}

This section, \cref{sec:enrwfs} and \cref{sec:cofibrations} constitute the second part of the paper,
whose ultimate goal is to construct two weak factorisation systems on $\cats E$, whose right
classes of maps are the fibrations and trivial fibrations of \cref{sec:fib_cat}, assuming that
$\cats E$ is a countably lextensive category. This section recalls some basic facts about lextensive categories.
Throughout it, we consider a fixed category with finite limits~$\cat{E}$
and study diagrams in~$\cat{E}$ indexed by a category $D$.
When convenient, we will regard cones under such diagrams as diagrams over
the category~$D^\rhd$, obtained by adding a new terminal object~$\star$ to~$D$.
We start by recalling the general notion of van Kampen colimit~\cites{Lurie,Rezk} in our setting.

\begin{definition} Let $Y_\bullet  \colon D \to \cat{E}$ be a diagram and assume
   $Y_\star = \colim_{d \in D} Y_d$ is its colimit in $\cat{E}$. We say that $Y_\star$ is
   \begin{parts}
  \item \emph{universal}, if it is preserved by pullbacks, i.e., if for every map $X_\star \to Y_\star$,
    $X_\star$ is the colimit of the induced diagram $X_d = X_\star \fibprod_{Y_\star} Y_d$.
  \item \emph{effective}, if given a \cartesian natural transformation $X \to Y$, the diagram $X$ has a colimit $X_\star$,
    and all the squares
    \begin{tikzeq*}
      \matrix[diagram]
      {
        |(Xd)| X_d & |(Xs)| X_\star \\
        |(Yd)| Y_d & |(Ys)| Y_\star \\
      };

      \draw[->] (Xd) to (Xs);
      \draw[->] (Yd) to (Ys);
      \draw[->] (Xd) to (Yd);
      \draw[->] (Xs) to (Ys);
    \end{tikzeq*}
    are pullback squares, i.e., the extended natural transformation over $D^\rhd$ is also \cartesian.
  \item \emph{van Kampen}, if it is both universal and effective.
  \end{parts}
\end{definition}

\begin{lemma}\label{van-Kampen-pseudo-limit}
  A colimit $Y_\star = \colim_{d \in D} Y_d$ in $\cat{E}$ is van Kampen \iff{} it is preserved by
  the pseudo-functor $\cat{E}^\op \to \Cat$ sending each $X \in \cat{E}$ to the slice category $\cat{E} \slice X$
  (with morphisms acting by pullbacks).
  In other words, the slice category $\cat{E} \slice Y_\star$ is the pseudo-limit $\lim_{d \in D} (\cat{E} \slice Y_d)$.
\end{lemma}

\begin{proof}
Pullback along the structure morphisms of $Y_\star$ induces a functor $P \from \cat{E} \slice Y_\star \to \lim_d (\cat{E} \slice Y_d)$. We need to show that this functor is an equivalence if and only if the colimit of $Y_\bullet$ is a van Kampen colimit.

An object of $\lim_d (\cat{E} \slice Y_d)$ can be identified with a \cartesian transformation $X \to Y$.
If colimits of diagrams \cartesian over $Y_\bullet$ exist, then taking the colimit yields a left adjoint to the functor above:
\[
\colim \from \lim_d (\cat{E} \slice Y_d)  \leftrightarrows \cat{E} \slice Y_\star \from P
\text{.}\]
Conversely, we claim that if $P$ has a left adjoint, then the left adjoint computes the colimits of diagrams that are \cartesian over $Y_\bullet$.
Indeed, assume that the pullback functor $P \from \cat{E} \slice Y_\star \to \lim_d (\cat{E} \slice Y_d)$ has a left adjoint $X_\bullet \mapsto X_\star$, and let $Z$ be an arbitrary object of $\cat{E}$.
A map $X_\star \to Z$ in $\cat{E}$ is the same as a map $X_\star \to Z \times Y_\star$ in $\cat{E} \slice Y_\star$, which by the adjunction formula is the same as a natural transformation $X_d \to Z \times Y_d$ over $Y_\bullet$, but this is exactly the same as a natural transformation $X_d \to Z$ in $\cat{E}$, and hence this shows that $X_\star$ is the colimit of $X_d$.

Now, $Y_\star$ is universal \iff{} the counit of this adjunction is an isomorphism and it is effective \iff{} the unit is an isomorphism.
Hence, the colimit $Y_\star$ of $Y_\bullet$ is van Kampen if and only if the pullback functor described above has a left adjoint such that the unit and counit of the adjunction are isomorphisms, \ie, if and only if it is an equivalence.
\end{proof}

For example, an initial object $0$ is always vacuously effective and it is universal \iff{} it is \emph{strict}, i.e., if there is a morphism $X \to 0$,
then $X$ is initial itself. Instead, a coproduct $Y_\star = \bigcoprod_d Y_d$ is van Kampen \iff{} it is universal and \emph{disjoint},
\ie, $Y_d \pull_{Y_\star} Y_{d'}$ is initial for $d \neq d'$. This can be seen inspecting the proof of~\cite{CLW}*{Proposition~2.14}.

\begin{lemma}\label{lem:levelwise_VK_colimit} Let  $D$ be a small category.
 Let $Y_\bullet\from C \to \cat{E}^D$ be a diagram such that $Y_\bullet(d)$ admits a van Kampen colimit in $\cat{E}$ for all $d \in D$.
 Then $Y_\bullet$ has a van Kampen colimit in $\cat{E}^D$.
\end{lemma}

\begin{proof}
  If each $d \in D$, $\colim_{c \in C} Y_c(d)$ exists in $\cat{E}$, then it is functorial in $d$ and it is a colimit in~$\cat{E}^D$.
  In particular, an object over $\colim_{c} Y_c$ is a $D$-indexed diagram $X(d) \to \colim_C Y_c(d)$,
  which as these colimits are all van Kampen is the same as a $(C \times D)$-indexed diagram $X_c(d) \to Y_c(d)$
  which is \cartesian in the $C$-direction, which in turn is the same as a $C$-indexed diagram $X_\bullet \in \cat{E}^D$
  which is \cartesian over $Y_\bullet$, hence proving the lemma.
\end{proof}

\medskip

We now recall the definition of various kinds of lextensive categories~\cite{CLW}.

\begin{definition} \label{thm:lextensive} Let $\cat{E}$ be a category with finite limits. For a regular cardinal $\alpha$, we say that~$\cat{E}$ is \emph{$\alpha$-lextensive}
 if $\alpha$-coproducts exist and are van Kampen colimits. Furthermore, we say that~$\cat{E}$ is
  \begin{parts}
 \item  \emph{lextensive} if it is $\omega$-lextensive, i.e., finite coproducts exist and are van Kampen colimits,
 \item  \emph{countably lextensive} if it is $\omega_1$-lextensive, i.e., countable coproducts exist and are van Kampen colimits,
 \item  \emph{completely lextensive} if it is $\alpha$-lextensive for all $\alpha$, i.e., all small coproducts exist and are van Kampen colimits.
 \end{parts}
\end{definition}

\begin{example} \label{thm:example-of-lextensive}
  There are numerous examples of  lextensive categories.
  \begin{parts}
  \item Any presheaf category is completely lextensive.
    In particular, for any group $G$ the category of $G$-sets is countably lextensive.
  \item More generally, any Grothendieck topos is completely lextensive.
    In fact, Giraud's theorem characterises Grothendieck toposes as the locally presentable categories in which
    coproducts and (in an appropriate sense) quotients by equivalence relations are van Kampen colimits.
  \item The category of topological spaces is completely lextensive.
    The same is true for many of its subcategories such as categories of Hausdorff spaces, compactly generated spaces,
    weakly Hausdorff compactly generated spaces, etc.
  \item The category of affine schemes is lextensive, the category of schemes is completely lextensive.
  \item The category of countable sets is countably lextensive.
  \item A category with finite limits $\cat{E}$ has the free coproduct completion which can be constructed as
    the category $\Fam \cat{E}$ of families of objects in $\cat{E}$.
    Explicitly, an object is pair $(S, (X_s)_{s \in S})$ where~$S$ is a set and~$(X_s)_{s \in S}$ is an $S$-indexed family of objects of $\cat{E}$.
    A morphism $(S, (X_s)) \to (S', (X'_{s'}))$ consists of a function $f \from S \to S'$ and morphisms $X_s \to X'_{f(s)}$ for all $s \in S$.
    $\Fam \cat{E}$ is completely lextensive.
    The $\alpha$-coproduct completion, $\Fam_\alpha \cat E$, obtained by restricting to $\alpha$-small families, is an $\alpha$-lextensive category.
  \end{parts}
\end{example}

For $S \in \Set$ and $X \in \cat{E}$, we write $S \tensorSetE X$ for the tensor of $X$ with $S$, when it
exists. If~$\cat{E}$ has countable coproducts, then this tensor exists for countable $S$ and can be
defined as
\begin{equation}
\label{equ:tensorSetE}
S \tensorSetE X = \bigcoprod_{s \in S} X \text{.}
\end{equation}
The global sections functor $\cat{E}(1, -) \from \cat{E} \to \Set$ has a partial left adjoint, defined by mapping a countable set $S$ to
\begin{equation}
\label{equ:fset}
\fset{S} \eqdef S \tensorSetE 1  = \bigcoprod_{s \in S} 1 \, .
\end{equation}
We extend this notation to diagram categories in a levelwise fashion: if $\cat{E}$ has countable coproducts and $D$ a small category, then the levelwise global sections functor $\cat{E}^D \to \Set^D$ has a partial left adjoint, sending a levelwise countable diagram $K \in \Set^D$ to $\fset{K} \in \cat{E}^D$, which is defined by levelwise application of $S \mapsto \fset{S}$. These functors will be used frequently in the paper. For example, we will use them in \cref{sec:cofibrations} to transfer the sets of boundary inclusions and horn inclusions in \eqref{equ:boundary-horns} from $\sSet$ to $\cats E$, so as to obtain generating sets for weak factorisation systems in
$\cats E$. We establish some of their basic properties in the next lemmas.

\begin{lemma} \label{product-with-discrete}
If $\cat{E}$ is countably lextensive, then for every countable set $S$ and $X \in \cat{E}$, we have $\fset{S} \times X \iso S \tensorSetE X$, naturally in $S$.
\end{lemma}

\begin{proof}
Since $\cat{E}$ is countably lextensive, it is countably distributive.
Thus, product with $X$ preserves countable coproducts, in particular tensors with countable sets.
This reduces the claim to the natural isomorphism $1 \times X \iso X$.
\end{proof}

The next lemma will be used, sometimes implicitly, in \cref{sec:cofibrations}.

\begin{lemma}\label{L-finite-limits}
If $\cat{E}$ is countably lextensive, then the functor $S \mapsto \fset{S}$ from countable sets
to $\cat E$ preserves finite limits.
\end{lemma}

\begin{proof}
  The functor $S \mapsto \fset{S}$ preserves terminal objects by definition.
  It also preserves pullbacks.
  Indeed, every pullback diagram of (countable) sets decomposes as a (countable) coproduct of
  product diagrams.
  These products are preserved since products preserve countable coproducts in each variable
  by lextensivity.
\end{proof}

The next lemma will be applied in \cref{sec:dependent-products}.

\begin{lemma}\label{lem:slice_presheaves_set}
  Let $\cat{E}$ be an $\alpha$-lextensive category.
  If $D$ is a small category and $S \from D \to \Set$ is a functor which takes values in $\alpha$-small sets, then there is an equivalence of categories
  \begin{equation*}
    \cat{E}^D \slice \fset{S} \,\simeq\, \cat{E}^{D \slice S}
  \end{equation*}
  where $D \slice S$ denotes the category of elements of $S$.
\end{lemma}

\begin{proof}
  The proof is similar to that of \cref{van-Kampen-pseudo-limit}.
  There is a functor $\cat{E}^{D \slice S} \to \cat{E}^D \slice \fset{S}$ which sends a functor $F \from D \slice S \to \cat{E}$ to the functor $V \from D \to \cat{E}$ defined by:
  \[
  V(d) = \bigcoprod_{s \in S(d)} F(d,s)
  \text{.}\]
  It comes with an obvious map to $\fset{S}$, which was defined as $\fset{S}(d) = \bigcoprod_{s \in S(d)} 1$.
  This functor has a right adjoint $\cat{E}^D \slice \fset{S} \to \cat{E}^{D \slice S}$ sending a functor $V \from D \to \cat{E}$ with a natural transformation $ V \to \fset{S}$ to the functor $F \from D \slice S \to \cat{E}$ where $F(d,s)$ is defined as the following pullback:
\begin{tikzeq*}{fibration-extension-square}
\matrix[diagram,column sep={6em,between origins}]
{
  |(0)| F(d,s)  & |(1)| V(d)                  \\
  |(2)| 1       & |(3)| \fset{S}(d)  \text{.} \\
};

\draw[->]  (0) to (1);
\draw[->]  (1) to (3);
\draw[->]  (0) to (2);
\draw[->]  (2) to (3);

\pb{0}{3};
\end{tikzeq*}
These two adjoints functor are equivalences.
Indeed, the counit of this adjunction is an isomorphism by universality of coproducts and the unit is an isomorphism by effectivity of coproducts.
\end{proof}

We now turn our attention to the class of complemented inclusions.
These will be useful for construction of certain colimits whose existence is not immediately obvious in lextensive categories and,
especially, in their diagram categories.
First of all, recall that
  a morphism $i \from A \to B$ in $\cat{E}$  is a \emph{complemented inclusion} if it has a \emph{complement}, i.e.,
  a morphism $j \from C \to B$ \st{} $i$ and $j$ exhibit $B$ as a coproduct of $A$ and $C$ in $\cat{E}$.
  In other words, $i$ is isomorphic to the coproduct inclusion $A \to A \coprod C$.
We will often say simply that $C$ is a complement of $A$.
The notation $A \cto B$ will be sometimes used to indicate \di{}s.
Note that \di{}s are sometimes (e.g., in our previous work \cites{GSS,H}) called \emph{decidable inclusions} in reference to the notion of decidability
in constructive logic.

\begin{lemma}\label{coproduct-decidable-colimit}  \label{lem:pushout_decidable} \hfill
  \begin{parts}
   \item If $\cat{E}$ is lextensive, then the pushout of a complemented inclusion along any morphism exists and is again a complemented inclusion.
    Moreover, such pushouts are preserved by functors (and pseudo-functors) that preserve finite coproducts and thus are van Kampen colimits.
   \item \label{extensive-sequential-colimit}
    If $\cat{E}$ is countably lextensive, then the colimit of a sequence of complemented inclusions exists and is again a complemented inclusion.
    Moreover, such colimits are preserved by functors (and pseudo-functors) that preserve countable coproducts and thus are van Kampen colimits.
 \end{parts}
\end{lemma}


\begin{proof}
  If $i \from A \to B$ is a complemented inclusion with complement $C$,
  then the pushout of~$i$ along~$A \to D$ is $C \coprod D$.
  Similarly, if $i_k \from A_k \to A_{k + 1}$ are complemented inclusions with complements~$C_{k + 1}$,
  then $\colim_k A_k$ is $\bigcoprod_k C_k$ (where $C_0 = A_0$).
  The claims on preservation by functors then follow immediately.

  These presentations of colimits as coproducts remain when we consider $\cat{E}$ as a bicategory.
  Recall from \Cref{van-Kampen-pseudo-limit} that a colimit is van Kampen exactly if it is preserved by a certain pseudo-functor.
  Since (finite or countable) coproducts are assumed van Kampen, so are the presented colimits.
\end{proof}

\begin{lemma}\label{decidable-closure}
Assume $\cat{E}$ is lextensive.
  \begin{parts}
  \item \label{decidable-union}
    complemented subobjects in $\cat{E}$ are closed under finite unions.
  \item \label{decidable-limit}
    complemented inclusions in $\cat{E}$ are closed under finite limits, i.e., if $X \to Y$ is
    a natural transformation between finite diagrams in $\cat{E}$ that is
    a levelwise complemented inclusion, then so is the induced morphism $\lim X \to \lim Y$.
  \end{parts}
\end{lemma}

\begin{proof}
  The proof of~\cite{GSS}*{Lemma~1.1.4} applies verbatim.
\end{proof}

\begin{lemma} \label{finite-limit-and-sequential-colimit}
Assume that $\cat{E}$ is countably lextensive.
Then the full subcategory of $[\omega, \cat{E}]$ consisting of sequences of \di{}s has finite limits which are preserved by the colimit functor
(sending each sequence to its colimit in $\cat{E}$).
\end{lemma}

\begin{proof}
First note that the category of sequences of \di{}s has finite limits by \cref{decidable-limit} of \cref{decidable-closure}.
Moreover, \cref{extensive-sequential-colimit} of \cref{coproduct-decidable-colimit} implies that colimits of such sequences exist.
It suffices to show that this colimit functor preserves terminal objects and pullbacks.
Terminal objects are preserved since $\omega$ is a connected category (it has an initial object).
For the case of pullbacks, we consider a span $A \to C \leftarrow B$ of sequences of complemented inclusions.
We need to show that the map
\[
\colim_{k \in \omega} A_k \times_{C_k} B_k \to \colim A \times_{\colim C} \colim B
\]
is invertible.
We decompose this map into three factors:
\begin{tikzeq*}
\matrix[diagram,column sep={15em,between origins}]
{
  |(S0)| \colim_{k \in \omega} A_k \times_{C_k} B_k & |(S3)| \colim A \times_{\colim C} \colim B  \rlap{\text{.}} \\
  |(S1)| \colim_{k \in \omega} A_k \times_{\colim C} B_k & |(S2)| \colim_{i, j \in \omega} A_i \times_{\colim C} B_j \\
};

\draw[->] (S0) to (S3);
\draw[->] (S0) to (S1);
\draw[->] (S1) to (S2);
\draw[->] (S2) to (S3);
\end{tikzeq*}
The left map is invertible even before taking colimits because $C_k \to \colim C$ is a monomorphism.
The bottom map is invertible because the diagonal functor $\omega \to \omega \times \omega$ is final (it has a left adjoint).
The right map is invertible by universality of the van Kampen colimits $\colim A$ and $\colim B$ (\cref{extensive-sequential-colimit} of \cref{coproduct-decidable-colimit}).
\end{proof}





Let $D$ be a small category. We say that a morphism $\phi \from F \rightarrow G$ in $\cat{E}^D$,
is a {\em levelwise complemented inclusion} if its components $\phi_d \colon F_d \to G_d$,
for $d \in D$, are complemented inclusions in $\cat{E}$.
Note that this is considerably less restrictive than asking for $\phi$ to be a complemented inclusion in
$\cat{E}^D$.

\begin{corollary}\label{colimits-diagram-category}
Let $D$ be a small category.
\begin{parts}
 \item \label{colimits-diagram-category:pushout}
 If $\cat{E}$ is lextensive, then pushouts along levelwise complemented inclusions exist, are computed levelwise and are van Kampen colimits in $\cat{E}^D$.
 \item \label{colimits-diagram-category:sequential}
 If $\cat{E}$ is countably lextensive, then colimits of sequences of levelwise complemented inclusions exist, are computed levelwise and are van Kampen colimits in $\cat{E}^D$.
 \end{parts}
\end{corollary}

\begin{proof}
  This follows immediately from \cref{lem:levelwise_VK_colimit,lem:pushout_decidable}.
\end{proof}

\begin{lemma}\label{ldi-pushout-product} Let $D$ be a small category.
  If $\cat{E}$ is lextensive, then the pushout products of \ldi{}s in $\cat{E}^D$
  with arbitrary morphisms exist.
  Moreover, the pushouts involved are van Kampen.
\end{lemma}

\begin{proof}
  By universality of coproducts, \ldi{}s are closed under pullbacks.
  Thus a pushout computing a pushout product with a \ldi{} is a pushout along a \ldi{}.
  They are van Kampen by \cref{colimits-diagram-category}.
\end{proof}

The following statement will be needed in \cref{sec:cofibrations} to prove \cref{thm:reedy-colims}.

\begin{lemma} \label{thm:subobject-colimit}
Let $\cat C$ be a category, $P$ a poset with binary meets, $X \in \cat C$ an object and
\begin{equation*}
  A  = (A_p \ito X \ | \ p \in P)
\end{equation*}
a diagram of subobjects of $X$ closed under intersection, \ie, such that $A_p \cap A_q = A_{p \cap q}$.
Then if $A$ has a van Kampen colimit, the colimit is also a subobject of $X$.
\end{lemma}

\begin{proof}
We assume that $\colim_{p \in P} A_p$ exists and is a van Kampen colimit, and we show that the diagonal map $\colim_{p \in P} A_p  \to F = \left( \colim_{p \in P} A_p \right) \times_X \left( \colim_{p \in P} A_p \right)$ is an isomorphism.
First, we form pullbacks:
\begin{tikzeq*}
	\matrix(D)[diagram,column sep={between origins,6em}]{
          A_p \cap A_q & F_q & A_q \\
          F_p & F & \colim_q A_q \\
          A_p & \colim_p A_p & X \\
        };

	\draw[->] (D-1-2) to (D-1-3);
	\draw[->] (D-1-2) to (D-2-2);
	\draw[->] (D-2-1) to (D-2-2);
	\draw[->] (D-2-2) to (D-2-3);
	\draw[->] (D-1-3) to (D-2-3);
	\draw[->] (D-2-1) to (D-3-1);
	\draw[->] (D-3-1) to (D-3-2);
	\draw[->] (D-2-2) to (D-3-2);
	\draw[->] (D-1-1) to (D-2-1);
	\draw[->] (D-1-1) to (D-1-2);
	\draw[->] (D-2-3) to (D-3-3);
	\draw[->] (D-3-2) to (D-3-3);

	\pb{D-2-2}{D-3-3};
	\pb{D-2-1}{D-3-2};
        \pb{D-1-1}{D-2-2};
        \pb{D-1-2}{D-2-3};
\end{tikzeq*}
Using that the colimits are van Kampen, we have that $F = \colim_p F_p$ and $F_p = \colim_q A_q \cap A_p$ and hence $F = \colim_{p,q} A_p \cap A_q$
with the two maps  $F \to \colim_p A_p$ being induced by the maps $A_p \cap A_q \to A_p$ and $A_p \cap A_q \to A_q$.
We conclude by observing that $\colim_p \left( A_p \cap A_q \right) = A_q$.
Indeed the map $P \to (\downarrow q)$ that send $p \in P$ to $p\cap q$ is right adjoint to the inclusion of $(\downarrow q)$ to $P$, so it is a final functor.
It hence follows that
\begin{equation*}
\colim_{p \in P} A_{p \cap q} = \colim_{p \leqslant q} A_p = A_q
\end{equation*}
So this implies that $F = \colim_q A_q$, with the projection map $F \to \colim_q A_q$ being the identity, hence proving that $\colim_q A_q \to X$ is a monomorphism.
\end{proof}

We prove a statement relating van Kampen colimits and the pullback evaluation $\hat{\ev}$ functor, defined in~\eqref{equ:pullback-evaluation}.
This statement will be needed in \cref{sec:equivalence-extension}.

%


\begin{lemma} \label{van-Kampen-pullback-weighted-limit}
Let $D$ be a small category.
Let $Y \from C \to [D^\op, \cat{E}]$ be a diagram with levelwise van Kampen colimit $\colim Y$.
Let $p \from X \to Y$ be a \cartesian transformation, which we regard as a $C$-indexed diagram of arrows in $[D^\op, \cat{E}]$.

Let $q \from A \to B$ be a map in $[D^\op, \Set]$ with $B$ representable such that $[D^\op, \cat{E}]$ supports evaluation at $A$.
Then $\hat{\ev}_q$ (valued in arrows of $\cat{E}$) preserves the colimit of $p$, the resulting colimit is computed separately on source and target, and all maps of the colimit cocone are pullback squares.
\end{lemma}

\begin{proof}
First note that by levelwise effectivity of $\colim Y$, we obtain $\colim X$ (and hence $\colim p$).
The square $p_c \to \colim p$ is a pullback for all $c \in C$.

Consider the functor $F$ sending an arrow $M \to N$ in $[D^\op, \cat{E}]$ to the sequence of arrows
\begin{tikzeq*}{fibration-extension-square}
\matrix[diagram,column sep={12em,between origins}]
{
  |(0)| M(B) & |(1)| M(A) \times_{N(A)} N(B) & |(2)| N(B) \rlap{\text{.}} \\
};

\draw[->]  (0) to (1);
\draw[->]  (1) to (2);
\end{tikzeq*}
The first arrow is the pullback evaluation at $q$ of $M \to N$.
Evaluation preserves limits, in particular pullbacks.
By pullback pasting, the action of $F$ on a map of arrows that is a pullback is a pasting of pullback squares.

Let us inspect the action of $F$ on the colimit cocone of $p$.
It will suffice to show that it results in objectwise colimit cocones.
Since the maps of the colimit cocone of $p$ are pullback squares, we obtain pastings of pullback squares upon applying $F$.
Recall that $\ev_B$ is computed by evaluation at the object representing $B$.
So by assumption, $(\colim Y)(B) = \ev_B(\colim Y)$ is colimit of $\ev_B \circ Y$ and van Kampen.
The claim follows by universality of this van Kampen colimit.
\end{proof}

  \section{An enriched small object argument}
\label{sec:enrwfs}


The goal of this section is to develop a version of the small object argument that allows us
to construct weak factorisation systems on the category of simplicial objects~$\cats E$,
where~$\cat E$ is  a countably lextensive category. In view of our application to
both simplicial objects in \cref{sec:cofibrations} and semisimplicial objects in \cref{sec:Semi-simplicial},
we develop our small object argument for diagram categories~$\cat{E}^D$ in general.
Importantly, our weak
factorisation systems are {\em enriched}, in the sense of~\cite{RiehlE:catht}.
We will be constructing $\psh \cat{E}$-enriched weak factorisation systems
on $\cat{E}^D$, where where $\psh \cat{E}$ denotes the category of presheaves over $\cat{E}$.
This is because the category of diagrams $\cat{E}^D$ is not necessarily $\cat{E}$-enriched,
but it is $\psh \cat{E}$-enriched, as we now recall.


For~$E \in \cat E$ and~$X \in \cat{E}^D$, we define~$E \tensorEsE X \in  \cat{E}^D$ by letting
\begin{equation}
  \label{equ:psh-tensor}
  (E \tensorEsE X)_d \eqdef E \times  X_d  \, .
\end{equation}
Given $X, Y \in \cat{E}^D$, we then define the hom-object $\Hom_{\psh \cat{E}}(X, Y) \in \psh \cat{E}$
by letting:
\[
  \begin{array}{rccl}
    \Hom_{\psh \cat{E}}(X, Y)  \from & \cat{E}^{\op} & \to     & \Set                         \\
                                     & E             & \mapsto & \Hom_\Set(E \tensorEsE X, Y)
  \end{array}
\]
This makes
$\cat{E}^D$ into a $\psh \cat{E}$-enriched category, so that the formula in~\eqref{equ:psh-tensor}
provides the tensor of $E \in \psh{\cat E}$ and $X \in \cat{E}^D$ with respect to this enrichment.
When the presheaf is representable, the representing object is denoted by $\Hom_{\cat{E}}(X, Y)$.

Using the enrichment, we can define an internal version of the familiar lifting problems involved
in the definition of a weak factorisation systems. For morphisms $i \from A \to B$ and $p \from X \to Y$ in $\cat{E}^D$, we define  the \emph{presheaf of lifting problems} of $i$ against $p$ by letting
\[
  \Prob_{\psh \cat{E}}(i, p) \eqdef  \Hom_{\psh \cat{E}}(A, X) \pull_{\Hom_{\psh \cat{E}}(A, Y)} \Hom_{\psh \cat{E}}(B, Y) \, .
\]
When the relevant hom-objects are representable, then so is $ \Prob_{\psh \cat{E}}(i, p)$. In this case,
we write $\Prob_{\cat{E}}(i, p)$ for its representing object and call it  the \emph{object of lifting problems} of $i$ against $p$. Note that the induced pullback hom of $i$ and $p$ (\cf \cref{thm:leibniz-terminology}) has the form
\begin{equation}
  \label{equ:pullback-hom}
  \hat{\Hom}_{\psh \cat{E}}(i, p) \from \Hom_{\psh \cat{E}}(B, X) \to \Prob_{\psh \cat{E}}(i, p)
\end{equation}
Again, if the objects are representables, we have also an induced pullback hom in $\cat{E}$,
which has the form
\begin{equation}
  \label{equ:pullback-hom-in-E}
  \hat{\Hom}_{\cat{E}}(i, p) \from \Hom_{\cat{E}}(B, X) \to \Prob_{\cat{E}}(i, p)  \, .
\end{equation}
We are ready to define the $\psh{ \cat E}$-enriched counterparts of the standard lifting
properties.

\begin{definition}\label{elp}
  Let $i \from A \to B$ and $p \from X \to Y$ be morphisms of $\cat{E}^D$.
  \begin{itemize}
    \item We say that $i$ has the \emph{$\psh \cat{E}$-\ellp{}} \wrt{} $p$ and that $p$ has the \emph{$\psh \cat{E}$-\erlp{}} \wrt{} $i$
          if the induced pullback hom in~\eqref{equ:pullback-hom} is a split epimorphism in $\psh \cat{E}$.
    \item We say that $i$ has the \emph{$\cat{E}$-\ellp{}} \wrt{} $p$ and that $p$ has the \emph{$\cat{E}$-\erlp{}} \wrt{} $i$
          if the induced pullback hom in~\eqref{equ:pullback-hom-in-E}
          exists and is a split epimorphism in $\cat E$.
  \end{itemize}
\end{definition}

Since the Yoneda embedding is fully faithful and preserves pullbacks, as soon as
all relevant $\cat{E}$-valued hom-objects exist, the \emph{$\psh \cat{E}$-\ellp{}} and
\emph{$\cat{E}$-\ellp{}} are equivalent and $\Prob_{\psh \cat{E}}(i, p)$ is
represented by $\Prob_{\cat{E}}(i, p)$.

In both $\psh \cat{E}$ and $\cat{E}$, the class of split epimorphisms is the right class of a weak factorization system, with left class given by complemented inclusions.
As such, it enjoys a number of standard closure properties.
Our notions of enriched lifting property are defined from this class via the pullback hom.
Because of this, the classes of maps defined below by an $\psh \cat{E}$-enriched lifting property will inherit corresponding closure properties.
For example, split epimorphisms are closed under retracts.
Thus, classes of maps defined by an $\psh \cat{E}$-enriched lifting property are closed under retracts.

As is usual, we extend the terminology of enriched lifting properties from maps to classes of maps on either side by universal quantification.

\begin{definition}\label{I-cofibration}
  Let $I = \{ i \from A_i \to B_i \}$ be a set of morphisms of $\cat{E}^D$.
  \begin{itemize}
    \item An \emph{(enriched) $I$-fibration} is a morphism with the \erlp{} \wrt{} $I$.
    \item An \emph{(enriched) $I$-cofibration} is a morphism with the \ellp{} \wrt{} $I$-fibrations.
  \end{itemize}
\end{definition}

\medskip

When the left map of a $\psh{\cat E}$-enriched lifting problem comes from $\Set^D$ via
levelwise application of the operation in~\eqref{equ:fset},
we may simplify the lifting problem (assuming some technical conditions hold). Indeed, the pullback hom~\eqref{equ:pullback-hom-in-E} reduces to a pullback evaluation.
We record this in the next couple of statements, which are phrased using $D^{\op}$ instead of $D$ in order to exploit the language of representable functors.
We make use of the evaluation functor $\ev_K \from [D^\op, \cat{E}] \to \cat{E}$ defined for finite colimits $K$ of representables by letting:
\[
  \ev_K(X) = \coend_{d \in D^{\op}} X_d^{K_d}
  .\]
This generalises the evaluation functor defined in \cref{equ:evaluation}, which is the case~$D = \Delta$.
As in \cref{evaluation}, we may equivalently view $\ev_K(X)$ as the $K$-weighted limit of $X$,
which implies that $\ev$ is a (partial) two-variable functor.

\begin{lemma} \label{nicolas-grouping}
  Let $K \in [D^\op, \Set]$ be levelwise countable.
  \begin{parts}
    \item \label{simplifying-E-tensor} There is an isomorphism $(E \tensorEsE \fset{K})_d \iso K_d \tensorSetE E$ natural in $K$, $E \in \cat{E}$, and $d \in D$.
    \item \label{exponential-via-evaluation} Assume that $K$ is a finite colimit of representables. Then the hom-presheaf $\Hom_{\psh{\cat{E}}}(\fset{K}, X)$ is representable for $X \in [D^\op, \cat{E}]$ and we have an isomorphism $\Hom_{\cat{E}}(\fset{K}, X) \iso \ev_K(X)$, natural in $K$ and $X \in [D^\op, \cat{E}]$.
  \end{parts}
\end{lemma}

\begin{proof}
  \Cref{simplifying-E-tensor} follows from \cref{product-with-discrete}.
  For \cref{exponential-via-evaluation}, \cref{simplifying-E-tensor}
  implies that $\Hom_{\psh \cat{E}}(\fset{K}, X)$ is naturally isomorphic to the $\cat{E}$-presheaf $E \mapsto
    \Hom_\Set(d \mapsto K_d \tensorSetE E, X)$.
  A representing object for it is by definition the $K$-weighted limit of $X$, \ie, $\ev_K(X)$.
  This exists in our setting for $K$ a finite colimit of representables.
\end{proof}

\begin{proposition} \label{enriched-lifting-as-pullback-evaluation}
  Let $i \from A \to B$ be a map in $[D^\op, \Set]$ between objects that are levelwise countable and finite colimits of representables and let $p \from X \to Y$ be a map in $[D^\op, \cat{E}]$. Then the following are
  equivalent:
  \begin{parts}
    \item $\fset{i} \from \fset{A} \to \fset{B}$ has the $\cat{E}$-enriched left lifting property with respect to $p$,
    \item the pullback evaluation $\hat{\ev}_i(p)$ is a split epimorphism in $\cat{E}$.
  \end{parts}
\end{proposition}

\begin{proof}
  This is an immediate consequence of \cref{exponential-via-evaluation} of \cref{nicolas-grouping}.
\end{proof}

\cref{enriched-lifting-as-pullback-evaluation} will be used in \cref{sec:cofibrations} to relate (trivial) Kan fibrations in $\cats{E}$ in the sense of \cref{def:fibration_algebraic} with fibrations in the sense of \cref{I-cofibration} with respect to the images in $\cats{E}$ of horn inclusions (boundary inclusions, respectively) under the operation $\fset{(\uvar)} \from \Set \to \cat{E}$.

\medskip

We now turn our attention to $\psh \cat{E}$-enriched weak factorisation systems.

\begin{definition} \label{enriched-wfs}
  A \emph{$\psh \cat{E}$-\ewfs{}} on $\cat{E}^D$ is
  a pair $(\fml{L}, \fml{R})$ of classes of morphisms of $\cat{E}^D$ \st{}:
  \begin{itemize}
    \item a morphism belongs to $\fml{L}$ \iff{} it has the $\psh \cat{E}$-\ellp{} \wrt{} $\fml{R}$;
    \item a morphism belongs to $\fml{R}$ \iff{} it has the $\psh \cat{E}$-\erlp{} \wrt{} $\fml{L}$;
    \item every morphism of $\cat{E}^D$ factors as an $\fml{L}$-morphism followed by an $\fml{R}$-morphism.
  \end{itemize}
\end{definition}

The classes $\fml{L}$ and $\fml{R}$ in the above definition are closed under retract as they are characterized by $\psh \cat{E}$-enriched lifting properties.

We will abbreviate ``$\psh \cat{E}$-\elp{}'' to ``\elp{}'',
but we will be explicit about cases where it coincides with the $\cat{E}$-\elp{}.

\begin{lemma}\label{ordinary-wfs}
  Let $(\fml{L}, \fml{R})$ be an \ewfs{}.
  \begin{parts}
    \item A morphism is in $\fml{L}$ \iff{} it has the ordinary \llp{} \wrt{} $\fml{R}$.
    \item A morphism is in $\fml{R}$ \iff{} it has the ordinary \rlp{} \wrt{} $\fml{L}$.
  \end{parts}
  In particular, $(\fml{L}, \fml{R})$ is also an ordinary \wfs{}.
\end{lemma}

\begin{proof}
  For~(i), a morphism of $\fml{L}$ has the ordinary \llp{} \wrt{} $\fml{R}$ by
  evaluating the hom-presheaves at $1 \in \cat{E}$.
  Conversely, a morphism with the ordinary lifting property admits a lift against
  the second factor of its $(\fml{L}, \fml{R})$-factorisation, thus making it into a
  retract of the first factor (\cf also the proof of \cref{cofibration-as-retract-of-cell-complex}).
  The conclusion follows since $\fml{L}$ is closed under retracts.
  Part~(ii) follows by duality.
\end{proof}

We will fix a set $I$ and study a version of the \smo{} that produces an \ewfs{} of
$I$-cofibrations and $I$-fibrations under suitable assumptions.

\begin{definition}
  Let $i \from A \to B$ and $p \from X \to Y$ be morphisms of $\cat{E}^D$. Assume
  that we have  a factorisation
  \begin{tikzeq*}
    \matrix[diagram]
    {
      |(X)| X &           & |(Y)| Y \text{.} \\
              & |(Xp)| X' &                                 \\
    };

    \draw[->] (X)  to node[above]       {$p$}  (Y);
    \draw[->] (Xp) to node[below right] {$p'$} (Y);

    \draw[->] (X) to (Xp);
  \end{tikzeq*}
  We say that $p$ satisfies the \emph{$X'$-partial \erlp{}} \wrt{} $i$ if
  there is a lift in the diagram
  \begin{tikzeq*}
    \matrix[diagram,column sep={between origins,10em}]
    {
                                    & |(BcX)| \Hom_{\psh \cat{E}}(B, X')                         \\
      |(P)| \Prob_{\psh \cat{E}}(i, p) & |(Pc)|  \Prob_{\psh \cat{E}}(i, p') \text{.} \\
    };

    \draw[->,dashed] (P) to (BcX);

    \draw[->] (P)   to (Pc);
    \draw[->] (BcX) to (Pc);
  \end{tikzeq*}
\end{definition}

Such partial lifting properties are a crucial ingredient of the \smo{}, but they are only tractable when
$i$ is a \ldi{}. This is thanks to the next two lemmas, where we use the tensor defined in~\eqref{equ:psh-tensor}.

\begin{lemma}\label{ldi-saturation}
  Levelwise complemented inclusions in $\cat{E}^D$ are closed under:
  \begin{parts}
    \item $E \tensorEsE \uvar$ for all $E \in \cat{E}$;
    \item countable coproducts;
    \item pushouts along arbitrary morphisms;
    \item sequential colimits;
    \item retracts.
  \end{parts}
  Moreover, the colimits of parts (ii), (iii) and (iv) are preserved by $E \tensorEsE \uvar$ for all $E \in \cat{E}$.
\end{lemma}

\begin{proof} The functor $E \tensorEsE \uvar$ and all the colimits mentioned are computed levelwise in $\cat{E}$, so the results boil down to the fact that complemented inclusions in $\cat{E}$ are stable under all these constructions. Stability under $E \tensorEsE \uvar $ follows from distributivity of product over coproduct in complemented categories: if $A \to A \coprod B$ is a complemented inclusion, then its image under $E \tensorEsE \uvar$ is~$E \tensorEsE A \to (E \tensorEsE A) \coprod (E \tensorEsE B)$ and is a complemented inclusion. The case of a countable coproduct is also clear: if $A_k \to A_k \coprod B_k$ is a family of complemented inclusions, then their coproduct can be written as $\bigcoprod A_k \to \left(\bigcoprod A_k \right) \coprod \left(\bigcoprod B_k \right)$. Stability under pushout and sequential composition follows from \cref{lem:pushout_decidable}. The fact that they are preserved by $E \tensorEsE \uvar$ follows from \cref{lem:pushout_decidable}.  The case of retracts can be deduced from the stability under limits proved in \cref{decidable-closure} as retracts can be seen as limits.
\end{proof}

\begin{lemma} \label{saturation}
  Let $p \from X \to Y$ be a map in $\cat{E}^D$ and $\fml{L}$ a class of levelwise complemented inclusions in $\cat{E}^D$ that have the \ellp{} \wrt{} $p$.
  Then $\fml{L}$ is closed under the following operations:
  \begin{parts}
    \item tensors by objects of $\cat{E}$,
    \item countable coproducts,
    \item pushouts,
    \item colimits of sequences,
    \item retracts.
  \end{parts}
\end{lemma}

\begin{proof}
  For $X \in \cat{E}^D$, the functor $\Hom_{\psh \cat{E}}(\uvar, X)$ is not necessarily an adjoint.
  However, since split epimorphisms are closed under limits dual to the colimits listed above, it is sufficient to verify that
  it carries these colimits to limits.
  (In the case of tensors this means that
  $\Hom_{\psh \cat{E}}(F \tensorEsE A, X) \iso \Hom_{\psh \cat{E}}(A, X)^{\cat{E}(\uvar, F)}$
  for all $F \in \cat{E}$.)
  This follows directly from these colimits being preserved by the tensors as recorded in \cref{ldi-saturation}.
\end{proof}


\begin{definition} \label{E-finite} Let $A \in \cat{E}^D$. We say that $A$ is \emph{finite} if
  the following hold:
  \begin{parts}
    \item \label{E-finite:hom-exists} $\Hom_{\cat{E}}(A, X)$ exists for every $X \in \cat{E}^D$;
    \item \label{E-finite:sequential-colim} $\Hom_{\cat{E}}(A, \uvar)$ preserves colimits of sequences of \ldi{}s;
    \item \label{E-finite:ldi-to-di} $\Hom_{\cat{E}}(A, \uvar)$ sends \ldi{}s to \di{}s.
  \end{parts}
\end{definition}

The next lemma provides a supply of finite objects. For its statement, recall the functor $S \mapsto \fset{S}$ from \cref{extensive-categories}. As \cref{nicolas-grouping},
it is formulated using $D^{\op}$ instead of $D$ for convenience.
%

\begin{lemma}\label{finite-colimit-finite}
  Let $D$ be a locally countable category and assume that presheaf $A \in \psh D$ is
  a finite colimit of representables.
  Then $\fset{A} \in [D^\op, \cat{E}]$ is finite.
\end{lemma}

\begin{proof}
  First, note that since $D$ is locally countable, $A$ is levelwise countable and thus $\fset{A}$ exists.
  By \cref{exponential-via-evaluation} of \cref{nicolas-grouping}, $\Hom_{\cat{E}}(\fset{A}, \uvar)$ exists and is given by $\ev_A$ (evaluation at $A$).
  Call $X \in \psh D$ \emph{$\cat{E}$-finite} if it satisfies the conditions of \cref{E-finite} with $\Hom_{\cat{E}}(X, \uvar)$ replaced by $\ev_X$.
  Our goal then is to show that $A$ is $\cat{E}$-finite.
  This follows from the following observations:
  \begin{itemize}
    \item
          Representables are $\cat{E}$-finite.
          For this, recall that evaluation at a representable is given by evaluation at the representing object.
          \Cref{E-finite:sequential-colim} uses \cref{colimits-diagram-category:sequential} of \cref{colimits-diagram-category} to see that the colimit is computed levelwise.
    \item
          $\cat{E}$-finite presheaves are closed under finite colimits.
          For this, we use that the partial two-variable functor $\ev$ sends colimits in its first argument to limits.
          \Cref{E-finite:hom-exists} holds since $\cat{E}$ has finite limits.
          \Cref{E-finite:sequential-colim} holds since finite limits preserve colimits of sequences of complemented inclusions in $\cat{E}$ (\cref{finite-limit-and-sequential-colimit}).
          \Cref{E-finite:ldi-to-di} holds since complemented inclusions in $\cat{E}$ are closed under finite limits (\cref{decidable-limit} of \cref{decidable-closure}).
          \qedhere
  \end{itemize}
\end{proof}

The hypothesis of finiteness is used in the next result, where we use the notion of an $I$-fibration in the sense of \cref{I-cofibration}.

\begin{lemma}\label{partial-lifts-colimit}
  Assume that the domains and codomains of morphisms of $I$ are finite.
  Let $Y \in \cat{E}^D$ and $(X_k \to X_{k + 1} \ | \ k \in \mathbb{N})$ be
  a sequence of morphisms in $\cat{E}^D \slice Y$.
  If every $X_k \to X_{k + 1}$ is
  a \ldi{} and each $p_k \from X_k \to Y$ has $X_{k + 1}$-partial \erlp{} \wrt{} $I$,
  then $\colim_k X_k \to Y$ is an $I$-fibration.
\end{lemma}

\begin{proof}
  Fix a morphism $i \from A \to B$ of $I$.
  Since $A$ and $B$ are finite, the given partial enriched lifting properties are $\cat{E}$-enriched.
  Moreover, since $X_k \to X_{k + 1}$ is a \ldi{}, \cref{decidable-closure} implies that
  $\Prob_{\cat{E}}(i, p_k) \to \Prob_{\cat{E}}(i, p_{k + 1})$ is a \di{}.

  Proceeding by induction \wrt{} $k$, we can pick lifts
  \begin{tikzeq*}
    \matrix[diagram,column sep={between origins,10em}]
    {
                                    & |(BcX)| \Hom_{\cat{E}}(B, X_{k + 1}) \\
      |(P)| \Prob_{\cat{E}}(i, p_k) & |(Pc)|  \Prob_{\cat{E}}(i, p_{k + 1}) \\
    };

    \draw[->,dashed] (P) to (BcX);

    \draw[->] (P)   to (Pc);
    \draw[->] (BcX) to (Pc);
  \end{tikzeq*}
  that are natural in $k$.
  Indeed, since $\Prob_{\cat{E}}(i, p_{k - 1}) \to \Prob_{\cat{E}}(i, p_k)$ is a \di{},
  we can construct a compatible lift by assembling a previously constructed lift on
  $\Prob_{\cat{E}}(i, p_{k - 1})$ with a given lift on its complement.
  Since $A$ and $B$ are finite, we have
  \begin{align*}
    \colim_k \Hom_{\cat{E}}(B, X_k) & = \Hom_{\cat{E}}(B, \colim_k X_k)                                                                  \\
    \intertext{and}
    \colim_k \Prob_{\cat{E}}(i, p_k)
                                    & = \colim_k \left( \Hom_{\cat{E}}(A, X_k) \pull_{\Hom_{\cat{E}}(A, Y)} \Hom_{\cat{E}}(B, Y) \right) \\
                                    & = \left( \colim_k \Hom_{\cat{E}}(A, X_k) \right) \pull_{\Hom_{\cat{E}}(A, Y)} \Hom_{\cat{E}}(B, Y) \\
                                    & = \Hom_{\cat{E}}(A, \colim_k X_k) \pull_{\Hom_{\cat{E}}(A, Y)} \Hom_{\cat{E}}(B, Y)                \\
                                    & = \Prob_{\cat{E}}(i, \colim_k p_k) \text{,}
  \end{align*}
  the latter by universality of sequential colimits of \di{}s in $\cat{E}$ (\cref{lem:pushout_decidable}).
  Thus we obtain a diagram
  \begin{tikzeq*}
    \matrix[diagram,column sep={between origins,14em}]
    {
                                     & |(BcX)| \Hom_{\cat{E}}(B_i, \colim_k X_k) \\
      |(P)|  \Prob_{\cat{E}}(i, \colim_k p_k) &
      |(Pc)| \Prob_{\cat{E}}(i, \colim_k p_k) \text{.}           \\
    };

    \draw[->,dashed] (P) to (BcX);

    \draw[->] (P)   to (Pc);
    \draw[->] (BcX) to (Pc);
  \end{tikzeq*}
  where the bottom map is an identity, \ie,
  these lifts form a section that exhibits~$\colim_k X_k \to~Y$ as an $I$-fibration.
\end{proof}

The following lemma isolates a simpler version of the inductive step in the construction of lifts in \cref{partial-lifts-colimit}.
It is needed in \cref{sec:equivalence-extension}.

\begin{lemma} \label{structured-fibration-extension}
  Let
  \begin{tikzeq*}{structured-fibration-extension:square}
    \matrix[diagram]
    {
      |(X)| X & |(Y)| Y                         \\
      |(A)| A & |(B)| B \rlap{\text{.}} \\
    };

    \draw[->] (X) to (Y);
    \draw[->] (A) to (B);
    \draw[->] (X) to node[left] {$p$} (A);
    \draw[->] (Y) to node[right] {$q$} (B);
    \pb{X}{B};
  \end{tikzeq*}
  be a pullback square in $\cat{E}^D$ with $A \to B$ a levelwise complemented inclusion.
  Let $i \from U \to V$ be a map in~$\cat{E}^D$ between finite objects such that $\hat{\Hom}_{\cat{E}}(i, p)$ and $\hat{\Hom}_{\cat{E}}(i, q)$ have sections.
  Then, for any section $s$ of $\hat{\Hom}_{\cat{E}}(i, p)$, there is a section $t$ of $\hat{\Hom}_{\cat{E}}(i, q)$ such that the diagram
  \begin{tikzeq*}
    \matrix[diagram,column sep={10em,between origins}]
    {
      |(ps)| \Hom_{\cat{E}}(V, X) & |(qs)| \Hom_{\cat{E}}(V, Y)                         \\
      |(pp)| \Prob_{\cat{E}}(i, p) & |(qp)| \Prob_{\cat{E}}(i, q) \rlap{\text{.}} \\
    };

    \draw[->] (ps) to (qs);
    \draw[->] (pp) to (qp);
    \draw[->] (ps) to node[left] {$\hat{\Hom}_{\cat{E}}(i, p)$} (pp);
    \draw[->] (qs) to node[right] {$\hat{\Hom}_{\cat{E}}(i, p)$} (qp);
    \draw[->,bend right,dashed] (pp) to (ps);
    \draw[->,bend left,dashed] (qp) to (qs);
    \pb{ps}{qp};
  \end{tikzeq*}
  forms a morphism of retracts.
\end{lemma}


\begin{proof}
  The map $\Prob_{\cat{E}}(i, p) \to \Prob_{\cat{E}}(i, q)$ is a complemented inclusion by \cref{decidable-closure}.
  We construct $t$ by using $s$ on $\Prob_{\cat{E}}(i, p)$ and a given section on its complement.
\end{proof}

\begin{theorem}[Enriched \smo{}]\label{esmo}
  Let $I = (i \from A_i \to B_i \ | \ i \in I)$ be a countable set of \ldi{}s between finite objects of $\cat{E}^D$. Then
  $I$-cofibrations and $I$-fibrations form an \ewfs{} in $\cat{E}^D$.
\end{theorem}

\begin{proof}
  For a morphism $p_0 \from X_0 \to Y$ we form a sequence $X_0 \to X_1 \to X_2 \to \ldots$ in
  $\cat{E} \slice Y$ by iteratively taking pushouts
  \begin{tikzeq*}
    \matrix[diagram,column sep={between origins,10em}]
    {
      |(A)| \bigcoprod_{i \in I} \Prob_{\cat{E}}(i, p_k) \tensorEsE A_i & |(k)|  X_k                           & [-4em] \\
      |(B)| \bigcoprod_{i \in I} \Prob_{\cat{E}}(i, p_k) \tensorEsE B_i & |(k1)| X_{k + 1} &                                        \\[-2ex]
                                                                           &                                      & |(Y)| Y \text{.}        \\
    };

    \draw[->] (A) to (B);
    \draw[->] (k) to (k1);
    \draw[->] (A) to (k);
    \draw[->] (B) to (k1);

    \draw[->] (B) to[bend right=20]                           (Y);
    \draw[->] (k) to[bend left=40]  node[above right] {$p_k$} (Y);

    \draw[->,dashed] (k1) to node[above right] {$p_{k + 1}$} (Y);
    \pbdr{k1}{A};
  \end{tikzeq*}
  The adjoint transpose of $\Prob(i, p_k) \tensorEsE B_i \to X_{k + 1}$ witnesses
  the $X_{k + 1}$-partial \erlp{} of $p_k$ \wrt{} $i$.
  Moreover, by \cref{ldi-saturation}, $X_k \to X_{k + 1}$ is a \ldi{}.
  Thus \cref{partial-lifts-colimit} applies and shows that $\colim_k X_k \to Y$ is
  an $I$-fibration.
  Using \cref{saturation}, we show that $X_0 \to \colim_k X_k$ is an $I$-cofibration.
\end{proof}

\begin{remark}
  Essentially the the same argument used to prove \cref{esmo} can be used to  prove a more general statement.
  Namely, instead of $\cat{E}^D$ we consider an \emph{$\cat{E}$-module} $\cat{C}$, i.e., a category equipped with
  a \emph{tensor} functor $\uvar \tensorEsE \Uvar \from \cat{E} \times \cat{C} \to \cat{C}$ that is associative
  in the sense that the functor $\cat{E} \to \End \cat{C}$, given by $E \mapsto (E \tensorEsE \uvar)$,
  is monoidal (\wrt{} the \cartesian product on $\cat{E}$ and functor composition on $\End \cat{C}$).
  Then $\cat{C}$ carries a $\psh \cat{E}$-enrichment defined in the same way as the one on $\cat{E}^J$ which yields
  notions of an \elp{} and an \ewfs{}.
  The complication lies in the fact that the definition of \ldi{}s is not available in $\cat{C}$.
  However, if we assume that $\cat{C}$ is equipped with a class of morphisms $\fml{D}$ satisfying
  the conclusion of \cref{ldi-saturation}, then the proof of \cref{esmo} applies without changes.
  (Note that in this case the notion of finiteness in $\cat{C}$ depends on the choice of $\fml{D}$.)
  Examples of categories that can be endowed with such structure include the categories of internal categories in $\cat{E}$, internal groupoids in $\cat{E}$
  and marked simplicial objects in $\cat{E}$.
\end{remark}

We conclude this section by introducing the notion of a cell complex and establish a few results that will be useful later.

\begin{definition} \label{thm:cell-complex-def}
  For a family of maps $I = (i \from A_i \to B_i \ | \ i \in I)$, an \emph{$\cat{E}$-enriched $I$-cell complex} is a morphism of $\cat{E}^D$ that is a sequential colimit of maps $X \to Y$ arising as pushouts
  \begin{tikzeq*}
    \matrix[diagram,column sep={between origins,6em}]
    {
      |(A)| \bigcoprod_i E_i \tensorEsE A_i & |(X)| X \\
      |(B)| \bigcoprod_i E_i \tensorEsE B_i & |(Y)| Y \\
    };

    \draw[->] (A) to (B);
    \draw[->] (X) to (Y);
    \draw[->] (A) to (X);
    \draw[->] (B) to (Y);
    \pbdr{Y}{A};
  \end{tikzeq*}
  for some family $(E_i)_{i \in I}$ of objects of $E$.
\end{definition}

Below, we simply speak of an $I$-cell complex for brevity.

\begin{proposition} \label{cofibration-as-retract-of-cell-complex}
  Under the hypotheses of \cref{esmo}, a morphism of $\cat{E}^D$ is an $I$-cofibration \iff{} it is a codomain retract of an $I$-cell complex.
  In particular, every $I$-cofibration is a levelwise complemented inclusion.
\end{proposition}

\begin{proof}
  A retract of an $I$-cell complex is an $I$-cofibration by \cref{saturation}.
  It is furthermore a levelwise complemented inclusion by \cref{ldi-saturation}.
  Conversely, let $X \to Y$ be an $I$-cofibration and consider the factorisation $X \to X' \to Y$ defined in the proof of \cref{esmo}.
  Then $X \to X'$ is an $I$-cell complex by construction.
  Moreover, $X \to Y$ has the $\psh \cat{E}$-\ellp{} \wrt{} $X' \to Y$ and, in particular, it has the \ollp{}
  (by evaluating the hom-presheaves at the terminal object).
  Thus there is a lift in the diagram
  \begin{tikzeq*}
    \matrix[diagram,column sep={between origins,6em}]
    {
      |(X)| X & |(X')| X' \\
      |(Y)| Y & |(Yr)| Y  \\
    };

    \draw[->] (X)  to (Y);
    \draw[->] (X') to (Yr);
    \draw[->] (X)  to (X');
    \draw[->] (Y)  to (Yr);

    \draw[->,dashed] (Y) to (X');
  \end{tikzeq*}
  which exhibits $X \to Y$ as a codomain retract of $X \to X'$.
\end{proof}

\begin{lemma} \label{colimit-of-fibrations}
  In the setting of \cref{esmo}, the following hold.
  \begin{parts}
    \item \label{colimit-of-fibrations:coproduct}
    Consider a countable family of maps $f_k$ in the arrow category of $\cat{E}^D$.
    If $f_k$ is an $I$-fibration for all $k$, then so is the coproduct $\bigcoprod_k f$. When $\cat E$ is $\alpha$-lextensive, the same
    holds for $\alpha$-coproducts.
    \item \label{colimit-of-fibrations:pushout}
    Consider a span $f_0 \leftarrow f_{01} \to f_1$ in the arrow category of $\cat{E}^D$.
    Assume that both legs form pullback squares and that $f_{01} \to f_0$ is a levelwise complemented inclusion on codomains.
    If $f_k$ is an $I$-fibration for $k = 0, 1, 01$, then so is the pushout $\colim f$.
    \item \label{colimit-of-fibrations:sequential-composition}
    Consider a sequential diagram $f_0 \to f_1 \to \ldots$ in the arrow category of $\cat{E}^D$.
    Assume that the maps $f_k \to f_{k+1}$ form pullback squares and are levelwise complemented inclusions on codomains.
    If $f_k$ is an $I$-fibration for all $i$, then so is $\colim f$.
  \end{parts}
\end{lemma}

\begin{proof}
  In all three parts, the colimit $\colim f$ exists and is computed separately on sources and targets where they form van Kampen colimits by \cref{colimits-diagram-category}.
  Let $C$ denote the shape of the diagram (which varies over the parts).
  We check that $\colim f$ is an $I$-fibration using \cref{enriched-lifting-as-pullback-evaluation}.
  For each $i \in I$, given a section of $\hat{\ev}_i(f_c)$ for $c \in C$, we have to construct a section of $\hat{\ev}_i(\colim f)$.
  Using \cref{van-Kampen-pullback-weighted-limit} and functoriality of colimits, it suffices to construct a family of section of $\hat{\ev}_i(f_c)$ that is natural in $c \in C$.

  For \cref{colimit-of-fibrations:coproduct}, the naturality is vacuous.
  For \cref{colimit-of-fibrations:pushout}, we pull the section of $\hat{\ev}_i(f_1)$ back to a section of $\hat{\ev}_i(f_{01})$ and then use \cref{structured-fibration-extension} to replace the section of $\hat{\ev}_i(f_0)$ by one that coheres with the one of $\hat{\ev}_i(f_{01})$.
  For \cref{colimit-of-fibrations:sequential-composition}, we recurse on $k$ and use \cref{structured-fibration-extension} to replace the given section of $\hat{\ev}_i(f_{k+1})$ by one that coheres with the one of $\hat{\ev}_i(f_k)$.
  In all three cases, the sections form a $D$-shaped natural transformation as required.
\end{proof}


We consider the application functor $\app \from [\cat{C}, \cat{D}] \times \cat{C} \to \cat{D}$ and record some commonly used facts about
pushout applications in the following statement.
We regard the pushout application of a natural transformation $[\cat{C}, \cat{D}]$ to an arrow in $\cat{C}$ to be defined if the pushout in the evident commuting square exists.
Recall that the pushout application is the induced arrow from the pushout corner.

\begin{lemma} \label{pushout-application}
  Let $u \from X \to Y$ be a map in $[\cat{C}, \cat{D}]$.
  Then pushout application $\hat{\app}(u, -) \from \cat{C}^{[1]} \to \cat{D}^{[1]}$ forms a partial functor with the following properties.
  \begin{parts}
    \item \label{pushout-application:colimit}
    Let $c \from I \to \cat{C}^{^{[1]}}$ be a diagram of arrows with levelwise colimit (\ie, a colimit that is computed separately on sources and targets in $\cat{C}$).
    If $X$ and $Y$ preserve this levelwise colimit and $\hat{\app}(u, -)$ is defined on all values of $c$, then $\hat{\app}(u, -)$ preserves the levelwise colimit of $c$.
    \item \label{pushout-application:pushout}
    Let $f \to g$ be a morphism in $\cat{C}^{^{[1]}}$ that is a pushout square.
    If $X$ and $Y$ preserve this pushout and $\hat{\app}(u, -)$ is defined on $f$ and $g$, then $\hat{\app}(u, f) \to \hat{\app}(u, g)$ is a pushout square.
    \item \label{pushout-application:transfinite-composition}
    For an ordinal $\alpha$, let $A_0 \to A_1 \to \ldots \to A_\alpha$ be an $\alpha$-composition in $\cat{C}$.
    If this $\alpha$-composition is preserved by $X$ and $Y$ and $\hat{\app}(u, -)$ is defined on $A_\beta \to A_{\beta'}$ for $\beta \leq \beta' \leq \alpha$, then $\hat{\app}(u, -)$ preserves the given the $\alpha$-composition and the resulting step map at $\beta < \alpha$ is a pushout of $\hat{\app}(u, -)$ applied to $A_\beta \to A_{\beta+1}$.
  \end{parts}
\end{lemma}

\begin{proof}
  This is folklore technique in abstract homotopy theory.
  Similar proofs (in a slightly different context) can be found in~\cite{Riehl-Verity}*{Sections~4 and~5}, in particular~\cite{Riehl-Verity}*{Lemma~4.8} for \cref{pushout-application:colimit} and~\cite{Riehl-Verity}*{Lemma~5.7} for \cref{pushout-application:pushout,pushout-application:transfinite-composition}.
\end{proof}

\begin{lemma} \label{pushout-application-instance2}
  Let $F, G \from \cat{E}^{D} \to \cat{E}^{D'}$ be two functors that preserves levelwise complemented maps, their pushouts and their sequential compositions. We assume that $F$ and $G$ are equipped with isomorphisms
  \begin{align*}
    F(E \tensorEsE X) \iso E \tensorEsE F(X) &  & G(E \tensorEsE X) \iso E \tensorEsE G(X)
  \end{align*}
  natural in $E \in \cat{E}$ and $X \in \cat{E}^{D}$ (respectively, $X \in \cat{E}^{D'}$) and let $\lambda \from F \to G$ be a natural transformation compatible with these isomorphisms.
  Let $I_D \subseteq (\cat{E}^{D})^{[1]}$ and $I_{D'} \subseteq (\cat{E}^{D'})^{[1]}$ be countable sets of arrows satisfying the conditions of \cref{esmo}.
  If for each $i \in I_{D}$, the pushout application $\hat{\app}(\lambda, i)$ is an $I_{D'}$-cofibration, then for each $I_{D}$-cofibration $i$, the pushout application $\hat{\app}(\lambda, i)$ is an $I_{D'}$-cofibration.
\end{lemma}

\begin{proof}
  First, because of \cref{ldi-saturation}, all $I_{D}$-cofibrations are levelwise complemented inclusions, so their image under $F$ are again levelwise complemented inclusions and hence pushouts along them exist.
  This shows that $\hat{\app}(\lambda,i)$ always exists when $i$ is an $I_{D}$-cofibration.

  By \cref{cofibration-as-retract-of-cell-complex}, a general a $I_{D}$-cofibration is a retract of a sequential composite of pushouts of countable coproducts of the form $E \tensorEsE A \to E \tensorEsE B$ for a map $A \to B$ in $I_{D}$ and $E \in \cat{E}$. A map $E \tensorEsE i : E \tensorEsE A \to E \tensorEsE B$ is sent by $\hat{\app}(\lambda, \uvar)$ to the map $E \tensorEsE \hat{\app}(\lambda,i)$, so as we are assuming that for each $i \in I_D$ the map $\hat{\app}(\lambda,i)$ is an $I_{D'}$-cofibration, it follows that the map of the form $E \tensorEsE i$ are also sent to $I_{D'}$-cofibration.

  Using \cref{pushout-application} one concludes that any transfinite composition of pushouts of maps of the form $E \tensorEsE i$ for $i \in I_D$ is also sent by $\hat{\app}(\lambda, \uvar)$ to a $I_{D'}$-cofibration. Finally, as $\hat{\app}(\lambda, \uvar)$ is a functor it preserves retract, and so retracts of such maps are also sent to $I_{D'}$-cofibration, and this concludes the proof as any $I_D$-cofibration is a retract of such a transfinite composition of pushouts.
\end{proof}

\begin{proposition}\label{pushout-product-saturation}
  Let $j \from X \to Y$ be a morphism of $\cat{E}^D$.
  Under the hypothesis of \cref{esmo}, if $i \pprod j$ is an $I$-cofibration for all $i \in I$,
  then $f \pprod j$ is an $I$-cofibration for all $I$-cofibrations $f$.
\end{proposition}

\begin{proof}
  We apply \cref{pushout-application-instance2} to the natural transformation $\uvar \times j \from \uvar \times X \to \uvar \times Y$ of endofunctors on $\cat{E}^D$.
  Let us check the needed preservation properties of the endofunctor $\uvar \times Z$ on $\cat{E}^D$ for $Z \in \cat{E}$.
  Preservation of levelwise complemented inclusions follows from preservation of complemented inclusions in $\cat{E}$ under product with a fixed object (a consequence of lextensivity).
  Preservation of the relevant colimits involving levelwise complemented inclusions is an instance of \cref{colimits-diagram-category}.
  Preservation of tensors with objects of $\cat{E}$ reduces to associativity and commutativity of products in $\cat{E}$; this is natural, so the map $\uvar \times j \from \uvar \times X \to \uvar \times Y$ respects the witnessing isomorphism as appropriate.
\end{proof}

  \section{The two weak factorisation systems} \label{sec:cofibrations}



In this section we consider a countably lextensive category $\cat{E}$.
We construct two \wfs{}s on the category $\cats{E}$ of simplicial objects in $\cat{E}$ that will be proven to form a model structure in \Cref{sec:model-structure}.
Our main goal is to describe the resulting cofibrations in \cref{lem:cof_charac} which relies on identification of one of the factorisation systems as
a Reedy factorisation system (\cref{cof-as-reedy-di}).
In our setting, the category $\cats{E}$ has relatively few colimits and consequently much of this section is committed to
discussion of the Reedy theory under these weak hypotheses.

We will use the enriched small object argument of \cref{esmo} with the generating sets
obtained by applying the partial functor of \eqref{equ:fset} to the sets of boundary inclusions
and horn inclusions in~\eqref{equ:boundary-horns}, \ie,
\begin{equation*}
  I_{\cats{E}} = \{ \fset{\bdsimp{n}} \to \fset{\simp{n}} \mid n \ge 0 \} \text{ and }
  J_{\cats{E}} = \{ \fset{\horn{n,k}} \to \fset{\simp{n}} \mid n \ge k \ge 0, n > 0 \} \text{.}
\end{equation*}
We will refer to $\fset{\simp{m}}$ as a simplex in $\cats{E}$ and similarly for boundaries and horns.
We say that a map in $\cats E$ is a \emph{cofibration} if it is a $I_{\cats{E}}$-cofibration
and that it is a  \emph{trivial cofibration} if it is a~$J_{\cats{E}}$-cofibration.
Moreover, we note that notions of (Kan) fibrations and trivial (Kan) fibrations as introduced in \cref{def:fibration_algebraic} coincide
with the notions of $J_{\cats{E}}$-fibrations and $I_{\cats{E}}$-fibration.

\begin{proposition} \label{fibration-levelwise} Let $f \from X \to Y$ be a map in $\cats{E}$.
\begin{parts}
\item $f$ is a fibration if and only if it is a $J_{\cats{E}}$-fibration;
\item $f$ is a trivial fibration if and only if it is a $I_{\cats{E}}$-fibration.
\end{parts}
\end{proposition}

\begin{proof}
By \cref{enriched-lifting-as-pullback-evaluation}, the condition of \cref{def:fibration_algebraic} for $f$ being a (trivial) Kan fibration is equivalent to the $\cat{E}$-enriched right lifting property of $f$ with respect to $J_{\cats{E}}$ (respectively, $I_{\cats{E}}$).
\end{proof}

The existence of \wfs{}s linking these classes is a direct consequence of the results of \cref{sec:enrwfs}.

\begin{theorem} \label{two-ewfss} Let $\cat E$ be a countably lextensive category. The
category $\cats E$ of simplicial objects in $\cat E$ admits two weak factorisation systems:
  \begin{itemize}
  \item cofibrations and trivial fibrations, cofibrantly generated by $I_{\cats{E}}$;
  \item trivial cofibrations and fibrations, cofibrantly generated by $J_{\cats{E}}$.
  \end{itemize}
\end{theorem}

\begin{proof}
  All morphisms of $I_{\cats{E}}$ and $J_{\cats{E}}$ are \ldi{}s since $S \mapsto \fset{S}$ preserves \di{}s.
  Moreover, their domains and codomains are finite colimits of representables and thus
  \cref{finite-colimit-finite} implies that the assumptions of \cref{esmo} are satisfied.
\end{proof}

Recall that $\cat{E}$ admits a \wfs{} consisting of complemented inclusions as left maps and split epimorphisms as right maps.
We now wish to characterise our cofibrations and trivial fibrations in terms of the induced Reedy weak factorisation on $\cats{E}$.
Traditional treatments of Reedy theory such as~\cite{Riehl-Verity} tacitly assume that the underlying category is bicomplete; this is not the case here.
Separately, there is the treatment~\cite{rb} of Reedy theory in the context of a (co)fibration category, but it only considers Reedy left or right maps between Reedy left or right objects; in our setting, not all objects are Reedy cofibrant or fibrant.
Let us thus discuss some of the details of the Reedy \wfs{} on $\cats{E}$.

Let $m \geq 0$.
We write $\corepresentable{m}$ for $\Delta([m], -)$, \ie, the functor in $[\Delta, \Set]$ corepresented by $m$.
The \emph{coboundary} $\coboundary{m}$ of $\Delta$ at level $m$ is the subobject of $\corepresentable{m}$ consisting of those maps which are not face operators.
Equivalently, $\coboundary[k]{m} \subseteq \Delta([m], [k])$ consists of those maps $[m] \to [k]$ whose degeneracy-face factorisation has non-identity degeneracy operator.

Let $A \in \cats{E}$.
The \emph{latching object} $L_m A$, if it exists, is the colimit of $A$ weighted by $\coboundary{m}$.
We have a canonical map $L_m A \to A$.

Let $i \from A \to B$ be a map in $\cats{E}$ and $m \geq 0$.
We wish to consider the relative latching map of $i$.
Ordinarily, we would define it as the map $A_m \push_{L_m A} L_m B \to B_m$.
However, its domain depends on the existence of the latching objects $L_m A$ and $L_m B$ and a pushout.
We wish to avoid these assumptions.
Consider the functor $\cats{E} \to \Set$ sending $X$ to the set of pairs consisting of a map $u \from A_m \to X$ and a natural family $v_f \from B_k \to X$ for $f \from [m] \to [k]$ not a face operator such that $u \circ A f = v_f \circ i_k$.
If this functor has a corepresenting object, we denote it by $A_m \push_{L_m A} L_m B$ and obtain the \emph{relative latching map} $A_m \push_{L_m A} L_m B \to B_m$ of $i$ at level $m$.
If $L_m A$ and $L_m B$ exist, this agrees with the description in terms of the pushout suggested by our notation.

We desire a more abstract view on the relative latching map.
For this, we introduce the notion of pushout weighted colimit.
Consider the two-variable functor
\begin{equation} \label{abstract-weighted-colimit}
H \from [\Delta, \Set]^\op \times \cats{E}^\op \to [\cat{E}, \Set]
\end{equation}
sending $W$ and $X$ to $I \mapsto [\Delta, \Set](W, \cat{E}(X(-), I))$
Recall that a $W$-weighted colimit of $X$, denoted $\colim^W X$, is by definition a representing object of $H(W, X)$.
The \emph{pullback construction} of $H$
is the two-variable functor
\[
\widehat{H} \from ([\Delta, \Set]^\op)^{[1]} \times (\cats{E}^\op)^{[1]} \to [\cat{E}, \Set]^{[1]}
\]
sending $w \from U \to V$ in $[\Delta, \Set]$ and $i \from A \to B$ in $\cats{E}$ to the map
\begin{equation} \label{pullback-construction-for-pushout-weighted-colimit}
H(V, B) \to H(V, A) \times_{H(U, A)} H(U, B)
\end{equation}
in $[\cat{E}, \Set]$.
Assume that domain and codomain of~\eqref{pullback-construction-for-pushout-weighted-colimit} have representing objects $Y$ and $X$, respectively (in particular, $Y$ is the $V$-weighted colimit of $B$).
Then under the Yoneda embedding of $\cat{E}^\op$ into $[\cat{E}, \Set]$, \eqref{pullback-construction-for-pushout-weighted-colimit} corresponds to a map $X \to Y$ in $\cat{E}$.
We define this to be the \emph{pushout weighted colimit} with $w \from U \to V$ of $i \from A \to B$ and denote it by $\widehat{\colim}^w i$.
It forms a partial two-variable functor
\[
{\textstyle\widehat{\colim}^{(\uvar)} (\Uvar)} \from [\Delta, \Set]^{[1]} \times \cats{E}^{[1]} \to [\cat{E}, \Set]^{[1]}
.\]
Note that this is more general than a partially defined pushout construction of the two-variable weighted colimit functor because we do not require the individual colimits of $A$ with weight $V$ and $B$ with weights $U$ and $V$ to exist.

Unfolding the codomain of~\eqref{pullback-construction-for-pushout-weighted-colimit}, we see that the relative latching map of $i \from A \to B$ at level $m$ is precisely the pushout weighted colimit of $i$ with the coboundary inclusion $\coboundary{m} \to \corepresentable{m}$.
Each side exist when the other does.
This point of view is useful because it enables us to obtain pushout weighted colimits of $i$ with certain inclusions as cell complexes of relative latching maps.

We call a map $i$ a \emph{Reedy complemented inclusion} if, for all $m$, the relative latching map of $i$ at level $m$ exists and is a complemented inclusion.
This condition for $m < k$ suffices to guarantee the existence of the relative latching map at level $m = k$.
Thus, in the inductive verification that a map is a Reedy complemented inclusion, the relevant latching maps always exist.
Given a map $X \to Y$ in $\cats{E}$, the \emph{relative matching map} at level $m$ is its weighted limit, \ie, pullback evaluation, at $\partial \simp{m} \to \simp{m}$, \ie, the map $X_m \to Y_m \times_{\ev_{\partial \simp{m}} Y} \ev_{\partial \simp{m}} X$.
We call $X \to Y$ a \emph{Reedy split epimorphism} if all its relative matching maps are split epimorphisms.

Following standard Reedy theory, Reedy complemented inclusions and Reedy split epimorphisms form a \wfs{}.
For this, we observe that instantiating the treatment of~\cite{Riehl-Verity} and making use of \cref{pushout-application}, the use of (co)limits in $\cats{E}$ may be reduced to pushouts along complemented inclusions and pullbacks along split epimorphisms. We now relate this weak factorisation system to that of cofibrations and trivial fibrations, given
in \cref{two-ewfss} (\cf also \cref{fibration-levelwise}).

\begin{proposition} \label{cof-as-reedy-di}
The \wfs{} of cofibrations and trivial fibrations of \cref{two-ewfss,fibration-levelwise} coincides with the \wfs{} of Reedy complemented inclusions and Reedy split epimorphisms.
\end{proposition}

\begin{proof}
Two \wfs{}s coincide as soon as their right classes do.
But, by inspecting the definition of a trivial fibration in \cref{def:fibration_algebraic}, a map in $\cats E$ is a Reedy split epimorphism if and only if it is a trivial Kan fibration.
\end{proof}


The next lemma will be useful to simplify some saturation arguments in \cref{sec:dependent-products}, as it allows us to avoid considering
retracts, \cf the notion of a cell complex in \cref{thm:cell-complex-def}.

\begin{lemma} \label{thm:cell-pres-cof}
  Every cofibration in $\cats{E}$ is an $I_{\cats{E}}$-cell complex.
\end{lemma}

\begin{proof}
  If $A \to B$ is a cofibration, then $B$ can be written as the colimit of its skeleta relative to $A$:
  \begin{tikzeq*}
  \matrix[diagram,column sep={5em,between origins}]
  {
    |(-1)| \Sk^{-1}_A B & |(0)| \Sk^0_A B & |(1)| \Sk^1_A B & |(D)| \ldots \\
  };

  \draw[->] (-1) to (0);
  \draw[->]  (0) to (1);
  \draw[->]  (1) to (D);
  \end{tikzeq*}
  where $\Sk^{-1}_A B = A$ and for $k \ge 0$ the square
  \begin{tikzeq*}
  \matrix[diagram,column sep={14em,between origins}]
  {
      |(b)| B_k \times \bdsimp{k} \union (A_m \push_{L_m A} L_m B) \times \simp{k} & |(k-1)| \Sk^{k - 1}_A B \\
      |(p)| B_k \times \simp{k}                                                    & |(k)|   \Sk^k_A B       \\
  };

  \draw[->] (b) to (k-1);
  \draw[->] (p) to (k);

  \draw[->] (b)   to (p);
  \draw[->] (k-1) to (k);
  \end{tikzeq*}
  is a pushout.
  These statements are justified analogously to the proofs of \cite{GSS}*{Lemma~2.3.1, Corollary~2.3.3}.
  The colimits used in the construction exist by \Cref{colimits-diagram-category} since they are colimits of sequences of \ldi{}s and pushouts along \ldi{}s
  which is ensured by the assumption that $A \to B$ is a cofibration.
\end{proof}

Our next goal is to provide a characterisation of cofibrations in terms of actions of degeneracy operators, stated in \cref{lem:cof_charac} below.
This is a generalisation of~\cite{Hwms}*{Proposition 5.1.4} or \cite{GSS}*{Proposition~1.4.4} to a setting without arbitrary colimits.
The proof is made significantly more complex by the fact that $\cat E$ is not assumed to be a Grothendieck topos.
Instead, the required exactness properties are substituted by \cref{thm:subobject-colimit}.
We also need the following statement.
For this, we observe that our discussion of Reedy theory and latching objects for the case of $\Delta$ applies just as well to arbitrary countable Reedy categories of countable height.
Note that the assumption of a Reedy cofibrant diagram includes the hypothesis that all latching objects exist.

\begin{lemma} \label{thm:reedy-colims}
Let $D$ be a finite direct category.
Let $F \from D \to \cats{E}$ be a Reedy cofibrant diagram.
Then the colimit of $F$ exists and is van Kampen.
\end{lemma}

\begin{proof}
We proceed by induction on the height of $D$.
For height $0$, note that $D$ is the empty and the claim holds because initial objects are van Kampen since $\cats{E}$ is lextensive.

Now assume the claim for height $n$ and let $D$ have height $n + 1$.
Let $D'$ of height $n$ denote the restriction of $D$ to objects of degree below $n$.
Let $I$ be the collection of objects of $D$ of degree $n$.
As per usual Reedy theory, we may compute the colimit of $F$ as the following pushout:
\begin{tikzeq*}
\matrix[diagram,column sep={14em,between origins}] {
  |(L)| \bigcoprod_{i \in I} L_i F & |(D')| \colim_{D'} F|_{D'}          \\
  |(y)| \bigcoprod_{i \in I} F(i)  & |(D)| \colim_D F           \rlap{.} \\
};

\draw[->] (L)  to (y);
\draw[->] (D') to (D);
\draw[->] (L)  to (D');
\draw[->] (y)  to (D);
\pbdr{D}{L};
\end{tikzeq*}
Here, the left map is a cofibration because it is a finite coproduct of cofibrations, and hence the pushout exists and is van Kampen by \cref{coproduct-decidable-colimit}.
By the inductive hypothesis, the colimit computing the latching object $L_i F$ for $i \in I$ is van Kampen, and so is the colimit of $F|_{D'}$.
The finite coproducts are van Kampen since $\cats{E}$ is lextensive.
Using the characterisation of van Kampen colimits given by \cref{van-Kampen-pseudo-limit}, one sees that $\colim_D F$ is van Kampen.
\end{proof}

\begin{theorem}[Characterisation of cofibrations] \label{lem:cof_charac}
Let $i \from A \to B$ be a map in $\cats{E}$.
Then the following are equivalent:
\begin{parts}
\item \label{lem:cof_charac:cof}
the map $i$ is a cofibration;
\item
the map $i$ is a levelwise complemented inclusion and the map $A_m \push_{A_n} B_n \to B_m$ is a complemented inclusion for every degeneracy operator $[m] \sto [n]$.
\label{lem:cof_charac:local} 
\end{parts}
\end{theorem}

\begin{proof}
We use from \cref{cof-as-reedy-di} that cofibrations are the same as Reedy complemented inclusions.
As in~\cite{Riehl-Verity}, we work freely with pushout weighted colimits in $\cat{E}$, with index category both $\Delta$ and its wide subcategory $\Delta_-$ of degeneracy operators.
As explained above (in the case of $\Delta$), these are partial two-variable functors in our situation.
Mirroring our notation for $\Delta$, we write $\coboundarydeg{m}$ for the subobject of $\corepresentabledeg{m} = \Delta([m], -)$ in $[\Delta_-, \Set]$ consisting of the non-identity maps.
Recall that the coboundary inclusion $\coboundary{m} \to \corepresentable{m}$ arises as left Kan extension along $\Delta_- \to \Delta$ of the coboundary inclusion $\coboundarydeg{m} \to \corepresentabledeg{m}$.
For working with weighted colimits, we recall that left Kan extension on the side of the weight corresponds to restriction on the side of the diagram.

We start with the direction from~\ref{lem:cof_charac:cof} to~\ref{lem:cof_charac:local}.
Let $i$ be a Reedy complemented inclusion.
Then
the pushout weighted colimit of $i$ with any finite cell complex (finite composite of pushouts) of coboundary inclusions is a complemented inclusion.
In particular, the pushout weighted colimit of the restriction $i |_{\Delta_-}$ of $i$ to $\Delta_-$ with any finite cell complex of coboundary inclusions $\coboundarydeg{k} \to \corepresentable{k}$ of $\Delta_-$ is a complemented inclusion.
For $m \geq 0$, the map $A_m \to B_m$ is the pushout weighted colimit of $i |_{\Delta_-}$ with such a finite cell complex $\varnothing \to \corepresentabledeg{m}$, hence a complemented inclusion.
Every degeneracy operator $[m] \sto [n]$ is a split epimorphism.
It follows that $\corepresentabledeg{n} \to \corepresentabledeg{m}$ is an inclusion with levelwise finite complement, thus we can write it as a finite cell complex of coboundary inclusions of $\Delta_-$.
Therefore, the pushout weighted colimit of $i$ with $\corepresentabledeg{n} \to \corepresentabledeg{m}$ is a complemented inclusion.
But this is the map $A_m \push_{A_n} B_n \to B_m$.

We finish with the direction from~\ref{lem:cof_charac:local} to~\ref{lem:cof_charac:cof}.
We show that the relative latching map $A_m \push_{L_m A} L_m B \to B_m$ of $i$ is a complemented inclusion by induction on $m$.
Recall that this is the pushout weighted limit of $i|_{\Delta_-}$ with $\coboundarydeg{m} \to \corepresentabledeg{m}$.
Let $\partial(\Delta_-^\op \downarrow [m])$ denote the opposite of the poset of non-identity degeneracy operators with source $[m]$.
Consider the diagram $F \from \partial(\Delta_-^\op \downarrow [m]) \to \cat{E} \downarrow B_m$ sending a degeneracy operator $[m] \sto [n]$ to the object $A_m \push_{A_n} B_n$.
It lives canonically under the object $A_m$ over $B_m$.
By switching from the weighted colimit to the conical colimit point of view, the object $A_m \push_{L_m A} L_m B$ is the colimit of $F$ in the category of factorisations of $A_m \to B_m$.
Equivalently, in the slice over $B_m$, the object $A_m \push_{L_m A} L_m B$ is the colimit of the diagram $F_*$ that is $F$ with shape adjoined with an initial object sent to $A_m$.

Note that, using our assumptions, we can regard $F$ as a diagram of complemented subobjects of $B_m$ that are bounded from below by the complemented subobject $A_m$.
It remains to show that the colimit of $F_*$ in the slice over $B_m$ has a complemented inclusion as underlying map.
It will suffice to show that this colimit is subterminal.
For then, it is given by the non-empty finite union of the subobjects that constitute the values of $F_*$, and complemented subobjects are closed under finite unions by \cref{decidable-union} of \cref{decidable-closure}.

The indexing category of $F_*$ is a finite direct category.
The latching map of $F_*$ at the initial object is $0 \to A_m$, a complemented inclusion.
The latching map of $F_*$ at an object $[m] \sto [n]$ is a pushout of the relative latching map of $A \to B$ at $[m]$, a complemented inclusion by induction hypothesis.
Thus, the diagram $F_*$ is Reedy cofibrant.
By \cref{thm:reedy-colims}, the colimit of $F_*$ is van Kampen.
All of this holds both in $\cat{E}$ as well as its slice over $B_m$.

Given a complemented subobject $U \to B_m$ and an arbitrary subobject $V \to B_m$, the pushout corner map in the pullback of $U \to B_m$ and $V \to B_m$ exists.
If it is a monomorphism, it computes the union $U \cup V \to B_m$ of the given subobjects.
Since degeneracy operators are split epimorphisms, the natural transformation $i|_{\Delta_-}$ is \cartesian.
This makes the value of $F$ at an object $[m] \sto [n]$ the union of the subobjects $A_m \to B_m$ and $B_n \to B_m$.

Since $\Delta$ is elegant~\cite{Bergner-Rezk}, given non-identity degeneracy operators $[m] \sto [n_i]$ for $i = 1, 2$, we have an absolute pushout
\begin{tikzeq*}
\matrix[diagram,column sep={between origins,6em}]{
  |(I)| [m]    & |(J1)| [n_1] \\
  |(J2)| [n_2] & |(K)| [k]    \\
};

\draw[surj] (I) to (J1);
\draw[surj] (J2) to (K);
\draw[surj] (I) to (J2);
\draw[surj] (J1) to (K);
\pbdr{K}{I};
\end{tikzeq*}
in $\Delta$ with $[n_1] \sto [k]$ and $[n_2] \sto [k]$ degeneracy operators.
Note that $[m] \sto [k]$ is distinct from the identity.
By absoluteness, we obtain a pullback
\begin{tikzeq*}
\matrix(B)[diagram,column sep={between origins,6em}]{
  |(BK)| B_k      & |(BJ1)| B_{n_1}     \\
  |(BJ2)| B_{n_2} & |(BI)| B_m \rlap{.} \\
};

\draw[->] (BJ1) to (BI);
\draw[->] (BK) to (BJ2);
\draw[->] (BJ2) to (BI);
\draw[->] (BK) to (BJ1);
\pb{BK}{BI};
\end{tikzeq*}
We now work in subobjects of $B_m$.
From the above pullback, we have $B_k = B_{n_1} \cap B_{n_2}$.
Using from \cref{lem:pushout_decidable} twice that pushouts along complemented inclusions are stable under pullback, we compute
\begin{align*}
(A_m \cup B_{n_1}) \cap (A_m \cup B_{n_2})
&=
((A_m \cup B_{n_1}) \cap A_m) \cup ((A_m \cup B_{n_1}) \cap B_{n_2})
\\&=
A_m \cup ((A_m \cup B_{n_1}) \cap B_{n_2})
\\&=
A_m \cup (B_{n_1} \cap B_{n_2})
\\&=
A_m \cup B_k.
\end{align*}
We obtain, in subobjects of $B_m$, that $F$ at $[m] \sto [n]$ is the intersection (computed as pullback) of $F$ at $[m] \to [n_1]$ and $[m] \to [n_2]$.
Thus, in subobjects of $B_m$, the diagram $F$ (and then also $F_*$) preserves binary meets.
Recollecting from above that the colimit of $F_*$ in the slice over $B_m$ is van Kampen, \cref{thm:subobject-colimit} shows that it is subterminal.
\end{proof}

\section{Closure properties of cofibrations}
\label{sec:prop-of-cof}

This section is devoted to further study of weak factorisation systems constructed in \cref{sec:cofibrations}, in preparation
for the proof of the existence of the effective model structure. We begin with a simple verification.

%

\begin{lemma} \label{fset-pres-cof}
If $A \to B$ is a (trivial) cofibration between levelwise countable simplicial sets, then
$\fset{A} \to \fset{B}$ is a (trivial) cofibration in $\cats E$.
\end{lemma}

\begin{proof}
Recall that the partial functor $X \mapsto \fset{X}$ is a partial left adjoint to the levelwise global sections functor.
This is equivalently the functor $\Hom_\sSet(1, -)$ with $1 \in \cats{E}$ from \cref{sec:fib_cat}.
By adjointness using the \wfs{}s of \cref{two-ewfss} and \cref{fibration-levelwise}, it suffices to show that $\Hom_\sSet(1, -)$ preserves (trivial) fibrations.
This holds by \cref{prop:fibration_pointwise}.
\end{proof}

\begin{proposition} \label{triv-fib-is-fib} \leavevmode
\begin{parts}
\item Trivial fibrations are fibrations.
\item Trivial cofibrations are cofibrations.
\end{parts}
\end{proposition}

\begin{proof}
The first part is immediate since trivial Kan fibrations are Kan fibrations in simplicial sets.
The second parts follows by adjointness using the \wfs{}s of \cref{two-ewfss}.
\end{proof}


We now establish some formal properties of the two enriched weak factorisation systems, regarding the
pushout-product, pushout-tensor and pullback-cotensor functors (\cf \cref{thm:leibniz-terminology}).

%

\begin{proposition}[Pushout-product properties] \label{cofibration-pushout-product}  \label{cofibration-trivial-cofibration-pushout-product} \leavevmode
\begin{parts}
\item \label{pushout-prod-part-i} In $\cats{E}$, cofibrations are closed under pushout product.
\item \label{pushout-prod-part-ii} In $\cats{E}$, the pushout product of a cofibration and a trivial cofibration is a trivial cofibration.
\end{parts}
\end{proposition}

\begin{proof} For \cref{pushout-prod-part-i}, recall that cofibrations in $\sSet$ are closed under pushout product.\footnote{See~\cite{Hwms}*{Proposition~5.1.5} or \cite{GSS}*{Proposition~1.3.1} for the constructive version of this fact.}
  Since $S \mapsto \fset{S}$ preserves pushouts and products, it follows that the pushout product of generating cofibrations in $\cats{E}$
  is a cofibration.
  The same follows for general cofibrations in~$\cats{E}$ by \cref{pushout-product-saturation}.
  These pushout products exist by \cref{ldi-pushout-product}.

  For \cref{pushout-prod-part-ii}, The result holds in~$\sSet$ by\footnote{See~\cite{Hwms}*{Corollary~5.2.3} or \cite{GSS}*{Proposition~1.3.1} for the constructive version of this fact.}~\cite{GZ}*{Proposition~IV.2.2} and thus it carries over to~$\cats{E}$ by
  the argument of \cref{pushout-prod-part-i}.
\end{proof}

\begin{lemma}\label{sSet-tensor}
  Let $X \in \cats{E}$. For every finite simplicial set $K$, the tensor $K \tensorsSetsE X$
  exists and is given by $\fset{K} \times X$.
\end{lemma}


\begin{proof}
  Given $Y \in \cats{E}$, a morphism $X \to K \cotensor Y$ consists of a family of morphisms
  $X_m \to Y_n^{(K \times \simp{m})_n}$, natural in $m$ and dinatural in $n$.
  This corresponds to a family of morphisms $\fset{K \times \simp{m}}_n \times X_m \to Y_n$,
  dinatural in $m$ and natural in $n$. Moreover:
  \[ \fset{K \times \simp{m}}_n \times X_m = \fset{K_n} \times \fset{\Hom([m],[n])} \times X_m \text{.} \]
Since $\coend^{[m]} \fset{\Hom([m],[n])} \times X_m = X_n$, such family of maps corresponds to a morphism $\fset{K}_n \times X_n \to Y_n$ natural in $n$, i.e., a morphism $\fset{K} \times X \to Y$ in $\cats E$.
\end{proof}

\begin{proposition}[Pushout tensor properties] \label{pushout-tensor}
  Let $A \to B$ be a cofibration between finite simplicial sets. Then,
  the pushout tensor with $A \to B$ exists. Furthermore,
  \begin{parts}
  \item \label{pushout-tensor:cof-triv-cof} it preserves trivial cofibrations,
  \item \label{pushout-tensor:cof-cof-cof} it preserves cofibrations,
  \item \label{pushout-tensor:triv-cof-cof} if $A \to B$ is a trivial cofibration, then
  it sends cofibrations to trivial cofibrations.
  \end{parts}
\end{proposition}

\begin{proof} The existence follows from \cref{colimits-diagram-category} and \cref{sSet-tensor}.
  These other statements  are dual to the ones of \cref{item:pbcot} of \cref{lem:basic_pointwise_fibrations} under the tensor-cotensor adjunction of \cref{sSet-tensor}.
  Note that for this conclusion it suffices to consider the underlying ordinary \wfs{} of \cref{ordinary-wfs}
  so that we do not need to verify that the adjunction is enriched over $\psh \cat{E}$.
\end{proof}

We now turn our attention to the cofibrations and the cofibrant objects in $\cats E$. From \cref{sec:enrwfs}
and \cref{fibration-levelwise} these are exactly the maps with the left lifting property with respect to Kan fibrations.
The next lemma provides us with a stock of cofibrant objects.

\begin{lemma} \label{cofibrant-properties} \leavevmode
\begin{parts}
\item \label{constant-cofibrant} Let $E \in \cat{E}$. The constant simplicial object $E \in \cats{E}$ is cofibrant.
\item  \label{generators-between-cofibrant} The domains and codomains of all morphisms of $I_{\cats{E}}$ and $J_{\cats{E}}$ are cofibrant.
\item  \label{cofibrant-cotensor} Let $X \in \cats{E}$ and  $K$ be a finite simplicial set. If $X$ is cofibrant, then so is $K \cotensor X$.
\end{parts}
\end{lemma}

\begin{proof} For \cref{constant-cofibrant},  by \cref{saturation}, the tensor of $\fset{\partial \simp{0}} \to \fset{\simp{0}}$ with $E$ is a cofibration.
By \cref{sSet-tensor}, this map is the tensor of $E \in \cats{E}$ with $\partial \simp{0} \to \simp{0}$, \ie, the map $\varnothing \to E$ in $\cats{E}$.
\Cref{generators-between-cofibrant} holds since $S \mapsto \fset{S}$ preserves cofibrations by \cref{fset-pres-cof}.\footnote{Constructively,  for \cref{generators-between-cofibrant} one needs to check also that the relevant objects are cofibrant in~$\sSet$. The simplices and their boundaries are cofibrant in $\sSet$ by~\cite{GSS}*{Lemma~1.3.5} and the horns by~\cite{GSS}*{Lemma~1.4.9}.}
Finally, for \cref{cofibrant-cotensor}, if $[m] \sto [n]$ is a degeneracy operator, then the map $(K \cotensor X)_n \to (K \cotensor X)_m$ can be identified with the map $X(K \times \simp{n}) \to X(K \times \simp{m})$. It follows from~\cite{H}*{Proposition 3.1.11} that when $K$ is a finite simplicial set, the map $K \times \simp{n} \to K \times \simp{m}$ is a finite composite of pushouts of degeneracy operators. This implies that the map $(K \cotensor X)_n \to (K \cotensor X)_m$ is a finite composite of pullbacks of degeneracy operator $X_a \to X_b$. As $X$ is cofibrant these maps are all \di{}s, hence as \di{}s are closed under pullback and composition, this implies that $(K \cotensor X)_n \to (K \cotensor X)_m$ is a \di{} as well.
\end{proof}

\begin{lemma}\label{cofibration-mono-pullback}
  Cofibrations are closed under pullback along a monomorphism.
\end{lemma}

\begin{proof} Consider a pullback square of simplicial objects:
  \begin{tikzeq*}
  \matrix[diagram]
  {
    |(S')| S' & |(S)| S \\
    |(A)| A & |(B)| B \\
  };

  \draw[->] (S') to (S);
  \draw[inj] (S)  to (B);
  \draw[cof] (A) to (B);
  \draw[inj] (S') to (A);
  \end{tikzeq*}
We check that $S' \to S$ is a cofibration using characterisation \ref{lem:cof_charac:local} of \cref{lem:cof_charac}. In an lextensive category, a pullback of a \di{} is a \di{}, hence the map $S' \to S'$ is a \ldi{}. Given any degeneracy operator $[m] \sto [n]$, as it is a split epimorphism and $S \to B$ is a monomorphism, the naturality square:

  \begin{tikzeq*}
  \matrix[diagram]
  {
    |(Sn)| S_n & |(Sm)| S_m \\
    |(Bn)| B_n & |(Bm)| B_m \\
  };

  \draw[inj] (Sn) to (Sm);
  \draw[inj] (Bn)  to (Bm);
  \draw[inj] (Sm) to (Bm);
  \draw[inj] (Sn) to (Bn);

  \end{tikzeq*}
is a pullback. The pushout $B_m \push_{A_m} A_n$ is a van Kampen colimit because the map $A_m \to B_m$ is a \di{}, it hence follows that we have a pullback square:

  \begin{tikzeq*}
  \matrix[diagram,column sep={between origins, 7em}]
  {
    |(DS)| S_n \push_{S'_n} S'_m & |(Sm)| S_m \\
    |(DB)| B_n \push_{A_n} A_m & |(Bm)| B_m \\
  };

  \draw[->] (DS) to (Sm);
  \draw[->] (DB)  to (Bm);
  \draw[inj] (DS) to (DB);
  \draw[inj] (Sm) to (Bm);

  \end{tikzeq*}
and hence as the bottom map is a \di{} by assumption, the top map is also a \di{}. This shows that $S' \to S$ is a cofibration.
\end{proof}

As discussed just before \cref{pullback-cotensor-slice}, the slice $\cats{E} \slice X$ is enriched over simplicial sets
and has cotensors by finite simplicial sets.
Under the present hypotheses, it also has tensors by finite (and even countable) simplicial sets, which are simply tensors in the underlying category $\cats{E}$.

\Cref{pullback-cotensor-slice-sec-4} of the next Proposition extends the pullback cotensor properties of \cref{item:pbcot} of \cref{lem:basic_pointwise_fibrations} to slice categories.

\begin{proposition} Let $X \in \cats{E}$.
\begin{parts}
\item \label{cofibration-pushout-product-slice}
Pushout products of cofibrations in $\cats{E} \slice X$ exist. Moreover, cofibrations in $\cats{E} \slice X$ are closed under pushout product.
\item \label{pushout-tensor-slice}  The pushout tensor properties of \cref{pushout-tensor} hold also in $\cats{E} \slice X$.
\item \label{pullback-cotensor-slice-sec-4}
The pullback cotensor in $\cats{E} \slice X$ of  a cofibration between finite simplicial sets and a fibration is a fibration.
  If the given cofibration or fibration  is trivial, then the result is a trivial fibration.
\end{parts}
\end{proposition}

\begin{proof} For \cref{cofibration-pushout-product-slice}, recall that pushout products in $\cats{E} \slice X$ are computed from pushout products in $\cats{E}$ by
  pulling back along the diagonal $X \to X \times X$.
  Since the latter is a monomorphism, the conclusion follows from \cref{cofibration-pushout-product} and \cref{cofibration-mono-pullback}.
  For \cref{pushout-tensor-slice}, note that the forgetful functor~$\cats{E} \slice X \to \cats E$ preserves tensors and pushouts and thus the pushout tensor properties
  follow directly from \cref{pushout-tensor}. \Cref{pullback-cotensor-slice-sec-4}
  was already established as \cref{pullback-cotensor-slice}, but now it also
    follows by the tensor-cotensor adjunction.
\end{proof}

\begin{proposition}  \label{thm:pullback-combined} \leavevmode
\begin{parts}
\item \label{pullback-cofibration}
  Let $f \from X \to Y$ be a morphism in $\cats{E}$. If $X$ is cofibrant, then
 the pullback functor $f^* \from \cats{E} \slice Y \to \cats{E} \slice X$ preserves cofibrations.
\item \label{pullback-cofibrant}
Let $A \to X$ and $B \to X$ be morphisms in $\cats{E}$. If $A$ and $B$ are cofibrant, then
so is $A \times_X B$.
\item \label{cofibrant-limit}
  Cofibrant objects in $\cats{E}$ are closed under finite limits.
\end{parts}
\end{proposition}

\begin{proof} For~\ref{pullback-cofibration},
  if $A \to B$ is a cofibration over $Y$, then its pullback along $f \from X \to Y$ coincides with the pushout product of
  $A \to B$ and $\emptyset \to X$ in $\cats{E} \slice Y$, which is a cofibration by \cref{pushout-prod-part-i} of \cref{cofibration-pushout-product}.
  \Cref{pullback-cofibrant} is a special case of \cref{pullback-cofibration}.
  Finally, for \cref{cofibrant-limit}, it suffices to check that cofibrant objects are closed under pullback and that the terminal object is cofibrant.
  The former follows from \cref{pullback-cofibrant}.
  The latter follows by definition since $0 \to 1$ is a generating cofibration.
\end{proof}
  
\section{Pushforward along cofibrations}
\label{sec:dependent-products}


This section and \cref{sec:frobenius,sec:equivalence-extension,sec:model-structure}
constitute the third part of the paper, in which we show how the two weak factorisation systems
of \cref{sec:cofibrations} give rise to the effective model structure~(\cref{model}).
For this, we shall work with a fixed countably lextensive category $\cat{E}$.
We do not assume that the category $\cat{E}$ is (locally) \cartesian closed, but we establish the existence of certain exponentials and pushforwards
required by our argument.
We also provide a criterion for the cofibrancy of some of these constructions.
We begin with a few remarks on exponentiable maps.

\begin{proposition} \label{thm:char-exponentiable}
Let $f \from X \to Y$ in $\cat{E}$. Then, \tfae{}:
\begin{conditions}
\item the pullback functor $f^*\from \cat{E} \slice Y \to \cat{E} \slice X$ has a right adjoint $f_* \from \cat{E} \slice X \to \cat{E} \slice Y$,
\item $X$ is exponentiable as an object of $\cat{E} \slice Y$.
\end{conditions}
\end{proposition}

\begin{proof}
This follows from~\cite{Elephant}*{Lemma~A1.5.2~(i)} and (the proof of)~\cite{Elephant}*{Corollary~A1.5.3}.
\end{proof}

When the equivalent conditions of \cref{thm:char-exponentiable} hold, we say that $f$ is \emph{exponentiable} and refer to the right adjoint $f_*$ as the
\emph{pushforward along $f$}.
(It is also known as the \emph{dependent product along $f$}.)

\begin{example}
Let $S$ be a finite set. Then, $\fset{S} \in \cat{E}$ defined in~\eqref{equ:fset} is exponentiable in~$\cat{E}$ and the exponential of $X$ by $\fset{S}$
is the product $X^S$. Indeed, as finite coproducts in $\cat{E}$ are universal, $\fset{S} \times X \iso \bigcoprod_{s \in S} X$.
Hence, a map $\fset{S} \times A \to X$ is the same as an $S$-indexed collection of maps $A \to X$, that is the same as a map $A \to X^S$.
\end{example}

\begin{proposition}\label{prop:pullback_exponentiable} Let
\begin{tikzeq*}
  \matrix[diagram]
  {
    |(W)| W & |(X)| X \\
    |(Y)| Y & |(Z)| Z \\
  };

  \draw[->] (W) to node[above] {$v$} (X);
  \draw[->] (Y) to node[below] {$u$} (Z);
  \draw[->] (W) to node[left]  {$g$} (Y);
  \draw[->] (X) to node[right] {$f$} (Z);
  \pb{W}{Z};
\end{tikzeq*}
be a pullback square in $\cat{E}$. If $f$ is exponentiable, then so is $g$
and the canonical natural transformation~$u^* f_* \to g_* v^*$
is an isomorphism.
\end{proposition}

\begin{proof}
This follows from~\cite{Elephant}*{Lemma~A1.5.2~(ii)} applied in the slice category over $Z$. If $K$ is an object over $W$, the pushforward $g_* K$ is constructed explicitly as the pullback:

  \begin{tikzeq*}
  \matrix[diagram]
  {
    |(gK)| g_* K & |(fK)| f_* K \\
    |(Y)| Y & |(fW)| f_* W  \\
  };

  \draw[->] (gK) to (fK);
  \draw[->] (Y) to (fW);
  \draw[->] (gK) to (Y);
  \draw[->] (fK) to (fW);
  \pb{gK}{fW};
\end{tikzeq*}
where the bottom arrow is the unit of adjunction $Y \to f_* f^* Y = f_* W$.
\end{proof}

\begin{proposition}\label{prop:VK_colim_of_exponentiable}
Let $D$ be a small category and $f_\bullet \from X_\bullet \to Y_\bullet$ a natural transformation between two $D$-diagrams in $\cat{E}$ such that $f_\bullet$ is \cartesian, $f_d$ is exponentiable for every $d \in D$,
and $Y_\bullet$ has a van Kampen colimit in $\cat{E}$. Then the colimit map
\[f
\from \colim_{d \in D} X_d \to \colim_{d \in D} Y_d
\]
is exponentiable, and up to the equivalences
 \[ \cat{E} \slice \colim_D X_d \simeq \lim_D \left( \cat{E} \slice X_d \right) \text{, }  \qquad \cat{E} \slice \colim_D Y_d \simeq \lim_D \left( \cat{E} \slice Y_d \right),\]
the functor $f_*$ coincides with the collection of functors $(f_d)_*$.
\end{proposition}

\begin{proof}
 The claim follows from a general  fact. If $F \from \cat A \to \cat B$ is a pseudo-natural transformation between two diagrams $\cat A \, , \cat B \from D \to \ncat{Cat}$ of categories such that each $F_d$ has a right adjoint $R_d$ and for each naturality square of $F_d$ the Beck--Chevalley conditions are satisfied, then the isomorphisms given by the Beck--Chevalley condition exhibit $R_d \from \cat B_d \to \cat A_d$ as a pseudo-natural transformation, and $\lim R_d$ is a right adjoint to $\lim F_d$, with the unit and counit of this adjunction being levelwise the unit and counit of the adjunction $F_d \dashv R_d$.
\end{proof}

We now move on to discuss how exponentiability interacts with cofibrancy. In particular, the aim of
the rest of the section is to prove the following result.

\begin{theorem}\label{th:Dependent_prod_along_cof}
Let $i \from A \to B$ be a cofibration between cofibrant object in $\cats E$. Then:
  \begin{parts}
  \item $i$ is exponentiable,
  \item $i_*$ sends cofibrant objects to cofibrant objects.
  \end{parts}
\end{theorem}

We will prove this theorem by a saturation argument.
For this purpose, we introduce now the class $\mathcal{G}$ of cofibrations between cofibrant objects that satisfy properties (i) and (ii) of the theorem.

\medskip

Assume $i \from A \to B$ an exponentiable monomorphism in $\cat E$. Then, for any $X \in \cat{E} \slice A$, the unit of the adjunction $i^* \dashv i_*$ induces a pullback square
\begin{tikzeq}{mono-exp-pullback}
  \matrix[diagram]
  {
    |(X)| X & |(iX)| i_* X          \\
    |(A)| A & |(B)|  B     \rlap{.} \\
  };

  \draw[->] (X) to (iX);
  \draw[->] (A) to node[below] {$i$} (B);
  \draw[->] (X) to (A);
  \draw[->] (iX) to (B);
  \pb{X}{B};
\end{tikzeq}
Indeed, since $i$ is a monomorphism, the counit $i^* i_! \to \id$ of the adjunction $i_! \dashv i^*$ is invertible, and therefore so is the unit $\id \to i^* i_*$.

\begin{lemma}\label{lem:i^*i_*=Id} \label{rk:x->i_*x_cofibrant}
Let $i \from A \to B$ be a map in $\mathcal{G}$.
For cofibrant $X \in \cat{E} \slice A$, the map $X \rightarrow i_* X$ is a cofibration.
\end{lemma}

\begin{proof} The claim follows from \cref{pullback-cofibration} of \cref{thm:pullback-combined}, since the map $X \rightarrow i_* X$ is a pullback of
  a cofibration between cofibrant objects by (\ref{mono-exp-pullback}) above.
\end{proof}


\begin{proposition}\label{prop:Pushout_in_G}
The class $\mathcal{G}$ is closed under pushouts along maps with cofibrant target.
\end{proposition}

\begin{proof}

If $i \from A \to B$ is in $\mathcal{G}$ and $f \from A \to X$ is an arbitrary arrow in $\cats E$ with $X$ cofibrant, we consider the diagram
\begin{tikzeq*}
\matrix[diagram]
{
  |(ul)| X & |(um)| A & |(ur)| A \\
  |(ll)| X & |(lm)| A & |(lr)| B \rlap{\text{.}} \\
};

\draw[->] (ul) to (ll);
\draw[->] (um) to (lm);
\draw[cof] (ur) to node[right] {$i$} (lr);

\draw[->] (um) to (ul);
\draw[->] (um) to (ur);

\draw[->] (lm) to node[below] {$f$} (ll);
\draw[cof] (lm) to node[below] {$i$} (lr);

\pb{um}{lr};
\pbur{um}{ll};
\end{tikzeq*}
Then the two squares are pullbacks (because $i$ is a monomorphism for the one on the right) the vertical maps are all exponentiable by assumption, so by \cref{prop:VK_colim_of_exponentiable}, the map between the colimit of the first row to the colimit of the second row, that is the map
\[
j: X \rightarrow X \coprod_A B \text{,}
\]
is indeed exponentiable. Moreover, still by \cref{prop:VK_colim_of_exponentiable}, if $K$ is a cofibrant object over $X$, it corresponds with respect to the van Kampen pushout of the first row to the \cartesian natural transformation
\begin{tikzeq*}
\matrix[diagram]
{
  |(ul)| K & |(um)| f^* K & |(ur)| f^*K \\
  |(ll)| X & |(lm)| A     & |(lr)| A \rlap{\text{.}} \\
};

\draw[->] (ul) to (ll);
\draw[->] (um) to (lm);
\draw[->] (ur) to (lr);

\draw[->] (um) to (ul);
\draw[->] (um) to (ur);

\draw[->] (lm) to node[below] {$f$} (ll);
\draw[->] (lm) to (lr);

\pb{um}{lr};
\pbur{um}{ll};
\end{tikzeq*}
Hence its image by $j_*$ corresponds to the \cartesian natural transformation
\begin{tikzeq*}
  \matrix[diagram]
  {
    |(ul)| K & |(um)| f^* K & |(ur)| i^* f^*K \\
    |(ll)| X & |(lm)| A     & |(lr)| B \rlap{\text{.}} \\
  };

  \draw[->] (ul) to (ll);
  \draw[->] (um) to (lm);
  \draw[->] (ur) to (lr);

  \draw[->] (um) to (ul);
  \draw[->] (um) to (ur);

  \draw[->] (lm) to node[below] {$f$} (ll);
  \draw[cof] (lm) to node[below] {$i$} (lr);

  \pb{um}{lr};
  \pbur{um}{ll};
  \end{tikzeq*}
So, by gluing along the bottom van Kampen colimit, we have a pushout square
\begin{tikzeq*}
  \matrix[diagram]
  {
    |(fK)| f^* K & |(ifK)| i_* f^* K \\
    |(K)|  K     & |(jK)|  j_*K      \\
  };

  \draw[cof] (fK) to (ifK);
  \draw[->] (K) to (jK);
  \draw[->] (fK) to (K);
  \draw[->] (ifK) to (jK);
  \pbdr{jK}{fK};
\end{tikzeq*}
where the top arrow is a cofibration by \cref{rk:x->i_*x_cofibrant} and the assumption that $i \in \mathcal{G}$ applied to the cofibrant object $f^*K$. It follows that $j_* K$ is cofibrant.\end{proof}

\begin{proposition}\label{prop:trans_comp_in_G}
The class $\mathcal{G}$ is closed under sequential composition.
\end{proposition}

\begin{proof}
The class $\mathcal{G}$ is clearly closed under finite composition.  Given an $\omega$-chain $A_0 \overset{i_0}{\cto} A_1 \overset{i_1}{\cto} A_2 \overset{i_2}{\cto} \dots$ of arrows in $\mathcal{G}$, we consider the diagram:

\begin{tikzeq*}
\matrix[diagram]
{
  |(u0)| A_0 & |(u1)| A_0 & |(u2)| A_0 & |(u3)| \ldots          \\
  |(l0)| A_0 & |(l1)| A_1 & |(l2)| A_2 & |(l3)| \ldots \rlap{.} \\
};

\draw[->] (u0) to (u1);
\draw[->] (u1) to (u2);
\draw[->] (u2) to (u3);

\draw[cof] (l0) to (l1);
\draw[cof] (l1) to (l2);
\draw[cof] (l2) to (l3);

\draw[->] (u0) to (l0);
\draw[->] (u1) to (l1);
\draw[->] (u2) to (l2);

\pb{u0}{l1};
\pb{u1}{l2};
\pb{u2}{l3};
\end{tikzeq*}

Each vertical map is in $\mathcal{G}$ as a composite of maps in $\mathcal{G}$; each square is a pullback as all these maps are monomorphisms, so by \cref{prop:VK_colim_of_exponentiable}, the comparison map $j \from A_0 \to \colim A_i$ between the two colimit is exponentiable. If $K$ is a cofibrant object over $A_0$, then again by \cref{prop:VK_colim_of_exponentiable} its image by $j_*$ corresponds to the \cartesian natural transformation:

\begin{tikzeq*}
\matrix[diagram]
{
  |(u0)| K_0 & |(u1)| K_1 & |(u2)| K_2 & |(u3)| \ldots \\
  |(l0)| A_0 & |(l1)| A_1 & |(l2)| A_2 & |(l3)| \ldots \\
};

\draw[cof] (u0) to (u1);
\draw[cof] (u1) to (u2);
\draw[cof] (u2) to (u3);

\draw[cof] (l0) to (l1);
\draw[cof] (l1) to (l2);
\draw[cof] (l2) to (l3);

\draw[->] (u0) to (l0);
\draw[->] (u1) to (l1);
\draw[->] (u2) to (l2);

\pb{u0}{l1};
\pb{u1}{l2};
\pb{u2}{l3};
\end{tikzeq*}
where $K_0 = K$ and $K_{n+1} = (i_n)_* K_n$, hence all the maps in the top row are cofibrations, and so $j_* K = \colim K_i$ is cofibrant.
\end{proof}

\begin{proposition}\label{prop:tensor_in_G}
The class $\mathcal{G}$ is closed under tensors by objects of $\cat E$.
\end{proposition}

\begin{proof} Let $i  \from A \cto B$ an arrow in $\mathcal{G}$, and let $X$ an object of $\cat E$. The square
\begin{tikzeq*}
  \matrix[diagram,column sep={6em,between origins}]
  {
    |(AX)| A \times X & |(BX)| B \times X \\
    |(A)|  A          & |(B)|  B          \\
  };

  \draw[cof] (AX) to node[above] {$j$} (BX);
  \draw[cof] (A)  to node[below] {$i$} (B);
  \draw[->] (AX) to (A);
  \draw[->] (BX) to (B);
\end{tikzeq*}
is a pullback, so $j$ is exponentiable by \cref{prop:pullback_exponentiable}. Moreover, the formula for $j_*$ given in the proof of \cref{prop:pullback_exponentiable} gives that $K$ over $A \times X$ we have a pullback square
\begin{tikzeq*}
  \matrix[diagram,column sep={6em,between origins}]
  {
    |(jK)| j_* K      & |(iK)| i_* K \\
    |(BX)| B \times X & |(AX)| i_*(A \times X) \\
  };

  \draw[->] (jK) to (iK);
  \draw[->] (BX) to (AX);
  \draw[->] (jK) to (BX);
  \draw[->] (iK) to (AX);
  \pb{jK}{AX};
\end{tikzeq*}
Since $i \in \mathcal{G}$ and $B \times X$ is cofibrant, $i_* K$ is cofibrant, and so $j_* K$ is cofibrant,
as required.
\end{proof}

In order to conclude the proof of \cref{th:Dependent_prod_along_cof}, it remains to show that the generating cofibrations $i \from \bdsimp n \cto \simp n$ are in $\mathcal{G}$.
This is based on an explicit description of $i_*$ using the characterisation of $\cats E \slice \bdsimp n$ and $\cats E \slice \simp n$ of \cref{lem:slice_presheaves_set}.

\begin{proposition}\label{lem:generating_cof_in_G}
The generating cofibrations $i \from \bdsimp n \cto \simp n$ are in $\mathcal{G}$.
\end{proposition}

\begin{proof}
Under the equivalence of \cref{lem:slice_presheaves_set}, the pullback functor $i^* \from \cats E \slice \simp n \to \cats E \slice \bdsimp n$ coincides with the functor
\[
\cats E ^{\Simp^{\op} \slice \simp n} \to \cats E ^{\Simp^{\op} \slice \bdsimp n}
\]
obtained by reindexing along the sieve inclusion: $\Simp^{\op} \slice \bdsimp n \to \Simp^\op \slice \simp n$, hence its right adjoint, if it exists, is the right Kan extension along this sieve inclusion.
So if we prove that the pointwise right Kan extension along this sieve inclusion exists, it will coincide with $i_*$.
If $\mathcal{F} \in \cats E \slice \bdsimp n$, then this pointwise right Kan extension evaluated at $\simp k \to \simp n \in \Simp \slice \simp n $ is given by the limit
\[(
 i_* \mathcal{F})([k]) = \lim_{p \in P} \mathcal{F}(p) \text{,} \quad \text{ where } P =\left\lbrace\begin{tikzcd}
\simp a \ar[r] \ar[dr,"p"description] & \simp k \ar[d] \\
& \simp n \rlap{,}
  \end{tikzcd} \text{$p$ not surjective}\right\rbrace \rlap{.}
  \]
This is a limit over an infinite category so it is not guaranteed to exists, but the category $P$ has a finite reflective category given by the objects such that the map $\simp a \rightarrow \simp k$ is injective, with the reflection given by the image factorisation of this map, and hence this limit coincides with
\[(
 i_* \mathcal{F})([k]) = \lim_{p \in P^+} \mathcal{F}(p) \text{,} \quad \text{ where }  P^+ =\left\lbrace\begin{tikzcd}
\simp a \ar[r,hook] \ar[dr,"p"description] & \simp k \ar[d] \\
& \simp n \rlap{,}
  \end{tikzcd} \text{$p$ not surjective}\right\rbrace \rlap{,}
  \]
which is a finite limit, hence exists, which proves the existence of $i_*$.

Next, we assume that $\mathcal{F}$ is cofibrant, and we will show that $i_* \mathcal{F}$ is cofibrant.
That is, given a degeneracy $[k] \sto [k']$ the action $i_* \mathcal{F} ([k']) \rightarrow i_* \mathcal{F}( [k])$ is a complemented inclusion (by \cref{lem:cof_charac}).
The map $i_* \mathcal{F}([k]) \rightarrow \simp n ([k])$ gives a decomposition of the map above into a coproduct indexed by all the map $\alpha \from [k] \to [n]$, so it is enough to show that the fiber above each such map is a complemented inclusion. The fiber over such a map $\alpha$ of $i_* \mathcal{F}( [k])$, is by definition of $i_*$ the object classifying maps $P \rightarrow \mathcal{F}$ over $\bdsimp n$ where $P$ is the pullback square
\begin{tikzeq*}
  \matrix[diagram,column sep={6em,between origins}]
  {
    |(P)|  P          & |(k)| \simp{k} \\
    |(bd)| \bdsimp{n} & |(n)| \simp{n} \rlap{\text{.}} \\
  };

  \draw[->] (P) to (k);
  \draw[->] (bd) to (n);
  \draw[->] (P) to (bd);
  \draw[->] (k) to node[right] {$\alpha$} (n);
  \pb{P}{n};
\end{tikzeq*}
The fiber of $i_* \mathcal{F} ([k'])$ over $\alpha$ is described similarly with $P'$ the pullback of $\simp{k'} \rightarrow \simp n$, and the map we are interested in is induced by the map $P' \rightarrow P$ obtained as the pullback of $\simp{k'} \rightarrow \simp k$. But it follows from~\cite{H}*{Proposition 3.1.11} that a pullback of a degeneracy operator is an iterated pushout of degeneracy operators, in this case a finite such iterated pushout as $P'$ is finite.
As $\mathcal{F}$ is cofibrant, this decomposes $\mathcal{F}(P) \rightarrow \mathcal{F}(P')$ as a composite of complemented inclusions, and hence concludes the proof.
\end{proof}

\begin{proof}[Proof of \cref{th:Dependent_prod_along_cof}]
We show that all cofibrations with cofibrant domain are in $\mathcal{G}$.
By \cref{thm:cell-pres-cof}, it suffices to show that the generating cofibrations are in $\mathcal{G}$ and that $\mathcal{G}$ is closed under operations appearing in a cell complex.
The case of generators is \cref{lem:generating_cof_in_G}.
Closure under tensoring by objects of $\cat{E}$ is \cref{prop:tensor_in_G}, closure under pushout (along maps with cofibrant target) is \cref{prop:Pushout_in_G}, and closure under sequential composition is \cref{prop:trans_comp_in_G}.
\end{proof}

An analysis of the proof of \cref{th:Dependent_prod_along_cof} shows that the assumption that $A$ is cofibrant is not needed for the exponentiability of $i$, as it is only used for the part of the argument regarding  preservation of cofibrant objects by $i_*$.

  \section{The Frobenius property} \label{sec:frobenius}

We adapt the notion of a strong homotopy equivalence and the associated concepts from~\cite{Gambino-Sattler}*{Section~3} to our setting.
Recall that a map $f \from A \to B$ is a 0-oriented (respectively, \emph{1-oriented}) homotopy equivalence if there is a map $g \from B \to A$ with homotopies $u \from g f \sim \id_A$ and $v \from f g \sim \id_B$ (respectively, $u \from \id_A \sim g f$ and $v \from id_B \sim f g$).
Such a homotopy equivalence is called \emph{strong} if the homotopies satisfy the coherence condition $f u = v f$.

We recall the abstract characterisation of strong homotopy equivalences.
The commuting square
\begin{tikzeq*}
\matrix[diagram]
{
  |(e)| \emptyset  & |(0)| \braces{0} \\
  |(1)| \braces{1} & |(s)| \simp{1}   \\
};

\draw[->] (e) to node[above] {$!$}        (0);
\draw[->] (e) to node[left]  {$!$}        (1);
\draw[->] (0) to node[right] {$\lambda^0_1$} (s);
\draw[->] (1) to node[below] {$\lambda^1_1$} (s);
\end{tikzeq*}
induces maps $\theta_0 \from ! \to \lambda^0_1$ and $\theta_1 \from ! \to \lambda^1_1$ in the arrow category of $\sSet$.
(We will use $\lambda^i_k$ to denote the horn inclusion $\horn{k,i} \to \simp{k}$.)
Note that $!$ is the unit of the pushout tensor and pullback cotensor of the enrichment of $\cats{E}$ in $\sSet$.
Recall that pushout tensors with levelwise complemented inclusions between finite simplicial sets such as $!, \lambda^0_1, \lambda^0_1$ exist by \cref{pushout-tensor}.

\begin{lemma} \label{she-as-retract} Let $f \from X \to Y$ be a map in $\cats{E}$.
For $k \in \braces{0, 1}$, the following are equivalent:
  \begin{conditions}
  \item \label{she-as-retract:original} $f$ is a $k$-oriented strong homotopy equivalence,
  \item \label{she-as-retract:split-mono} $\theta_k \hatop{\tensorsSetsE} f \from f \to \lambda^k_1 \hatop{\tensorsSetsE} f$ is a split monomorphism,
  \item \label{she-as-retract:split-epi} $\theta_k \hatop{\cotensorsSetsE} f \from \lambda^k_1 \hatop{\cotensorsSetsE} f \to f$ is a split epimorphism.
  \end{conditions}
\end{lemma}

\begin{proof}
  Identical to~\cite{Gambino-Sattler}*{Lemma~4.3} and~\cite{GSS}*{Lemma~3.1.1}.
\end{proof}

\begin{corollary} \label{tensor-with-she}
Let $i$ be a levelwise complemented inclusion between finite simplicial sets that is a strong homotopy equivalence.
For any map $f$ in $\cats{E}$, the pushout tensor $i \hatop{\tensorsSetsE} f$ is a strong homotopy equivalence in $\cats{E}$.
\end{corollary}

\begin{proof}
This is a formal consequence of the characterisation~\ref{she-as-retract:split-mono} of strong homotopy equivalences given by \cref{she-as-retract}.
We have $\theta_k \hatop{\tensorsSetsE} (i \hatop{\tensorsSetsE} f) \iso (\theta_k \hatop{\times} i) \hatop{\tensorsSetsE} f$, a formal consequence of the isomorphism $A \tensorsSetsE (B \tensorsSetsE X) \iso (A \times B) \tensorsSetsE X$ natural in $A, B \in \sSet$ and $X \in \cats{E}$.
By assumption, $\theta_k \hatop{\times} i$ has a retraction, hence also its image under $(-) \hatop{\tensorsSetsE} f$.
\end{proof}

Strong homotopy equivalences can be used to relate cofibrations and trivial cofibrations.

\begin{corollary} \label{she-mediating-between-cofibrations-and-trivial-cofibrations} \leavevmode
\begin{parts}
\item \label{she-mediating-between-cofibrations-and-trivial-cofibrations:generators}
For a horn inclusion $j \in J_\sSet$ and $E \in \cat{E}$, the map $j \tensorsSetsE E$ is a strong homotopy equivalence and cofibration between cofibrant objects.
\item \label{she-mediating-between-cofibrations-and-trivial-cofibrations:trick}
Any cofibration that is a strong homotopy equivalence is a trivial cofibration.
\end{parts}
\end{corollary}

\begin{proof}
For \cref{she-mediating-between-cofibrations-and-trivial-cofibrations:generators}, recall from~\cite{GZ}*{Chapter~IV, Section~2, Paragraph~2.1.3} that the horn inclusion $j$ in $\sSet$ is a strong homotopy equivalence.
By \cref{tensor-with-she}, it follows that $j \tensorsSetsE E$ is a strong homotopy equivalence.
The object $E \in \cats{E}$ is cofibrant by \cref{constant-cofibrant} of \cref{thm:pullback-combined}.
By \cref{pushout-tensor}, it follows that $j \tensorsSetsE E$ is a cofibration between cofibrant objects.

\Cref{she-mediating-between-cofibrations-and-trivial-cofibrations:trick} follows from the characterisation of strong homotopy equivalences in \cref{she-as-retract:split-mono} of \cref{she-as-retract}, closure of trivial cofibrations under retracts (\cref{saturation}), and \cref{pushout-tensor} (using that $\lambda^0_1$ and $\lambda^1_1$ are trivial cofibrations).
\end{proof}

\begin{lemma} \label{she-pullback-square}
  Let
  \begin{tikzeq*}
  \matrix[diagram]
  {
    |(B)| B & |(A)| A \\
    |(X)| X & |(Y)| Y \\
  };

  \draw[->] (B) to node[left]  {$g$} (X);
  \draw[->] (A) to node[right] {$f$} (Y);

  \draw[->]  (B) to (A);
  \draw[fib] (X) to (Y);
  \pb{B}{Y};
  \end{tikzeq*}
  be a pullback square  with $X$ cofibrant. If, $f$ is a $k$-oriented strong homotopy equivalence, where $k \in \braces{ 0, 1}$, then so
  is~$g$.
\end{lemma}

\begin{proof}
This is identical to~\cite{GSS}*{Lemma~3.1.3}, but played out in $\cats{E}_\cof$ instead of $\sSet_\cof$.
The pushout product with $\braces{1} \to \simp{1}$ (for $k = 0$) becomes a pushout tensor, which sends the cofibration $\varnothing \to X$ to a trivial cofibration by \cref{pushout-tensor}.
\end{proof}

\begin{corollary} \label{she-stable-under-pullback}
Let $f \from X \fto Y$ be a Kan fibration with $X$ cofibrant.
The pullback functor $f^* \from \cat{E} \slice Y \to \cat{E} \slice X$ preserves maps that in $\cats{E}$ are strong homotopy equivalences with cofibrant target.
\end{corollary}

\begin{proof}
This follows from \cref{she-pullback-square} using \cref{item:retract_pb_com} of \cref{lem:basic_pointwise_fibrations} and stability of cofibrant objects under pullback along maps with cofibrant source using \cref{pullback-cofibrant} of \cref{thm:pullback-combined}.
\end{proof}

\begin{proposition}[Frobenius property] \label{frobenius}
Let $f \from X \fto Y$ be a Kan fibration with $X$ cofibrant.
The pullback functor $f^* \from \cat{E} \slice Y \to \cat{E} \slice X$ preserves trivial cofibrations.
\end{proposition}

\begin{proof}
Let $j$ be a trivial cofibration over $Y$.
By \cref{cofibration-as-retract-of-cell-complex}, its underlying map in $\cats{E}$ can be written as a retract of a $J_{\cats{E}}$-cell complex $j'$.
The retraction (including $j'$) lifts uniquely to the slice over $Y$.
Since functors preserve retracts, this makes $f^* j$ a retract of $f^* j'$.
By \cref{saturation}, it will thus suffice to show that $f^* j'$ is a trivial cofibration.

Recall that $J_{\cats{E}}$ consists of levelwise complemented inclusions.
By countable lextensivity, \cref{ldi-saturation}, and \cref{colimits-diagram-category}, the pullback functor $f^*$ preserves the colimits (countable coproducts, pushouts, sequential colimit) forming the cell complex $j'$.
By \cref{saturation}, it thus remains to show that $f^*$ sends to a trivial cofibration any map that in $\cats{E}$ is of the form $E \tensorEsE \fset{j''}$ where $E \in \cats{E}$ and $j'' \in J_\sSet$.
Using \cref{sSet-tensor}, this simplifies to $j'' \tensorsSetsE E$.
Here, we see $E$ as a constant simplicial object in $\cat{E}$.

By \cref{she-mediating-between-cofibrations-and-trivial-cofibrations:generators} of \cref{she-mediating-between-cofibrations-and-trivial-cofibrations}, $j'' \tensorsSetsE E$ is a strong homotopy equivalence and cofibration between cofibrant objects.
By \cref{she-stable-under-pullback}, $f^* (j'' \tensorsSetsE E)$ is a strong homotopy equivalence (using that $f$ is a Kan fibration).
By \cref{pullback-cofibration}, $f^* (j'' \tensorsSetsE E)$ is a cofibration between cofibrant objects.
By \cref{she-mediating-between-cofibrations-and-trivial-cofibrations:trick} of \cref{she-mediating-between-cofibrations-and-trivial-cofibrations}, we conclude that $f^* (j'' \tensorsSetsE E)$ is a trivial cofibration.
\end{proof}

  \section{Fibration extension properties} \label{sec:equivalence-extension}

In this section, we establish two important ingredients in the construction of the effective model structure:
the trivial fibration extension property (\cref{trivial-fibration-extension}) and the fibration extension property (\cref{fibration-extension}).
These arguments are based on the equivalence extension property (\cref{equivalence-extension}).
We work purely within the cofibrant fragment $\cats{E}_\cof$ of $\cats{E}$.
Our earlier preliminaries allow us to prove the equivalence extension property in $\cats{E}_\cof$ following~\cite{Sattler}*{Proposition~5.1} and~\cite{GSS}*{Proposition~3.2.1}.

We begin with some observations on homotopy equivalences, which we introduced in \cref{sec:fib_cat},
and an analysis of the restriction of the fibration category structure on $\cats{E} \fslice X$ established in \cref{fibcat-fiberwise} to cofibrant objects.
Since the tensor of $X \in \cats E$ with a finite simplicial set exists and is defined by the formula in \eqref{equ:tensorSetE},
we may equivalently write a homotopy $H$ between $f_0, f_1 \from X \to Y$ in $\cats{E}$ or one of its slices,
which was defined using cotensors in \eqref{equ:homotopy}, via a map
\begin{equation}
\label{equ:homotopy-with-tensor}
H \from \simp{1} \tensorsSetsE X \to Y \text{.}
\end{equation}
In $\cat{E}$ and its slices, the homotopy relation between maps with cofibrant source and fibrant target is an equivalence relation.
This is a formal consequence of \cref{item:pbcot} of \cref{lem:basic_pointwise_fibrations} and \cref{pullback-cotensor-slice}.
It follows that homotopy equivalences between cofibrant and fibrant objects compose as usual.

\begin{proposition} \label{triv-map-is-h-equiv} \leavevmode
\begin{parts}
\item \label{triv-map-is-h-equiv:triv-cof}
For every $X \in \cats{E}$, trivial cofibrations in $\cats{E} \fslice X$ are homotopy equivalences.
\item \label{triv-map-is-h-equiv:triv-fib}
Trivial fibrations $X \to Y$ in $\cats{E}_\cof$ are homotopy equivalences over $Y$.
\end{parts}
\end{proposition}

\begin{proof}
For \cref{triv-map-is-h-equiv:triv-cof}, in $\cat{E} \fslice X$, given a trivial cofibration $A \to B$, we take a lift
\begin{tikzeq*}
\matrix[diagram,column sep={6em,between origins}]
{
  |(A)| A & |(X)| A \times_B (\simp{1} \cotensorsSetsE B) \\
  |(B)| B & |(Y)| B \rlap{\text{.}}                       \\
};

\draw[ano] (A) to (B);
\draw[fib] (X) to (Y);
\draw[->] (A) to (X);
\draw[->] (B) to (Y);
\draw[->,dashed] (B) to (X);
\end{tikzeq*}
Here, the right map is a composition of the pullback cotensor with $\partial \simp{1} \to \simp{1}$ of $B \to 1$ and a pullback of $A \to 1$, hence a fibration by parts~\ref{item:pbcot} and~\ref{item:retract_pb_com} of \cref{lem:basic_pointwise_fibrations}.
The lift exhibits $A \to B$ as a strong deformation retract, in particular a homotopy equivalence.

For \cref{triv-map-is-h-equiv:triv-fib}, given a fibration $X \to Y$ in $\cat{E}_\cof$, we take a lift
\begin{tikzeq*}
\matrix[diagram,column sep={6em,between origins}]
{
  |(A)| X                                    & |(X)| X                 \\
  |(B)| Y \push_X (\simp{1} \tensorsSetsE X) & |(Y)| Y \rlap{\text{.}} \\
};

\draw[cof]  (A) to (B);
\draw[tfib] (X) to (Y);
\draw[->] (A) to (X);
\draw[->] (B) to (Y);
\draw[->,dashed] (B) to (X);
\end{tikzeq*}
Here, the left map is a composition of a pushout of $\varnothing \to Y$ and the pushout tensor with $\partial \simp{1} \to \simp{1}$ of $\varnothing \to X$, hence a cofibration by \cref{saturation,pushout-tensor}.
The lift exhibits $X \to Y$ as the dual of a strong deformation retract, in particular a homotopy equivalence over $Y$.
\end{proof}

\begin{proposition} \label{fibcat-cof-fiberwise}
Let $X \in \cats{E}_\cof$.
The fibration category structure on $\cats{E} \fslice X$ of \cref{fibcat-fiberwise} restricts to $\cats{E}_\cof \fslice X$.
Path objects are given by cotensor with $\simp{1}$.
The weak equivalences coincide with homotopy equivalences over $X$.
\end{proposition}

\begin{proof}
By \cref{cofibrant-limit} of \cref{thm:pullback-combined}, $\cats{E}_\cof \fslice X$ has finite limits and they are computed as in $\cats{E} \slice X$.
By \cref{cofibrant-cotensor} of \cref{cofibrant-properties}, cotensor with $\simp{1}$ over $X$ preserves cofibrant objects.
Thus, all aspects of the fibration category $\cats{E} \fslice X$ of \cref{fibcat-fiberwise} restrict to cofibrant objects.
This includes path objects, which are given by cotensor with $\simp{1}$.

It remains to show that pointwise weak equivalences in $\cats{E}_\cof \fslice X$ coincide with homotopy equivalences over $X$.
Every homotopy equivalence is a pointwise weak equivalence by \cref{h-equiv-is-we}.
For the reverse direction, we use the mapping path space factorisation in $\cats{E}_\cof \fslice X$, which has a homotopy equivalence over $X$ as first factor and fibration as second factor.
Since pointwise weak equivalences and homotopy equivalences over $X$ satisfy the 2-out-of-3 property, it suffices to show that every pointwise weak equivalence that is a fibration (hence a trivial fibration) is a homotopy equivalence over $X$.
This is \cref{triv-map-is-h-equiv:triv-fib} of \cref{triv-map-is-h-equiv}.
\end{proof}

\begin{proposition}[Equivalence extension property]\label{equivalence-extension}
  In $\cats{E}_\cof$, consider the solid part of the diagram
  \begin{tikzeq}{equivalence-extension:0}
  \matrix[diagram,column sep={between origins,4em},row sep={between origins,8ex}]
  {
    |(X0)| X_0 &            & |(Y0)| Y_0 &            \\[-2ex]
               & |(X1)| X_1 &            & |(Y1)| Y_1 \\
    |(A)|  A   &            & |(B)|  B   &            \\
  };

  \draw[->]        (X0) to node[above right] {$\sim$} (X1);
  \draw[->,dashed] (Y0) to node[above right] {$\sim$} (Y1);
  \draw[cof]       (A)  to node[below]       {$i$}    (B);

  \draw[fib] (X0) to (A);
  \draw[fib] (X1) to (A);

  \draw[fib,dashed] (Y0) to (B);
  \draw[fib]        (Y1) to (B);

  \draw[->,dashed] (X0) to (Y0);
  \draw[->,over]   (X1) to (Y1);

  \pbs{X1}{B};
  \end{tikzeq}
  where the lower square is a pullback and $X_0 \to X_1$ is a homotopy equivalence over $A$.
  Then there is $Y_0$ as indicated such that the back square is a pullback and $Y_0 \to Y_1$ is a homotopy equivalence over $B$.
\end{proposition}

\begin{proof}
The proof of~\cite{GSS}*{Proposition~3.2.1} applies, but played out in $\cats{E}_\cof$ instead of $\sSet_\cof$.
We limit ourselves to listing the key claims used in the proof and why they hold in our setting.
\begin{itemize}
\item
The slice categories $\cats{E}_\cof \fslice A$ and $\cats{E}_\cof \fslice B$ admit
fibration category structures, established in \cref{fibcat-cof-fiberwise}, in which
weak equivalences are given by fiberwise homotopy equivalences.
\item
The dependent product functor $i_*$ along $i$ exists and preserves cofibrant objects, as shown in \cref{th:Dependent_prod_along_cof}.
\item
The functor $i_*$ preserves trivial fibrations, which follows by adjointness since $i^*$ preserves cofibrations, as stated in \cref{pullback-cofibration} of \cref{thm:pullback-combined}.
\item
In the slice over $B$, pullback cotensor with a cofibration preserves trivial fibrations, which holds by  \cref{pullback-cotensor-slice}.
\qedhere
\end{itemize}
\end{proof}

In $\cats{E}_\cof$, we say that a (trivial) fibration $X \fto A$ \emph{extends} along a map $A \to B$ if there is a pullback square
\begin{tikzeq}{fibration-extension-square}
\matrix[diagram]
{
  |(X)| X & |(Y)| Y \\
  |(A)| A & |(B)| B \\
};

\draw[->,dashed]  (X) to (Y);
\draw[->]         (A) to (B);
\draw[fib]        (X) to (A);
\draw[fib,dashed] (Y) to (B);
\pb{X}{B};
\end{tikzeq}
 with the \emph{extension} $Y \to B$ of $X \to A$ again a (trivial) fibration.
If $A \to B$ has this property for all (trivial) fibrations $X \fto A$, we say that it has the \emph{(trivial) fibration extension property}.

\begin{lemma} \label{extension-cancellation}
Let $f$ and $g$ be composable maps in $\cats{E}_\cof$.
If $g \circ f$ has the (trivial) fibration extension property, then so does $f$.
\end{lemma}

\begin{proof}
We extend along $f$ by extending along $g \circ f$ and pulling back along $g$ (using \cref{item:retract_pb_com} of \cref{lem:basic_pointwise_fibrations} and \cref{pullback-cofibrant} of \cref{thm:pullback-combined}).
\end{proof}

\begin{proposition}[Trivial fibration extension property] \label{trivial-fibration-extension}
Cofibrations in $\cats{E}$ have the trivial fibration extension property.
\end{proposition}

\begin{proof}
This is the special case of \cref{equivalence-extension} where $X_1 \to A$ and $Y_1 \to B$ are the identities on $A$ and $B$, respectively.
We use \cref{fibcat-fiberwise,fibration-levelwise} to go between trivial fibrations and fibrations that are weak equivalences.
\end{proof}

\begin{lemma} \label{weak-equivalence-from-fibration-over-line}
Let $p \from X \fto \simp{1} \tensorsSetsE A$ be fibration in $\cats{E}$ with $A$ and $X$ cofibrant.
Then there is a homotopy equivalence between $X|_{\braces{0} \tensorsSetsE A}$ and $X|_{\braces{1} \tensorsSetsE A}$ over $A$.
\end{lemma}

\begin{proof}
Take the pullback
\begin{tikzeq*}{fibration-extension-square}
\matrix[diagram,column sep={8em,between origins}]
{
  |(P)| P & |(X')| \simp{1} \cotensorsSetsE X                                          \\
  |(A)| A & |(A')| \simp{1} \cotensorsSetsE (\simp{1} \tensorsSetsE A) \rlap{\text{.}} \\
};

\draw[->]  (P) to (X');
\draw[->]  (A) to (A');
\draw[fib] (P) to (A);
\draw[fib] (X') to node[right] {$\simp{1} \cotensorsSetsEslice p$} (A');
\pb{P}{A'};
\end{tikzeq*}
Here, the bottom map is the unit of the tensor-cotensor adjunction.
The right map is a fibration by \cref{item:pbcot} of \cref{lem:basic_pointwise_fibrations}, hence the left map is a fibration by \cref{item:retract_pb_com} of \cref{lem:basic_pointwise_fibrations}.
The top right object is cofibrant is cofibrant by \cref{item:pbcot} of \cref{lem:basic_pointwise_fibrations} and \cref{pullback-cofibrant}, hence the top left object is cofibrant by \cref{pullback-cofibrant}.

We will argue that there are trivial fibrations from $P$ to $X|_{\braces{0} \tensorsSetsE A}$ and $X|_{\braces{1} \tensorsSetsE A}$ over $A$.
These trivial fibrations are homotopy equivalences over $A$ by \cref{triv-map-is-h-equiv:triv-fib} of \cref{triv-map-is-h-equiv}.
Inverting and composing them as needed gives the desired weak equivalence.

We only construct the trivial fibration from $P$ to $X|_{\braces{0} \tensorsSetsE A}$ (the other case is dual).
Consider the diagram
\begin{tikzeq*}{fibration-extension-square}
\matrix[diagram,column sep={12em,between origins}]
{
  |(P)| P                                & |(X')| \simp{1} \cotensorsSetsEslice X                                                                                                                  \\
  |(X0)| X|_{\braces{0} \tensorsSetsE A} & |(X0')| X \times_{\simp{1} \tensorsSetsE A} \simp{1} \cotensorsSetsEslice (\simp{1} \tensorsSetsE A) & |(X)| X                                          \\
  |(A)| A                                & |(A')| \simp{1} \cotensorsSetsEslice (\simp{1} \tensorsSetsE A)                                      & |(A'')| \simp{1} \tensorsSetsE A \rlap{\text{.}} \\
};

\draw[->]  (P) to (X');
\draw[->]  (X0) to (X0');
\draw[->]  (A) to (A');
\draw[->]  (X0') to (X);
\draw[->]  (A') to node[above] {$\lambda^0_1 \cotensorsSetsEslice (\simp{1} \tensorsSetsE A)$} (A'');
\draw[tfib] (P) to (X0);
\draw[fib] (X0) to (A);
\draw[tfib] (X') to node[right] {$\lambda^0_1 \hatop{\cotensorsSetsEslice} p$} (X0');
\draw[fib] (X0') to (A');
\draw[fib] (X) to node[right] {$p$} (A'');
\pb{X0'}{A''};
\end{tikzeq*}
The two composite squares and the bottom right square are pullbacks by construction.
Pullback pasting induces the top left map and makes the top left square a pullback.
The top middle map is a trivial fibration by \cref{item:pbcot} of \cref{lem:basic_pointwise_fibrations}, hence so is the top left map by \cref{item:retract_pb_com} of \cref{lem:basic_pointwise_fibrations}.
\end{proof}

Our aim now is to prove the fibration extension property for trivial cofibrations in $\cats{E}_\cof$.
For this purpose, we introduce the class $\mathcal{H}$ of cofibrations in $\cats{E}_\cof$ that have the fibration extension property.

\begin{lemma} \label{fibration-generating-extension}
The class $\mathcal{H}$ contains cofibrations in $\cats{E}_\cof$ that are strong homotopy equivalences.
\end{lemma}

\begin{proof}
Let $A \to B$ be a cofibration in $\cats{E}_\cof$ and $0$-oriented strong homotopy equivalence (the $1$-oriented case is dual).
We will solve the extension problem~\eqref{fibration-extension-square}.
By the characterisation of strong homotopy equivalences given by part~(3) of \cref{she-as-retract}, we have a retract diagram
\begin{tikzeq}{fibration-generating-extension:1}
\matrix[diagram,column sep={11em,between origins}]
{
  |(Al)| A & |(pp)| (\simp{1} \tensorsSetsE A) \push_{\braces{0} \tensorsSetsE A} (\braces{0} \tensorsSetsE B) & |(Ar)| A                 \\
  |(Bl)| B & |(p)|   \simp{1} \tensorsSetsE B                                                                  & |(Br)| B \rlap{\text{.}} \\
};

\draw[->] (Al) to (Bl);
\draw[->] (pp) to (p);
\draw[->] (Ar) to (Br);

\draw[->] (Al) to (pp);
\draw[->] (pp) to (Ar);
\draw[->] (Bl) to node[above] {$\lambda^1_1 \tensorsSetsE B$} (p);
\draw[->] (p)  to (Br);
\end{tikzeq}
Let $Z \to \simp{1} \tensorsSetsE A \push_{\braces{0} \tensorsSetsE A} \braces{0} \tensorsSetsE B$ denote the pullback of $X \to A$ along the top right map.
Pulling back $Z$ to $\simp{1} \tensorsSetsE A$, $\braces{0} \tensorsSetsE A$ and $\braces{0} \tensorsSetsE B$ (the components of its base pushout),
we obtain the solid part of the diagram
\begin{tikzeq*}
\matrix[diagram,column sep={between origins,4em},row sep={between origins,8ex}]
{
  |(X0)| Z|_{\braces{1} \tensorsSetsE A} &                                        & |(Y0)| Y                 &                                        \\[-2ex]
                                         & |(X1)| Z|_{\braces{0} \tensorsSetsE A} &                          & |(Y1)| Z|_{\braces{0} \tensorsSetsE B} \\
  |(A)|  A                               &                                        & |(B)|  B \rlap{\text{,}} &                                        \\
};

\draw[->]        (X0) to node[above right] {$\sim$} (X1);
\draw[->,dashed] (Y0) to node[above right] {$\sim$} (Y1);
\draw[->]        (A)  to                            (B);

\draw[fib] (X0) to (A);
\draw[fib] (X1) to (A);

\draw[fib,dashed] (Y0) to (B);
\draw[fib]        (Y1) to (B);

\draw[->,dashed] (X0) to (Y0);
\draw[->,over]   (X1) to (Y1);

\pbs{X1}{B};
\end{tikzeq*}
with lower square a pullback.
Here, the weak equivalences over $A$ is given by \cref{weak-equivalence-from-fibration-over-line}.
We then complete the diagram using \cref{equivalence-extension}, making the back square a pullback.
Note that $Z|_{\braces{1} \tensorsSetsE A}$ is isomorphic to $X$ over $A$ by the retract~\eqref{fibration-generating-extension:1}.
The extension in~\eqref{fibration-extension-square} is then given by $Y \fto B$.
\end{proof}

\begin{corollary} \label{fibration-extension-generators}
For a horn inclusion $j \in J_\sSet$ and $E \in \cat{E}$, we have $j \tensorsSetsE E \in \mathcal{H}$.
\end{corollary}

\begin{proof}
This is the application of \cref{fibration-generating-extension} to \cref{she-mediating-between-cofibrations-and-trivial-cofibrations:generators} of \cref{she-mediating-between-cofibrations-and-trivial-cofibrations}.
\end{proof}

\begin{lemma} \label{fibration-extension-coproduct}
The class $\mathcal{H}$ is closed under countable coproducts.
\end{lemma}

\begin{proof}
Let $A_i \to B_i$ be a family of maps in $\mathcal{H}$ for $i \in I$ countable.
Note that $\bigcoprod_{i \in I} A_i \to \bigcoprod_{i \in I} B_i$ is a cofibration between cofibrant objects by \cref{saturation}.
Suppose we are given a fibration $X \to \bigcoprod_{i \in I} A_i$ in $\cats{E}_\cof$.
We aim to extend it along $\bigcoprod_{i \in I} A_i \to \bigcoprod_{i \in I} B_i$.
Note that $\bigcoprod_{i \in I} B_i$ is a van Kampen colimit since $\cats{E}$ is countably lextensive.

For each $i \in I$, we pull it back to a fibration $X_i \to A_i$ (with $X_i$ cofibrant by \cref{pullback-cofibrant}) and extend it to a fibration $Y_i \to B_i$.
We take their coproduct $\bigcoprod_{i \in I} Y_i \to \bigcoprod_{i \in I} B_i$.
This is a fibration by \cref{colimit-of-fibrations:coproduct} of \cref{colimit-of-fibrations}.
Its domain is cofibrant by \cref{saturation}.
By effectivity, it pulls back along $A_i \to \bigcoprod_{i \in I} B_i$ to the map $X_i \to A_i$ for $i \in I$.
By universality, it thus pulls back along $\bigcoprod_{i \in I} A_i \to \bigcoprod_{i \in I} B_i$ to the original fibration $X \to \bigcoprod_{i \in I} A_i$.
\end{proof}

\begin{lemma} \label{fibration-extension-pushout}
The class $\mathcal{H}$ is closed under pushouts in $\cats{E}$ along maps with cofibrant target.
\end{lemma}

\begin{proof}
Consider a pushout square
\begin{tikzeq*}
\matrix[diagram]
{
  |(A)| A & |(A')| A'                 \\
  |(B)| B & |(B')| B' \rlap{\text{.}} \\
};

\draw[->] (A) to node[right] {$\in \mathcal{H}$} (B);
\draw[->] (A') to                                (B');

\draw[->] (A) to (A');
\draw[->] (B) to (B');
\pbdr{B'}{A};
\end{tikzeq*}
with $A'$ cofibrant.
Note that $A' \to B'$ is a cofibration between cofibrant objects by \cref{saturation}.
The pushout is van Kampen by part~(i) of \cref{colimits-diagram-category}.
Suppose we are given a fibration $X' \fto A'$ in $\cats{E}_\cof$.
We aim to extend it along $A' \to B'$.

We pull the given fibration back along $A \to A'$ to a fibration $X \fto A$ (here, $X$ is cofibrant by \cref{pullback-cofibrant}) and extend it to a fibration $Y \fto B$.
Let $Y' \to B'$ be the pushout in the arrow category of these three maps.
By effectivity, it pulls back to them.
It is a fibration by \cref{colimit-of-fibrations:pushout} of \cref{colimit-of-fibrations}.
By \cref{pullback-cofibration}, $X \to Y$ is a cofibration, hence so is $X' \to Y'$ by \cref{saturation}.
This makes $Y'$ cofibrant.

We check that $Y' \to B'$ is a fibration using \cref{enriched-lifting-as-pullback-evaluation}.
For each horn inclusion $j \in J_\sSet$, we construct a section of $\hat{\ev}_j(Y' \to B')$ given sections of $\hat{\ev}_j(X' \to A')$ and $\hat{\ev}_j(Y \to B)$.
We pull the section of $\hat{\ev}_j(X' \to A')$ back to a section of $\hat{\ev}_j(X \to A)$ and then extend it using \cref{structured-fibration-extension} to a section of $\hat{\ev}_j(Y \to B)$.
The goal follows by \cref{van-Kampen-pullback-weighted-limit} and functoriality of colimits.
\end{proof}

\begin{lemma} \label{fibration-extension-sequential-colimit}
The class $\mathcal{H}$ is closed under sequential colimits.
\end{lemma}

\begin{proof}
Consider the colimit $B$ of a sequential diagram
\begin{tikzeq*}
\matrix[diagram]
{
  |(A0)| A_0 & |(A1)| A_1 & |(A2)| \ldots \rlap{\text{.}} \\
};

\draw[->] (A0) to node[above] {$\in \mathcal{H}$} (A1);
\draw[->] (A1) to node[above] {$\in \mathcal{H}$} (A2);
\end{tikzeq*}
Note that it is van Kampen by part~(ii) of \cref{colimits-diagram-category}.
Suppose we are given a fibration $X_0 \fto A_0$ in $\cats{E}_\cof$.
We aim to extend it along $A_0 \to B$.

By induction on $k$, we extend to a fibration $X_k \fto A_k$.
The maps $X_k \to X_{k+1}$ are cofibrations by \cref{pullback-cofibration}.
In the end, we take the colimit and obtain a map $Y \to B$.
By effectivity, it pulls back to the maps $X_k \fto A_k$.
It is a fibration by \cref{colimit-of-fibrations:sequential-composition} of \cref{colimit-of-fibrations}.
Note that $Y$ is cofibrant by \cref{saturation}.
\end{proof}

\begin{lemma} \label{fibration-extension-retract}
The class $\mathcal{H}$ is closed under codomain retracts.
\end{lemma}

\begin{proof}
This is an instance of \cref{extension-cancellation}.
\end{proof}

\begin{proposition}[Fibration extension property] \label{fibration-extension}
Trivial cofibrations in $\cats{E}_\cof$ have the fibration extension property.
\end{proposition}

\begin{proof}
We have to show that $\mathcal{H}$ includes all trivial cofibrations between cofibrant objects.
By \cref{cofibration-as-retract-of-cell-complex}, any such trivial cofibration can be written as a codomain retract of a sequential colimit of pushouts of countable coproducts of tensors with objects of $E$ of maps in $J_{\cats{E}}$.
By induction, all the stages of the sequential colimit are cofibrant.
This means that the above pushout squares all consist of cofibrant objects.
The claim now follows starting from \cref{fibration-extension-generators} using the closure properties of $\mathcal{H}$ given by \cref{fibration-extension-coproduct,fibration-extension-pushout,fibration-extension-sequential-colimit,fibration-extension-retract}.
\end{proof}

   \section{The effective model structure} \label{sec:model-structure}

The main goal of this section is to establish the existence of the effective model structure. Since the categories with which we work have finite limits but do not
necessarily have finite colimits, it is appropriate to consider a slight generalisation of the usual notion of a model structure. For a category $\cat{E}$  with an initial object and a terminal object, a \emph{model structure} on $\cat E$ consists of three classes of maps $\mathbf{W}$, $\mathbf{C}$, $\mathbf{F}$ such that
\begin{itemize}
\item  $(\mathbf{C}, \mathbf{F} \cap \mathbf{W})$ and $(\mathbf{C} \cap \mathbf{W}, \mathbf{F})$ are \wfs{}s;
\item $\mathbf{W}$ satisfies the 2-out-of-3 property;
\item $\cat{E}$ has pushouts along maps in $\mathbf{C}$;
\item $\cat{E}$ has pullbacks along maps in $\mathbf{F}$.
\end{itemize}
It can then be shown that $\mathbf{W}$ is closed under retracts, as the known proof of this fact (see~\cite{Joyal-Tierney}*{Proposition~7.8} and~\cite{RiehlE:catht}*{Lemma~11.3.3}) applies also assuming only the restricted limits and colimits above. Thus, when $\cat E$ is finitely complete and cocomplete, this notion is equivalent to the usual one. Similarly,  a model structure is determined by two of its three classes of maps also in this setting.

Let us now fix a countably lextensive category $\cat E$.
The existence of the effective model structure on $\cats{E}$ will be a formal consequence of the Frobenius property of \cref{sec:frobenius}, the (trivial) fibration extension property of \cref{sec:equivalence-extension}, and elementary properties of the two \wfs{}s of \cref{two-ewfss}.
To this end, we encapsulate what is used from \cref{sec:equivalence-extension} as a collection of extension operations that all follow the same pattern.

\begin{lemma} \label{extension-squares}
The following hold in $\cats{E}_\cof$.
\begin{parts}
\item \label{extension-squares:cof-triv-fib}
Let $A \to B$ be a cofibration and $X \to A$ be a trivial fibration.
There is a pullback square
\begin{tikzeq*}
\matrix[diagram]
{
  |(X)| X & |(Y)| Y                 \\
  |(A)| A & |(B)| B \rlap{\text{.}} \\
};

\draw[cof,dashed]  (X) to (Y);
\draw[cof]         (A) to (B);
\draw[tfib]        (X) to (A);
\draw[tfib,dashed] (Y) to (B);
\pb{X}{B};
\end{tikzeq*}
with $X \to Y$ a cofibration and $Y \to B$ a trivial fibration.
\item \label{extension-squares:triv-cof-fib}
Let $A \to B$ be a trivial cofibration and $X \to A$ be a fibration.
There is a pullback square
\begin{tikzeq*}
\matrix[diagram]
{
  |(X)| X & |(Y)| Y                 \\
  |(A)| A & |(B)| B \rlap{\text{.}} \\
};

\draw[ano,strike] (X) to (Y);
\draw[ano]        (A) to (B);
\draw[fib]        (X) to (A);
\draw[fib,dashed] (Y) to (B);
\pb{X}{B};
\end{tikzeq*}
with $X \to Y$ a trivial cofibration and $Y \to B$ a fibration.
\item \label{extension-squares:triv-cof-triv-fib}
Let $A \to B$ be a trivial cofibration and $X \to A$ be a trivial fibration.
There is a pullback square
\begin{tikzeq*}
\matrix[diagram]
{
  |(X)| X & |(Y)| Y                 \\
  |(A)| A & |(B)| B \rlap{\text{.}} \\
};

\draw[ano,strike]  (X) to (Y);
\draw[ano]         (A) to (B);
\draw[tfib]        (X) to (A);
\draw[tfib,dashed] (Y) to (B);
\pb{X}{B};
\end{tikzeq*}
with $X \to Y$ a trivial cofibration and $Y \to B$ a trivial fibration.
\end{parts}
\end{lemma}


\begin{proof}
\Cref{extension-squares:cof-triv-fib} is the combination of \cref{trivial-fibration-extension} with \cref{pullback-cofibration} of \cref{thm:pullback-combined}.
\Cref{extension-squares:triv-cof-fib} is the combination of \cref{fibration-extension} with \cref{frobenius}.
\Cref{extension-squares:triv-cof-triv-fib} follows from \cref{extension-squares:cof-triv-fib} using \cref{frobenius} (with \cref{triv-fib-is-fib}).
\end{proof}

Recall from \cref{sec:fib_cat} that a map $X \to Y$ in $\cats{E}_\fib$ is a weak equivalence in the fibration category of \cref{fibcat-Kan} if and only if it is a
pointwise weak equivalence in the sense of \cref{pwe}, \ie, $\Hom_{\sSet}(E, X) \to \Hom_{\sSet}(E, Y)$ is a weak homotopy equivalence of simplicial sets for all $E \in \cat{E}$.
Restricting to cofibrant objects, we obtain a notion of weak equivalence in $\cats{E}_{\cof,\fib}$ that satisfies 2-out-of-3 and interacts as expected with cofibrations and fibrations, as recollected below.

\begin{lemma} \label{model-cof-fib}
In $\cats{E}_{\cof,\fib}$, we have:
\begin{parts}
\item \label{model-cof-fib:cof}
a cofibration is a trivial cofibration exactly if it is a weak equivalence,
\item \label{model-cof-fib:fib}
a fibration is a trivial fibration exactly if it is a weak equivalence.
\end{parts}
\end{lemma}

\begin{proof}
\Cref{model-cof-fib:fib} is a corollary of \cref{fibration-levelwise}.
For \cref{model-cof-fib:cof}, the forward direction is the combination of \cref{triv-map-is-h-equiv:triv-cof} of \cref{triv-map-is-h-equiv} and \cref{h-equiv-is-we}.
With this, the reverse direction follows by the retract argument.
\end{proof}

In the following, we fix the following terminology regarding the \wfs{}s of \cref{two-ewfss}.
A \emph{fibrant replacement} of $X \in \cats{E}$ is a trivial cofibration $X \to X'$ with $X'$ fibrant.
By a fibrant replacement of a diagram, we mean a levelwise fibrant replacement: given a diagram $X \from \cat{S} \to \cats{E}$, this is a diagram $X' \from \cat{S} \to \cats{E}_\fib$ with a natural transformation $X \to X'$ that is levelwise a trivial cofibration.
If $\cat{S}$ is a finite Reedy category, we can always construct such a replacement using \cref{esmo} and the Reedy process.
In particular, for $[1]$ seen as a direct category, we obtain a fibrant replacement of any arrow that we call \emph{canonical}.
Note that the canonical fibrant replacement preserves trivial cofibrations.
We use dual terminology for \emph{cofibrant replacement}.

Let us write $\mathbf{W}_\cof$ for the class of maps in $\cats{E}_\cof$ whose canonical fibrant replacement is a weak equivalence in $\cats{E}_{\cof,\fib}$.
This will be the class of weak equivalences in the model structure on~$\cats{E}_\cof$ to be established in \cref{model-cof}.

\begin{lemma} \label{we-cof-invariant} Let $A \to B$ in $\cats{E}_{\cof}$. Then, the
the following are equivalent:
\begin{conditions}
\item
the map $A \to B$ is in $\mathbf{W}_\cof$,
\item
the map $A \to B$ has a fibrant replacement that is a weak equivalence in $\cats{E}_{\cof,\fib}$,
\item
all fibrant replacements of the map $A \to B$ are weak equivalences in $\cats{E}_{\cof,\fib}$.
\end{conditions}
\end{lemma}

\begin{proof}
This is a standard argument and goes exactly as in~\cite{GSS}*{Lemma~3.3.1}.
What is used is part~(i) of \cref{colimits-diagram-category} with the fact that trivial cofibrations are levelwise complemented inclusions (\cref{cofibration-as-retract-of-cell-complex}), and closure properties of trivial cofibrations (\cref{saturation}), the forward direction of \cref{model-cof-fib:cof} of \cref{model-cof-fib}, and 2-out-of-3 for weak equivalences in $\cats{E}_{\cof,\fib}$.
\end{proof}

\begin{corollary} \label{we-cof-2-out-of-3}
The class $\mathbf{W}_\cof$ satisfies the 2-out-of-3 property.
\end{corollary}

\begin{proof}
Using \cref{we-cof-invariant} with levelwise fibrant replacement of the given 2-out-of-3 diagram, this reduces to closure of weak equivalences in $\cats{E}_{\cof,\fib}$ under 2-out-of-3.
This is part of \cref{fibcat-Kan}.
\end{proof}

\begin{lemma} \label{model-cof-fibration}
In $\cats{E}_\cof$, a fibration is a trivial fibration if and only if it is in $\mathbf{W}_\cof$.
\end{lemma}

\begin{proof}
Let $X \to Y$ be a fibration in $\cats{E}_\cof$.
Take a fibrant replacement $Y \to \overline{Y}$.

If $X \to Y$ is a trivial fibration, we extend it to a trivial fibration $\overline{X} \to \overline{Y}$ using \cref{extension-squares:triv-cof-triv-fib} of \cref{extension-squares}.
Then $\overline{X} \to \overline{Y}$ is a weak equivalence by \cref{model-cof-fib:fib} of \cref{model-cof-fib}, hence $X \to Y$ is in $\mathbf{W}_\cof$ by \cref{we-cof-invariant}.

In the reverse direction, we extend $X \to Y$ to a fibration $\overline{X} \to \overline{Y}$ using \cref{extension-squares:triv-cof-fib} of \cref{extension-squares}.
If $X \to Y$ is in $\mathbf{W}_\cof$, then $\overline{X} \to \overline{Y}$ is a weak equivalence by \cref{we-cof-invariant}, hence a trivial fibration by \cref{model-cof-fib:fib} of \cref{model-cof-fib}.
Then its pullback $X \to Y$ is a trivial fibration by \cref{item:retract_pb_com} of \cref{lem:basic_pointwise_fibrations}.
\end{proof}

\begin{proposition} \label{model-cof}
The category $\cats{E}_{\cof}$ admits a model structure with weak equivalences $\mathbf{W}_\cof$ and the two \wfs{}s of \cref{two-ewfss}.
\end{proposition}

\begin{proof}
First note that $\cats{E}_{\cof}$ has finite limits by \cref{cofibrant-limit} of \cref{thm:pullback-combined}, an initial object by lextensivity, and pushouts of cofibrations by \cref{colimits-diagram-category:pushout} of \cref{colimits-diagram-category} (since cofibrations are levelwise complemented inclusions by \cref{cofibration-as-retract-of-cell-complex}).
The class $\mathbf{W}_\cof$ satisfies 2-out-of-3 by \cref{we-cof-2-out-of-3}.

It remains to show that a (co)fibration is trivial exactly if it is a weak equivalence.
For fibrations, this is \cref{model-cof-fibration}.
For cofibrations, the forward direction is immediate using \cref{we-cof-invariant}:
a given trivial cofibration has as fibrant replacement the identity on a fibrant replacement of its codomain; but identities are weak equivalences in $\cats{E}_{\cof,\fib}$ by \cref{fibcat-Kan}.
The backward direction follows from this by the retract argument.
\end{proof}

We write $\mathbf{W}$ for the class of maps in $\cats{E}$ whose canonical cofibrant replacement is in $\mathbf{W}_\cof$. This is the
class of weak equivalences of the effective model structure, to be established in \cref{model}.

\begin{lemma} \label{we-invariant} Let $A \to B$ in $\cats{E}$. Then,
the following are equivalent:
\begin{conditions}
\item
the map $A \to B$ is in $\mathbf{W}$,
\item
the map $A \to B$ has a cofibrant replacement in $\mathbf{W}_\cof$,
\item
all cofibrant replacements of the map $A \to B$ are in $\mathbf{W}_\cof$.
\end{conditions}
\end{lemma}

\begin{proof}
This is a standard argument, dual to the one of \cref{we-cof-invariant}.
What is used is closure properties of trivial fibrations (\cref{item:retract_pb_com} of \cref{lem:basic_pointwise_fibrations}) and the model structure on $\cat{E}_\cof$ of \cref{model-cof}.
\end{proof}

\begin{corollary} \label{we-2-out-of-3}
The class $\mathbf{W}$ satisfies the 2-out-of-3 property.
\end{corollary}

\begin{proof}
This is analogous to the proof of \cref{we-cof-2-out-of-3}.
\end{proof}

We can finally establish the existence of the effective model structure on $\cats E$.

\begin{theorem}[The effective model structure] \label{model}
Let $\cat{E}$ be a countably lextensive category.
\begin{parts}
\item \label{model:two-wfss} The category $\cats{E}$ of simplicial objects in $\cat E$ admits a model structure determined by the two \wfs{}s of \cref{two-ewfss}.
\item \label{model:between-fibrant} A map between fibrant objects is a weak equivalence in this model structure  if and only if it is a pointwise weak equivalence in the sense of \cref{pwe}.
\item \label{model:between-fibrant-slice} More generally, for $X \in \cats{E}$, a map in $\cats{E} \fslice X$ is a weak equivalence exactly if and only if it is a pointwise weak equivalence in $\cats{E}$ in the sense of \cref{pwe}.
\end{parts}
\end{theorem}

\begin{proof}
First note that $\cats{E}$ has finite limits by lextensivity and the required colimits of a model structure by the same reasoning used for \cref{model-cof}.
We define the class of weak equivalences to be $\mathbf{W}$.
It satisfies 2-out-of-3 by \cref{we-cof-2-out-of-3}.
It remains to show that a (co)fibration is trivial exactly if it is a weak equivalence.

Due do our definition of $\mathbf{W}$, we get for free that every trivial fibration is a weak equivalence, dually to the reasoning for trivial cofibrations in \cref{model-cof}.

For the reverse direction, let $X \to Y$ be a fibration and weak equivalence.
Let $\hat{X} \to \hat{Y}$ denote its canonical cofibrant replacement.
This is the Reedy cofibrant replacement over the inverse category~$[1]$, hence again a fibration.
Since $\hat{X} \to \hat{Y}$ is a fibration and weak equivalence in $\cat{E}_\cof$, it is a trivial fibration by \cref{model-cof}.
The composite $\hat{X} \to Y$ is a trivial fibration by \cref{item:retract_pb_com} of \cref{lem:basic_pointwise_fibrations}.
By \cref{item:tfib_right_cancel} of \cref{lem:basic_pointwise_fibrations}, we deduce that $X \to Y$ is a trivial fibration.

Let $A \to B$ be a trivial cofibration.
Take a cofibrant replacement $\hat{B} \to B$.
Let $\hat{A} \to A$ be its pullback along $A \to B$.
Then $\hat{A}$ is cofibrant by \cref{cofibration-mono-pullback} since trivial cofibrations are
monomorphisms by \cref{triv-fib-is-fib}, $\hat{A} \to A$ is a trivial fibration by \cref{item:retract_pb_com} of \cref{lem:basic_pointwise_fibrations}, and $\hat{A} \to \hat{B}$ is a trivial cofibration by \cref{frobenius}.
In particular, $\hat{A} \to \hat{B}$ is a cofibrant replacement of $A \to B$.
Since it is a trivial cofibration, it is a weak equivalence in $\cat{E}_\cof$ by \cref{model-cof}.
By \cref{we-invariant}, this makes $A \to B$ is a weak equivalence.

It remains to show that every cofibration that is a weak equivalence is a trivial cofibration.
As in \cref{model-cof}, this follows from what we have already established by the retract argument.

This finishes the verification of \cref{model:two-wfss}.
\Cref{model:between-fibrant,model:between-fibrant-slice} follow since every model structure induces a fibration category structure on its fibrant objects (and those of its slices) and the weak equivalences in a fibration category are determined by its fibrations and trivial fibrations.
In our case, we obtain the fibration categories of \cref{fibcat-Kan,fibcat-fiberwise}.
\end{proof}

By \cref{model:between-fibrant} of \cref{model}, a map is a weak equivalence in the effective model structure if and only if its fibrant replacement is a pointwise weak equivalence.
This gives us a description of weak equivalences independent from the class $\mathbf{W}$ used in the construction of the model structure.

The next remark compares the effective model structure \cref{model} to other model structures on categories of simplicial objects.

\begin{remark} \label{rmk:comparison}
When $\cat E$, and hence $\cats E$, is a locally presentable, then
one can use the enriched small object argument of \cite{RiehlE:catht}*{Chapter~13} to produce the two \wfs{}s on $\cats E$ whose fibrations and trivial fibrations are as in \cref{def:fibration_algebraic}. \cref{fibcat-Kan} then implies that $\cats E$ is a weak model category, for example using the dual of~\cite{Hwms}*{Proposition 2.3.3}. It then follows from~\cite{Hcwms}*{Theorem 3.7} that its left saturation (in the sense of~\cite{Hcwms}*{Theorem 4.1}) is a left semi-model category, and from~\cite{Hcwms}*{Theorem 3.8} that it is also is a right semi-model category. In general, this is not quite enough to conclude that it is a Quillen model category (it is what is called a two-sided model category in~\cite{Hcwms}*{Section~5}), but this is already  sufficient for many applications.

When $\cat E$ is an additive locally presentable category, then there is Quillen model structure on $\cats E$ whose fibrations and trivial fibrations are exactly as in \cref{def:fibration_algebraic}. The additional ingredient in this case is that for $A \in \cat E$ and $X \in \cats E$, the object $\Hom_\sSet(A,X)$ is a simplicial abelian group, hence is always a Kan complex. This shows that when $\cat E$ is additive, all objects of $\cats{E}$ are Kan complexes, hence in the discussion above it is immediate that $\cats E$ is left saturated (in the sense of~\cite{Hcwms}) and as it is a saturated right semi-model category where every object is fibrant it is a Quillen model category.
By the Dold--Kan correspondence, the category $\cats E$ is equivalent to the category of chain complexes
concentrated in non-negative degrees in $\cat E$ and under this equivalence the model structure is the so-called absolute (or Hurewicz) model structure on chain complexes (see, \eg,~\cite{christensen2002quillen}*{Corollary~6.4}).

A different model structure on $\cats E$ is established by Quillen in~\cite{Quillen}*{Section~II.4} assuming that
$\cat E$ has finite limits, enough projectives and is either cocomplete with a small set of generators (thus
permitting the small object argument) or such that every object in $\cats E$ is fibrant.
Quillen's weak equivalences and fibrations include the \emph{pointwise} weak equivalences and fibrations
defined here (as the former are defined using evaluation with respect to projective objects only) and the identity functor is a left Quillen functor
from Quillen's model structure to the effective one.  If effective epimorphisms split, then the two model structures coincide.

A class of model structures on $\cats E$ is also defined in~\cite{Goerss-Jardine}*{Chapter~II}.
The construction is parametrised by a functor~$G \from \cats E \to \cats \Set$ with a left adjoint from which weak equivalences and fibrations are created.
If $\cat E$ is complete and cocomplete and maps with the left lifting property with respect to fibrations are weak equivalences, one obtains a model structure.
This is quite different from the effective model structure and more in the spirit of generalizing~\cite{Quillen}*{Section~II.4}.

Finally, a model structure on $\cats E$ has been obtained also in \cite{Hormann}, which appeared shortly after the first version of the present paper and was developed independently. Theorem~6.1 therein is a special case of our \cref{model}, obtained under the additional assumption that every object
of $\cat E$ is a coproduct of $\mathbb{N}$-small objects (see~\cite{Hormann} for details).
\end{remark}

  \section{Descent and right properness} \label{sec:descent}

Having established the existence of the effective model structure on $\cats E$, we now
study some of its properties and those of its associated $\infty$-category $\Ho_\infty(\cats{E})$.
There are many (essentially equivalent) ways of associating an $\infty$-category to a model category,
and our result will make little use of a concrete details of how it is done beyond some very general results.
For the sake of completeness, when we say $\infty$-category we mean quasicategory, and for a general category $\cat{C}$ equipped with a class of weak equivalences,
we define $\Ho_\infty(\cat{C})$ as the $\infty$-category obtained by universally inverting the weak equivalences in $\cat{C}$.
We refer to \cite{cisinski2020higher}, especially its Chapter 7, for the general theory of such localisations.

We begin by studying the behaviour of colimits, using the notion of descent,
which was introduced in model categories by Rezk~\cite{Rezk} as a part of development of higher topos theory. We show that $\cats E$ and hence $\Ho_\infty(\cats{E})$ satisfies \emph{descent} whenever $\cats{E}$ is countably extensive.
This means that colimits in $\Ho_\infty(\cats{E})$ satisfy the higher categorical version of the van Kampen property. In the case of pushouts, this is spelled out in \cref{ms-descent-pushout} below. As in the ordinary categorical case, a colimit in an $\infty$-category $\cat{C}$ satisfies descent if and only if it is preserved by the functor from $\cat{C}^{op}$ to the $\infty$-category of $\infty$-category classified by the slice cartesian fibration. This is essentially proved in section 6.1.3 of \cite{Lurie}, see for example 6.1.3.9.

\begin{proposition}[Model structure descent for pushouts] \label{ms-descent-pushout}
Let $\cat{E}$ be a countably extensive category and let
\begin{tikzeq}{ms-descent-pushout:cube}
\matrix[diagram,column sep={between origins,4em},row sep={between origins,9ex}]
{
    |(X00)| X_{00} & & |(X01)| X_{01} & \\
  & |(X10)| X_{10} & & |(X11)| X_{11}   \\
    |(Y00)| Y_{00} & & |(Y01)| Y_{01} & \\
  & |(Y10)| Y_{10} & & |(Y11)| Y_{11}   \\
};

\draw[->] (X00) to (X01);
\draw[->] (Y00) to (Y01);

\draw[->] (X00) to (Y00);
\draw[->] (X01) to (Y01);

\draw[->,over] (X10) to (X11);
\draw[->]      (Y10) to (Y11);

\draw[->,over] (X10) to (Y10);
\draw[->]      (X11) to (Y11);

\draw[->] (X00) to (X10);
\draw[->] (X01) to (X11);
\draw[->] (Y00) to (Y10);
\draw[->] (Y01) to (Y11);
\end{tikzeq}
be a cube in $\cats{E}$.
Assume that the bottom face is a homotopy pushout and that the left and back faces are homotopy pullbacks.
Then the following are equivalent:
\begin{conditions}
\item \label{ms-descent-pushout:colim} The top face is a homotopy pushout,
\item \label{ms-descent-pushout:cart} the right and front faces are homotopy pullbacks.
\end{conditions}
\end{proposition}

\begin{proof}
Let us view $[1]$ as a Reedy category consisting only of face operators.
We consider the Reedy model structure $[D^\op, \cats{E}]$ of $\cats{E}$ over the Reedy category $D = [1] \times ([1] \times [1])^\op$.
The significance of taking opposites on the latter two factors is that the Reedy category structure is inverted; the face operators become degeneracy operators.
Recall from the beginning of \cref{sec:model-structure} that we regard only certain (co)limits to be part of a model structure; the theory of Reedy model structures makes sense in this setting as seen in~\cref{sec:cofibrations} for the case of the Reedy weak factorisation system over $\Delta$.

The given cube~\eqref{ms-descent-pushout:cube} forms an object of this category by sending $(0, a, b)$ to $Y_{ab}$ and $(1, a, b)$ to $X_{ab}$.
Recall that weak equivalences in the Reedy model structure are levelwise and homotopy pushouts and pullbacks are invariant under levelwise weak equivalences.
We replace the given cube by a cofibrant and fibrant object.
This reduces the claim to the case of~\eqref{ms-descent-pushout:cube} where all object are cofibrant and fibrant, all horizontal maps are cofibrations, and all vertical maps are fibrations.

Let us check the direction from \ref{ms-descent-pushout:colim} to \ref{ms-descent-pushout:cart}, \ie, universality.
Take the pullback of the bottom face along $X_{11} \fto Y_{11}$.
Since all vertical faces in~\eqref{ms-descent-pushout} are homotopy pullbacks, we obtain a square weakly equivalent to the top face.
This reduces the claim to the situation where in addition all vertical faces in~\eqref{ms-descent-pushout:cube} are pullbacks.
Note that the cofibrancy assumptions are preserved by \cref{pullback-cofibration} of \cref{thm:pullback-combined}.

Denote $Q$ the pushout in the bottom face.
Since $Y_{00} \to Y_{01}$ is a levelwise complemented inclusion (\cref{cofibration-as-retract-of-cell-complex}),
$P$ is a van Kampen pushout by \cref{lem:pushout_decidable}, in particular stable under pullback.
From universality, we obtain a pullback square
\begin{tikzeq}{ms-descent-pushout:0}
\matrix[diagram,column sep={between origins,5em}]
{
  |(P)| P      & |(Q)| Q      \\
  |(X)| X_{11} & |(Y)| Y_{11} \\
};

\draw[->]  (P) to (Q);
\draw[->]  (X) to (Y);
\draw[fib] (P) to (X);
\draw[fib] (Q) to (Y);
\pb{P}{Y};
\end{tikzeq}
where $P$ is the pushout in the top face.
Since $X_{00} \to X_{01}$ and $Y_{00} \to Y_{01}$ are cofibrations, the bottom and top faces are homotopy pushouts exactly if the maps $P \to X_{11}$ and $Q \to Y_{11}$ are weak equivalences, respectively.
The goal thus follows from right properness applied to~\eqref{ms-descent-pushout:0}.

Let us check the direction from \ref{ms-descent-pushout:cart} to \ref{ms-descent-pushout:colim}, \ie, effectivity.
Take the pushout in the horizontal faces.
Since all horizontal maps are cofibrations and the horizontal faces are homotopy pushouts, we obtain a cube weakly equivalent to the given cube.
This reduces the goal to the situation where all horizontal faces in~\eqref{ms-descent-pushout:cube} are pushouts, but note that we lose fibrancy properties involving $X_{11}$ and $Y_{11}$.
The cube is now determined (up to isomorphism) by just the left and back faces.
Weakly equivalent left and back faces give rise to weakly equivalent cubes.

Since the back face is a homotopy pullback and the vertical maps are fibrations, the map $X_{00} \to Y_{00} \times_{Y_{01}} X_{01}$ is a weak equivalence.
We apply the equivalence extension property of \cref{equivalence-extension} to this situation:
\begin{tikzeq*}
\matrix[diagram,column sep={between origins,4em},row sep={between origins,8ex}]
{
    |(X00)| X_{00}                       & & |(X'01)| X'_{01} & \\[-2ex]
  & |(P)|   Y_{00} \pull_{Y_{01}} X_{01} & & |(X01)|  X_{01}    \\
    |(Y00)| Y_{00}                       & & |(Y01)|  Y_{01} &  \\
};

\draw[cof] (Y00)  to node[below] {$i$} (Y01);

\draw[->,dashed]  (X00)  to (X'01);
\draw[fib]        (X00)  to (Y00);
\draw[fib,dashed] (X'01) to (Y01);

\draw[->]        (X00)  to node[above right] {$\weq$} (P);
\draw[->,dashed] (X'01) to node[above right] {$\weq$} (X01);

\draw[fib] (P)   to (Y00);
\draw[fib] (X01) to (Y01);

\draw[->,over] (P) to (X01);
\pbs{P}{Y01};
\end{tikzeq*}
We perform the same construction in the left face, obtaining $X_{10}'$.
Now, the squares
\begin{tikzeq*}
\matrix[diagram,column sep={between origins,5em}]
{
  |(X10)| X'_{10} & |(X00)| X_{00} & |(X01)| X'_{01} \\
  |(Y10)| Y_{10}  & |(Y00)| Y_{00} & |(Y01)| Y_{01}  \\
};

\draw[fib] (X00) to (Y00);
\draw[fib] (X10) to (Y10);
\draw[fib] (X01) to (Y01);

\draw[->] (X00) to (X10);
\draw[->] (X00) to (X01);
\draw[->] (Y00) to (Y10);
\draw[->] (Y00) to (Y01);

\pb{X00}{Y01};
\pbur{X00}{Y10};
\end{tikzeq*}
are weakly equivalent to the left and back faces, but are pullbacks.
We have thus reduced to the situation where additionally the left and back faces of~\eqref{ms-descent-pushout:cube} are pullbacks.

Having strictified the given homotopy pushouts and homotopy pullbacks, we proceed as follows.
The maps $X_{00} \to X_{01}$ and $X_{00} \to X_{10}$ are levelwise complemented inclusions by \cref{cofibration-as-retract-of-cell-complex}.
The bottom pushout is van Kampen by \cref{colimits-diagram-category:pushout} of \cref{colimits-diagram-category}.
In particular, the right and front faces are pullbacks.
For them to be homotopy pullbacks, it suffices for $X_{11} \to Y_{11}$ to be a fibration.
This holds by \cref{colimit-of-fibrations:pushout} of \cref{colimit-of-fibrations}.
\end{proof}

\begin{proposition}[Model structure descent for coproducts] \label{ms-descent-coproduct}
  Let $\cat{E}$ be an $\alpha$-extensive category, $X \to Y$ a morphism in $\cats{E}$ and $S$ an $\alpha$-small set.
  Given a square
  \begin{tikzeq*}
  \matrix[diagram]
  {
    |(Xs)| X_s & |(X)| X \\
    |(Ys)| Y_s & |(Y)| Y \\
  };

  \draw[->] (Xs) to (X);
  \draw[->] (Ys) to (Y);
  \draw[->] (Xs) to (Ys);
  \draw[->] (X)  to (Y);
  \end{tikzeq*}
  for each $s \in S$ \st{} the induced morphism $\bigcoprod_s Y_s \to Y$ is a weak equivalence, \tfae{}:
  \begin{conditions}
  \item the square above is a homotopy pullback for each $s \in S$,
  \item the induced morphism $\bigcoprod_s X_s \to X$ is a weak equivalence.
  \end{conditions}
\end{proposition}

\begin{proof}
This follows from a simpler variant of the previous argument, for $\alpha$-small coproducts instead of pushouts.
This uses \cref{colimit-of-fibrations:coproduct} instead of
\cref{colimit-of-fibrations:pushout} of \cref{colimit-of-fibrations}.
\end{proof}

\Cref{ms-descent-pushout,ms-descent-coproduct} have
an immediate counterpart at the $\infty$-categorical level.

\begin{theorem} \label{thm-descent-infty}
Let $\cat{E}$ be an $\alpha$-extensive category.
The $\infty$-category $\Ho_\infty(\cats{E})$ has all $\alpha$-small colimits.
These colimits satisfy descent.
\end{theorem}

\begin{proof}
  It follows from \cite{cisinski2020higher}*{Proposition 7.5.18} that $\Ho_\infty(\cats{E})$ has finite limits and that finite homotopy limits in $\cats{E}$ are sent to limits in $\Ho_\infty(\cats{E})$, the dual also holds for finite (homotopy) colimits. Moreover, one can deduce the same for $\alpha$-coproducts using \cite{cisinski2020higher}*{Proposition 7.7.1 and Theorem 7.5.30}.
  This, together with \cref{ms-descent-pushout,ms-descent-coproduct} immediately implies that pushouts and $\alpha$-coproducts satisfy descent in $\Ho_\infty(\cats{E})$.
  From there, \cite{Lurie}*{Proposition~4.4.2.6} shows that the existence of finite colimits and $\alpha$-coproducts implies the existence of all $\alpha$-small colimits in $\Ho_\infty(\cats{E})$. And given that a certain colimit satisfies descent if and only if it is preserved by the contravariant functor from $\Ho_\infty(\cats{E})$ to the $\infty$-category of $\infty$-categories classified by the slice fibration, \cite{Lurie}*{Proposition~4.4.2.7} shows that this implies that all $\alpha$-small colimits satisfy descent.
\end{proof}

We now move on to consider right properness of the effective model structure,
which will be the key to transfer local \cartesian closure from $\cat E$ to $\Ho_\infty(\cats{E})$.

\begin{proposition} \label{thm:right-proper}
Let $\cat{E}$ be a countably lextensive category.  The effective model structure on $\cats E$
is right proper.
\end{proposition}

\begin{proof} This follows from \cref{frobenius} using the argument
in~\cite{GSS}*{Proposition 4.1, Second proof}.
\end{proof}

\begin{theorem} \label{thm-lccc-infty} Let $\cat E$ be a countably lextensive category. If $\cat E$ is locally \cartesian closed, then the
$\infty$-category $\Ho_\infty(\cats{E})$ is locally \cartesian closed.
\end{theorem}

\begin{proof} We first observe that if $\cat E$ is countably lextensive and locally \cartesian closed, then $\cats E$ is also locally \cartesian closed. Indeed, if $\cat E$ is countably lextensive then $\cats E$ can be realised as the category of internal presheaves for the category object $\fset{\Simp} \in \cat E$. Such categories of internal presheaves over an internal category in a locally \cartesian closed categories are always locally \cartesian closed. Indeed, this follows from \cite{Elephant}*{Theorem A4.2.1 and Proposition B2.3.16}, using exactly the same argument as in the proof of \cite{Elephant}*{Corollary B2.3.17} (which deals with the similar statement for toposes instead of locally \cartesian closed categories). Note that we are applying these results taking the category $\mathbb{D}$ therein to be the canonical self-indexing of the base category $\cat E$, which satisfies the assumption of having $\cat E$-indexed products because of \cite{Elephant}*{Lemma B1.4.7, part~(iii)} since $\cat E$ is locally \cartesian closed.

An arbitrary map in $\Ho_\infty(\cats{E})$ can always be represented by a fibration $p \from X \to Y$ between fibrant objects in $\cats E$ with $X$ cofibrant.
The functor $p^*$ is a left adjoint functor since $\cats{E}$ is locally \cartesian closed, it preserves cofibrations by \cref{pullback-cofibration} of \cref{thm:pullback-combined} and it preserves trivial cofibrations by the Frobenius property of \cref{frobenius}.
It hence follows from \cite{cisinski2020higher}*{Proposition 7.6.16} that the pullback functor $\Ho_\infty(\cats{E})/Y \to \Ho_\infty(\cats{E})/X$ admits a right adjoint (given by the action of the right adjoint of $p^*$ on fibrant objects).
\end{proof}

We conclude this section by combining our results in the case $\cat E$ is a Grothendieck topos.

\begin{theorem} Let $\cat E$ be a Grothendieck topos. Then $\infty$-category $\Ho_\infty(\cats{E})$
is locally \cartesian closed and has all small colimits, which satisfy descent. \qed
\end{theorem}

For a Grothendieck topos $\cat E$, the effective model structure on $\cats E$ is typically not a model topos in the sense of Rezk~\cite{Rezk} and $\Ho_\infty(\cats{E})$
is not a higher topos in the sense of Lurie~\cite{Lurie}. Indeed, as we will see in \cref{ex:arrow_sets2}, if $\cat E = \Set^{[1]}$, then the category of
0-truncated objects in $\Ho_\infty(\cats{E})$ is neither a Grothendieck topos nor an elementary topos, as it does not have a subobject classifier. The situation
is reminiscent of that of Grothendieck toposes whose exact completion is neither a Grothendieck topos nor an elementary topos~\cite{Menni}.
  \section{A generalised Elmendorf theorem}
\label{sec:Elmendorf}

Elmendorf's theorem~\cites{elmendorf,stephan-elmendorf} states that the genuine equivariant model structure on $G$-spaces is equivalent to the projective model structure on presheaves of spaces on the category of orbits of~$G$. In this section, we show as \cref{th:Elmendorf} that, under the assumption that the category~$\cat E$ is completely lextensive and locally connected (in the sense of \cref{def:locally_connected} below), then the
effective model category structure on $\cats E$ models the $\infty$-category of small presheaves of spaces
on the full subcategory $\cat E^\con$ of connected objects in $\cat E$. Note that extension of Elmendorf's theorem beyond the case of group action already appears in the literature
(\cf~\cites{Chorny2,dwyer1984singular,dror1987homotopy}). The work in~\cite{Chorny2} is especially close to what we prove in the present section.


\begin{definition}\label{def:locally_connected} Let $\cat{E}$ be a lextensive category.
\begin{itemize}
\item An object $X \in \cat{E}$ is said to be \emph{connected} if it is not the initial object and whenever $X= A \coprod B$ then $A = \emptyset$ or $B = \emptyset$.
\item A lextensive category is said to be \emph{locally connected} if every object is a van Kampen coproduct of connected objects.
\end{itemize}
\end{definition}

The terminology of \cref{def:locally_connected} is compatible with the notion of a locally connected Gro\-then\-dieck topos. For example, the category of sheaves of set over a locally connected topological space is locally connected. The category of presheaves over a category $\cat I$ is locally connected, its connected objects are called the ``orbit'' of $I$, i.e., the presheaves whose category of elements is connected, or equivalently whose colimits is a singleton. The coproduct completion of a category with finite limits is also a locally connected category.

Let us now fix a lextensive category $\cat E$.
We denote by $\cat E^\con$ the full subcategory of of $\cat E$ of connected objects. It is important to note that even if $\cat E$ is a Grothendieck topos, this category is in general not a small category, as the next example illustrates.

\begin{example}\label{ex:arrow_sets}
If $\cat E = \Set^{[1]} = \Fam \Set$, then the connected objects of $\cat E$ are the objects of the form $X \rightarrow *$ for an arbitrary set $X$. In particular $\cat E^\con$ is equivalent to the category of all sets.
More generally, if $\cat{C}$ is a category with finite limits, and $\Fam \cat{C}$ is its coproduct completion, then $(\Fam \cat{C})^\con = \cat{C}$.
\end{example}

\begin{lemma}\label{lem:connected=preserves_coprod}
Let $X$ be a connected object in a lextensive category. Then $\Hom_\Set(X, \uvar)$ commutes with van Kampen coproducts.
\end{lemma}

\begin{proof}
  Given a map $f \from X \rightarrow \bigcoprod A_i$, then $X = \bigcoprod X_i$ where $X_i =X \fibprod_A A_i$,
  but as $X$ is connected all the $X_i$ except one are the initial object. As $X$ is itself non-initial, then exactly one of the $X_i$ is non initial
  and hence $X=X_i$ and the map $X \rightarrow \bigcoprod A_i$ factors into $X \rightarrow A_i$ for a unique $i$.
\end{proof}

For a possibly large category  $\cat D$, we write $\psh \cat D$ for the category of small presheaves on $\cat D$, that is the category of presheaves on $\cat D$ that can be written as small colimits of representables. We denote by $\spsh \cat D$ the category of small simplicial presheaves, or equivalently simplicial objects in $\psh \cat D$. In general, limits of small presheaves can fail to be small, but if we assume that $\cat D$ has $\alpha$-small limits, then $\psh \cat D$ also has $\alpha$-small limits. This is proved in~\cite{DayLack} as Theorem~4.3 applied to Example~4.1.1.

\begin{proposition}\label{prop:large_proj_MS}
Let $\cat D$ be a category with finite limits. Then  $\spsh \cat D$ carries the projective model structure, in which an arrow $f:X \rightarrow Y$ if a fibration, trivial fibration or weak equivalence if and only if for all $d \in \cat D$, the arrow $f_d \from X(d) \to Y(d)$ is one.\end{proposition}

\begin{proof}
 This is proved in~\cite{Chorny1} under the assumption that $\cat D$ has all limits. However, the proof applies unchanged if we only assume that $\spsh \cat D$ has finite limits, as long as we do not require that a model category has all limits, but only finite limits. Indeed the only use of limits in $\cat D$ in the proof is to show that $\psh \cat D$ has all limits. Moreover, \cite{DayLack}*{Theorem 4.3 applied to Example 4.1.1} shows that if the category $\cat D$ has finite limits then the category $\psh \cat D$ of small presheaves on $\cat D$ also has finite limits. Note that the existence of the corresponding weak factorisation system in $\spsh \cat D$ follows from the generalised small object argument with respect to locally small class of arrows exactly as explained
 in~\cite{Chorny1}
\end{proof}

The claim of \cref{prop:large_proj_MS} follows also from the assumption that $\spsh \cat D$ has finite limits, which  is a weaker condition than the existence of finite limits in $\cat D$.

\begin{remark}
The $\infty$-category associated to the projective model structure on $\spsh \cat D$ is really the $\infty$-category of small presheaves of spaces on $\cat D$, essentially by the same argument as for small categories.
\end{remark}

\begin{lemma}\label{lem:conected_yoneda}
Given a locally connected countably lextensive category $\cat E$.
\begin{parts}
\item The restricted Yoneda embedding $\yon \from \cat E \rightarrow \psh(\cat E^\con)$ is well-defined, fully faithful and preserves limits and all van Kampen coproducts.
\item The restricted Yoneda embedding $\yon \from \cats E \rightarrow \spsh(\cat E^\con)$ is well-defined, fully faithful and preserves limits, pushouts along a cofibration, tensoring by objects of $\cat E$ and $\sSet$ and colimits of sequences of cofibrations.
\end{parts}
\end{lemma}

\begin{proof}
For any connected object $X \in \cat E^\con$, $\Hom_\Set(X,\uvar)$ preserves coproducts by \cref{lem:connected=preserves_coprod}, hence as every object $Y \in E$ is a small van Kampen coproduct of connected objects, its image under the restricted Yoneda embedding is a small coproduct of representables, and hence is a small presheaf. This proves the existence and the preservation of coproducts by the Yoneda embedding. Preservation of limits is immediate.
It is fully faithful on connected objects by the Yoneda lemma, and this implies that it is fully faithful in general as morphisms between van Kampen coproducts of connected objects can be explicitly described as maps between their components.

The simplicial version is just the ordinary version applied levelwise in the simplicial direction so all results of part~(ii)  follow immediately. For the preservation of colimits we use the fact that a functor that preserves countable coproducts preserves pushouts of complemented inclusions and colimits of sequences of complemented inclusions, and all the colimits considered in the lemma are levelwise of this form.
\end{proof}

\begin{theorem}[Generalised Elmendorf's theorem]\label{th:Elmendorf} Let $\cat E$ a locally connected countably lextensive category.
\begin{parts}
\item\label{th:Elmendorf:detect} A map in $\cats E$ is a cofibration, fibration or weak equivalence if and only if its image by the restricted Yoneda embedding is one for the projective model structure.
\item\label{th:Elmendorf:equiv} If $\cat E$ is in addition completely lextensive, then the restricted
Yoneda embedding induces an equivalence between the full subcategories of cofibrant objects of $\cats E$ and $\spsh(\cat E^\con)$. In particular it induces an equivalence of the corresponding $\infty$-categories.
\end{parts}
\end{theorem}

\begin{proof}
The (cofibration, trivial fibration) and (trivial cofibration, fibration) weak factorisation systems on $\cats E$ are cofibrantly generated
in the (non-enriched) sense of \cite{Chorny1} by the classes of arrows $\set{i \tensorsSetsE E}{i \in I_\sSet, E \in \cat E}$ and
$\set{j \tensorsSetsE E}{j \in I_\sSet, E \in \cat E}$.
As every object in $\cat E$ is assumed to be a (van Kampen) coproduct of connected objects, one can restrict to $E \in \cat E^\con$.
Because of \cref{lem:conected_yoneda}, these generators are sent exactly to the generators of the projective model structure of $\spsh(\cat E^\con)$.

It immediately follows that an arrow in $\cats E$ is a (trivial) fibration if and only if it is one in $\spsh(\cat E^\con)$ as these classes are characterised by the same lifting property.

Moreover, also because of \cref{lem:conected_yoneda} the restricted Yoneda embedding preserves coproducts and pushouts of the generating cofibrations,
transfinite composition of cofibrations and retracts.
Thus because of how (trivial) cofibrations are constructed in $\cats E$ from the small object argument, it follows that their images in $\spsh(\cat E^\con)$
are projective (trivial) cofibrations. Conversely, an arrow in $\cats E$ which is a (trivial) cofibration in the projective model structure on $\spsh(\cat E^\con)$
has the lifting property against all (trivial) fibrations in $\spsh(\cat E^\con)$, but as the restricted Yoneda embedding is fully faithful and
preserves (trivial) fibrations, it follows that it also has the lifting property against all (trivial) fibrations in $\cats E$ and hence is a (trivial) cofibration in $\cats E$.
This proves \cref{th:Elmendorf:detect} for (trivial) cofibrations and (trivial) fibrations,
the case of equivalences also follows as an arrow is an equivalence if and only if it can be factored as trivial cofibration followed by a trivial fibration.

For \cref{th:Elmendorf:equiv}, we just make one additional observation. If $\cat E$ is completely lextensive, then any cofibrant object in $\spsh(\cat E^\con)$ is in the image of the Yoneda embedding. Indeed, the image of $\yon$ contains the initial object and the generating cofibrations, and is closed under pushout of cofibrations, transfinite composition of cofibrations and retract (because it is closed under finite limits). Therefore, it contains all cofibrant objects. So as $\yon$ is fully faithful it is an equivalence of categories between the categories of cofibrant objects.
\end{proof}

In short, \cref{th:Elmendorf} says that if $\cat E$ is completely lextensive and locally connected, the effective model category structure on $\cats E$ of \cref{model}  models the category of small presheaves of spaces on the large category $\cat E^\con$. Note that we cannot quite say that the restricted Yoneda embedding is a Quillen equivalence because it does not admit an adjoint in general. However it  follows from the theorem that if $\cat E$ has all colimits, then it is a right Quillen equivalence. Note that a very general Elmendorf's theorem was also proved in~\cite{Chorny2}*{Theorem 3.1}, which is similar to our version in many aspects. In fact, if we assume that $\cat E$ is both complete and cocomplete then we can deduce our result from Chorny's theorem.

\begin{example}\label{ex:arrow_sets2}
We take $\cat E$ to be the category $\Set^{[1]}$ of arrows in $\Set$.
It is completely lextensive and locally connected, and its connected objects are the ones of the form $X \to Y$ where $Y$ is the singleton.
Thus the category of connected objects can be identified with the category of sets, it hence follows by \cref{th:Elmendorf} that
the category $\Hoi(\cats E)$ can be identified with the category of small presheaves of spaces on the category of all sets. This $\infty$-category satisfies descent (all its colimits are van Kampen) and is locally cartesian closed, for example by \cref{thm-lccc-infty} and \cref{thm-descent-infty}. But, it is not a (locally) presentable $\infty$-category, so is not an $\infty$-topos in the sense of \cite{Lurie}*{Chapter~6}. It is also not an elementary $\infty$-topos in the sense of~\cite{Shulman-ncatcafe} or \cite{Rasekh}, for example its full subcategory of set-truncated objects is the category of small presheaves of sets on the category of all sets and is not an elementary topos as it does not have a subobject classifier. This category of set-truncated objects is however a pretopos (in the infinitary sense of the term) and is locally cartesian closed.
\end{example}

  \section{Semisimplicial objects and left properness}
\label{sec:Semi-simplicial}

In this section, we consider the category of semisimplicial objects $\catss{E}$.
While its homotopy theory is overall less well-behaved than its simplicial counterpart we developed so far, it is in some respects simpler.
This allows us to derive certain properties of $\cats{E}$ that we do not seem to be able to prove otherwise.
In particular, we use these results to show that the model structure on $\cats{E}$ is left proper (\cref{sE-left-proper})
and to establish certain universal property of the $\infty$-category associated with $\cats{E}$ in \cref{sec:the-infty-category}.

Our development will be mostly parallel to the simplicial one.
We will start under the assumption that $\cat{E}$ has finite limits and show that the category of Kan complexes in $\catss{E}$
carries a structure of a fibration category.
If $\cat{E}$ is countably lextensive, the category $\catss{E}$ also carries natural notions of cofibrations and trivial cofibrations,
but these do not fit into a model structure.
(They can be organised into certain weaker structures as discussed below in \cref{weak-model-structures}.)
Nonetheless, we show that they are sufficiently well-behaved for our purposes.
Indeed, a particularly simple characterisation of cofibrations (they coincide with \ldi{}s, see \cref{ex:cofibration_are_coprod_inclusion})
enables certain arguments unavailable in $\cats{E}$.

The critical result that is that the homotopy theories of simplicial and semisimplicial objects in $\cats{E}$ are equivalent (\cref{th:simp_ssimp_equivalence}).
We will show that under the assumption that $\cat{E}$ is either countably complete (\cref{th:simp_ssimp_eq_complete}) or
countably lextensive (\cref{th:simp_ssimp_eq_extensive}).

We begin by introducing some basic concepts.
Since these are largely analogous to the simplicial case, we only treat them briefly, mainly to fix the notation.
We write $\sSimp$ for the subcategory of $\Simp$ consisting of the face operators (i.e., the injective maps) and
$\catss{E} = [\sSimp^\op, \cat{E}]$ for the category of semisimplicial objects in $\cat{E}$.
In particular, $\ssSet$ is the category of semisimplicial sets.
The representable semisimplicial sets are denoted by $\ssimp{n}$.
For any finite semisimplicial set $K$, we define the \emph{evaluation functor} $\ev_K \from \cats{E} \to \cat{E}$ as
\begin{align*}
  \ev_K(X) = \coend_{[n] \in \sSimp} X_n^{K_n} \text{.}
\end{align*}

The category $\ssSet$ caries a non-\cartesian closed symmetric monoidal structure whose tensor is called the \emph{geometric product} and denoted by $\gprod$.
It is uniquely determined by the property that $\ssimp{m} \gprod \ssimp{n}$ is the semisimplicial set of non-degenerate simplices in
the nerve of the poset $[m] \times [n]$.

The forgetful functor $U \from \sSet \to \ssSet$ has both the left adjoint $L$ and the right adjoint $R$ given by
Kan extensions along the inclusion $\sSimp \to \Simp$.
The forgetful functor $U \from \cats{E} \to \catss{E}$ also has the left or the right adjoint if $\cat{E}$ is countably lextensive
(or even just finitely cocomplete) or countably complete, respectively.
These will be used in the proofs of the two variants of this section's main theorem announced above.

The homotopy theory of semisimplicial sets is well established.
Weak homotopy equivalences are defined as semisimplicial maps that become simplicial weak homotopy equivalences upon applying the functor $L$.
The category $\ssSet$ also carries classes of (trivial) fibrations and cofibrations, defined below.
These do not form a model structure, but they satisfy certain weaker axioms.
E.g., $\ssSet$ is a weak model category (and even a right semi-model category), see \cite{H}*{Section~5.5}.
For our purposes, \Cref{ssSet-fibcat} below is sufficient.

For a finite semisimplicial set $K$ and $X \in \catss{E}$ we define the cotensor $K \cotensor X \in \catss{E}$ by letting
\begin{align*}
  (K \cotensor X)_n = X(\ssimp{n} \gprod K)
\end{align*}
and the semisimplicial hom-object
\begin{align*}
  \Hom_\ssSet(X, Y)_n = \Hom_\Set(X, \ssimp{n} \cotensor Y) \text{.}
\end{align*}
Exactly as in the simplicial case, this makes $\catss{E}$ into a $\ssSet$-enriched category \wrt{} the geometric product and
$\cotensor$ becomes the cotensor for this enrichment.

The boundaries $\bdssimp{n}$ and horns $\shorn{n,k}$ are defined analogously to their simplicial counterparts
($\bdssimp{n}$ consists of non-degenerate simplices of $\bdsimp{n}$ and similarly for $\shorn{n,k}$).
This gives rise to the generating sets
\begin{align*}
  I_\ssSet = \{ \bdssimp{n} \to \ssimp{n} \} & \text{ and } J_\ssSet = \{ \shorn{n,k} \to \ssimp{n} \} \text{ in } \ssSet \\
  \text{and } I_\catss{E} = \{ \fset{\bdssimp{n}} \to \fset{\ssimp{n}} \} & \text{ and } J_\catss{E} = \{ \fset{\shorn{n,k}} \to \fset{\ssimp{n}} \}
  \text{ in } \catss{E} \text{.}
\end{align*}
Then a morphism $X \to Y$ in $\cats{E}$ is a \emph{fibration} if the pullback evaluation
\begin{align*}
  X(\ssimp{n}) \to X(\shorn{n,k}) \times_{Y(\shorn{n,k})} Y(\ssimp{n})
\end{align*}
has a section for all horn inclusions $\shorn{n,k} \to \ssimp{n}$ in $J_\ssSet$ and a \emph{trivial fibration} if
\begin{align*}
  X(\ssimp{n}) \to X(\bdssimp{n}) \times_{Y(\bdssimp{n})} Y(\ssimp{n})
\end{align*}
has a section for all boundary inclusions $\bdssimp{n} \to \ssimp{n}$ in $I_\ssSet$.
Similarly, \emph{cofibrations} and \emph{trivial cofibrations} are defined as $I_\catss{E}$-cofibrations and $J_\catss{E}$-cofibrations
in the sense of \cref{I-cofibration}.
Note that fibrations and trivial fibrations defined above coincide with $J_\catss{E}$-fibrations and $I_\catss{E}$-fibrations by
the same argument as in \Cref{fibration-levelwise}.

\begin{lemma}\label{ss-wfs}
  If $\cat{E}$ is countably lextensive, then $\catss{E}$ carries two \ewfs{}s consisting of:
  \begin{itemize}
    \item cofibrations and trivial fibrations,
    \item trivial cofibrations and fibrations.
  \end{itemize}
\end{lemma}

\begin{proof}
  This follows from \cref{esmo} with the assumptions verified exactly as in the proof of \cref{two-ewfss}.
\end{proof}

\begin{theorem}\label{ssSet-fibcat}
  The category of fibrant semisimplicial sets with weak homotopy equivalences as defined above
  (i.e., created by the free functor $L \from \ssSet \to \sSet$) is a fibration category.
\end{theorem}

\begin{proof}[Proof sketch]
  The claim can be deduced from the existence of the fibration category of fibrant simplicial sets in~\cite{GSS}*{Theorem~2.2.2}.
  The proof is analogous to the proof of~\cite{GSS}*{Theorem~2.2.2} itself and depends on the following fact.
  If $f \from X \to Z$ is a map between simplicial sets and $U f$ factors (in semisimplicial sets) as a composite of a cofibration $i \from U X \to B$
  and a fibration $p \from B \to U Z$, then $f$ factors as a composite of $i' \from X \to Y$ and $p' \from Y \to Z$ \st{} $i = U i'$ and $p = U p'$.
  (Note that, in particular, $B = U Y$, $i$ is a cofibration and $p$ is a fibration.)
  This holds by \cite{Steimle}*{Theorem~2.1 and~Addendum~2.2}.
  It will also rely the fact that $U$ preserves and reflects weak equivalences by~\cite{H}*{Lemma~2.2.1}.

  Compared to the proof of~\cite{GSS}*{Theorem~2.2.2}, the present argument requires only two modifications.
  First, to construct a path object on a fibrant semisimplicial set $K$, we first apply the fact above (with $X = \emptyset$, $Y = K$ and $Z = 1$) to obtain
  a simplicial Kan complex $A$ \st{} $U A = K$.
  Then we obtain a path object on $K$ by applying $U$ to a path object on $A$.
  Second, we observe that the facts above imply that a fibration in $\ssSet$ is acyclic \iff{} it is trivial (by reducing it to the same statement in $\sSet$).
  Thus acyclic fibrations are stable under pullback.\footnote{This is non-constructive, because of the use of \cite{Steimle}. An alternative argument which works constructively can be found in \cite{H}*{Theorem 5.5.6}. It shows that semisimplicial set have a weak model structure analogous to the Kan--Quillen model structure. Given that even constructively all semisimplicial sets are cofibrant this is enough to obtain that the full subcategory of fibrant objects is a fibration category.}
 \end{proof}

\begin{lemma}\label{ex:cofibration_are_coprod_inclusion}
  A map $f \from X \to Y$ in $\catss{E}$ is a cofibration \iff{} for all $n$ the map $X_n \to Y_n$ is a \di{}.
  In particular, every object of $\catss{E}$ is cofibrant.
\end{lemma}

\begin{proof}
  The claim follows already from the semisimplicial version of \cref{cof-as-reedy-di} since latching objects are empty, which is simpler to prove than \cref{cof-as-reedy-di}  due to absence of degeneracy operators.
\end{proof}

\begin{corollary} \label{lem:triv_fib_have_section}
  If $\cat{E}$ has finite limits, then every trivial fibration in $\catss{E}$ admits a section.
\end{corollary}

\begin{proof}
  First, note that if $\cat{E}$ is countably lextensive,
  this follows from \cref{ss-wfs,ex:cofibration_are_coprod_inclusion}.
  If $\cat{E}$ is merely finitely complete, then $\Fam_{\omega_1} \cat{E}$ is countably lextensive and
  the conclusion holds since the functor $\catss{E} \to \catss{\Fam}_{\omega_1} \cat{E}$ is fully faithful,
  \cf the explicit construction of $\Fam_\alpha$ in \cref{thm:example-of-lextensive}.
\end{proof}

A morphism $X \to Y$ between Kan complexes in $\catss{E}$ is a \emph{pointwise weak equivalence} if
\[
\Hom_{\ssSet}(E,X) \to \Hom_{\ssSet}(E,Y)
\]
is a weak equivalence in $\ssSet$ for all $E \in \cat{E}$.

\begin{theorem}\label{th:ss_fib_cat}
  Pointwise weak equivalences, fibrations and trivial fibrations equip the category of Kan complexes in~$\catss{E}$ with the structure of a fibration category.
\end{theorem}

\begin{proof}
  The proof is entirely analogous to the proof of \cref{fibcat-Kan} except for the construction of path objects.
  A path object on $X \in \catss{E}$ can be constructed as $X \to \ssimp{1} \cotensor X \fto X \times X$ as before.
  However, there is no semisimplicial map $\ssimp{1} \to \ssimp{0}$
  (i.e., $\ssimp{0}$ does not admit a cylinder object) and so the morphism $X \to \ssimp{1} \cotensor X$
  cannot be induced by functoriality of cotensors.
  The problem can be fixed by constructing a ``weak cylinder object'' on $\ssimp{0}$ in the sense of~\cite{Hwms}.

  There is a unique map $\shorn{2,2} \to \ssimp{1}$.
  It sends both $1$-simplices to the unique $1$-simplex of $\ssimp{1}$.
  We define $D$ to be the pushout of this map along the trivial cofibration $\shorn{2,2} \to \ssimp{2}$:
  \begin{tikzeq*}
    \matrix[diagram,column sep={between origins,6em}]
    {
      |(h)|  \shorn{2,2} & |(s1)| \ssimp{1}  \\
      |(s2)| \ssimp{2}   & |(D)|  D \text{.} \\
    };

    \draw[->] (h)  to (s2);
    \draw[->] (s1) to (D);
    \draw[->] (h)  to (s1);
    \draw[->] (s2) to (D);
  \end{tikzeq*}
  Thus $D$ has two $0$-simplices $b$ and $x$, two $1$-simplices $f \from b \to x$ and $e \from b \to b$ and
  a unique $2$-simplex that witnesses that $f \circ e \sim e$.
  Informally speaking, this forces $e$ to behave as an ``identity cell'' of~$b$.
  More precisely, we obtain a diagram
  \begin{tikzeq*}
    \matrix[diagram,column sep={between origins,6em}]
    {
      |(b)|  \bdssimp{1} & |(s0)| \ssimp{0} \\
      |(s1)| \ssimp{1}   & |(D)|  D         \\
    };

    \draw[->]   (b)  to (s1);
    \draw[tcof] (s0) to node[right] {$b$} (D);
    \draw[->]   (b)  to (s0);
    \draw[->]   (s1) to node[below] {$e$} node[above] {$\weq$} (D);
  \end{tikzeq*}
  which upon cotensoring into $X \in \catss{E}$ yields
  \begin{tikzeq*}
    \matrix[diagram,column sep={between origins,6em}]
    {
      |(b)|  X \times X            & |(s0)| X                      \\
      |(s1)| \ssimp{1} \cotensor X & |(D)|  D \cotensor X \text{.} \\
    };

    \draw[->]   (s1) to (b);
    \draw[tfib] (D)  to (s0);
    \draw[->]   (s0) to (b);
    \draw[->]   (D) to node[below] {$\weq$} (s1);
  \end{tikzeq*}
  When $X$ is a Kan complex, the right vertical morphism is a trivial fibration and hence it has a section
  by \cref{lem:triv_fib_have_section}.
  We obtain the required factorisation by composing $D \cotensor X \weto \ssimp{1} \cotensor X$ with such section. This last map is a pointwise weak equivalence, because applying $\Hom_{\ssSet}(E,\uvar)$ to it gives, up to isomorphism, the map
  \[ D \cotensor \Hom_{\ssSet}(E,X) \to \ssimp{1} \cotensor \Hom_{\ssSet}(E,X) \]
  which is a semisimplicial weak equivalence for each fibrant semisimplicial set $\Hom_{\ssSet}(E,X)$, for example because both evaluation maps to $\Hom_{\ssSet}(E,X)$ are trivial fibrations as the weak factorisation systems on $\ssSet$ are compatible to the monoidal structure on $\ssSet$ (see for eg. Theorem~5.5.6.(iii) of \cite{H}).
\end{proof}

The following theorem is the main result of this section.
It is valid under two separate sets of assumptions which require two independent proofs.
Thus we will consider them separately as \cref{th:simp_ssimp_eq_complete} and \cref{th:simp_ssimp_eq_extensive}.

\begin{theorem}\label{th:simp_ssimp_equivalence} If $\cat E$ is either countably lextensive or countably complete, then the forgetful functor $\cats E \to \catss E$ induces an equivalence of fibration categories between the fibration categories of \cref{fibcat-Kan,th:ss_fib_cat}.
\end{theorem}

We start with the case of a category $\cat E$ with countable limits, this is the proof that relies on the adjunction $U \dashv R$.

\begin{proposition} If $\cat E$ is countably complete, then the forgetful functor $U \from \cats E \to \catss E$
has a right adjoint $R$. Moreover, for every object $E \in \cat E$, evaluation at $E$ commutes with this right adjoint, \ie, the square
\begin{tikzeq*}
  \matrix[diagram]
  {
    |(ssE)| \catss{E} & |(sE)| \cats{E} \\
    |(ssS)| \ssSet    & |(sS)| \sSet    \\
  };

  \draw[->] (ssE) to node[above] {$R$}     (sE);
  \draw[->] (ssS) to node[below] {$R$}     (sS);
  \draw[->] (sE)  to node[right] {$\ev_E$} (sS);
  \draw[->] (ssE) to node[left]  {$\ev_E$} (ssS);
\end{tikzeq*}
commutes (up to canonical isomorphism).
\end{proposition}

\begin{proof}
We claim that for any $X \in \catss E$, seen as a functor $\sSimp^\op \to E$, its right Kan extension along $\sSimp^\op \to \Simp^\op$ exists and is a pointwise right Kan extension. Indeed, the pointwise right Kan extension computed at $[n] \in \Simp$ should be
\[
RV = \lim_{[m] \to [n] \in E} V([m])
\]
where $E$ is the comma category of $[m] \in \sSimp^\op$ endowed with a map $[m] \to [n]$ in $\Simp$. This category is countable, so as $\cat E$ is countably complete, the limit exists, and hence the pointwise right Kan extension exists. By definition taking this right Kan extension is right adjoint to the forgetful functor $\cats E \to \catss E$, so this proves the existence of the right adjoint. The commutation of the square in the proposition is because the evaluation functor preserves limits, and hence preserves this pointwise right Kan extension as well.
\end{proof}

\begin{theorem}
  \label{th:simp_ssimp_eq_complete} If $\cat{E}$ is countably complete, then both the forgetful functor and its right adjoint
  \begin{equation*}
    U : \cats E \leftrightarrows \catss E : R
  \end{equation*}
  restrict to equivalences of fibration categories between $\cats{E}_\fib$ and $\catss{E}_\fib$.
\end{theorem}

\begin{proof}
  The theorem is valid for simplicial and semisimplicial sets, i.e., in the case of $\cat{E} = \Set$.
  As both $U$ and $R$ commute with evaluation at $E \in \cat{E}$ and weak equivalences and fibrations are
  detected by these evaluations, it follows that:
  \begin{itemize}
  \item $U$ and $R$ preserve fibrant objects and are morphisms of fibrations categories;
  \item the unit and counit of the adjunctions are weak equivalences on fibrant objects. \qedhere
  \end{itemize}
\end{proof}

We now move to the case of a countably lextensive category $\cat{E}$.
Despite the fact that the theorem concerns only the fibrant objects of $\catss{E}$,
the proof will depend on the homotopy theory of all, not necessarily fibrant, semisimplicial objects in $\cat{E}$.
We define a general morphism of $\catss{E}$ to be a weak equivalence if it has a fibrant replacement (as constructed from factorisations of \Cref{ss-wfs})
that is a pointwise weak equivalence in $\catss{E}_\fib$.
This is analogous to the characterisation of weak equivalences between simplicial objects in the model structure of \Cref{model}.
The weak equivalences, fibrations and cofibrations defined in this section do not form a model structure on $\catss{E}$,
but we can still prove that they are sufficiently well-behaved for our purposes.
For example, the definition of weak equivalences immediately implies that trivial cofibrations are weak equivalences.
On the other hand, not all trivial fibrations are weak equivalences.

\begin{remark}\label{weak-model-structures}
  If $\cat{E}$ is countably lextensive then $\catss{E}$ is a weak model category in the sense of \cite{Hwms} with weak equivalences, fibrations and cofibrations as defined above.
  This can be derived from (the dual of) \cite{Hwms}*{Proposition~2.3.3} and properties of the classes established in this section. In fact, as every object of $\catss{E}$ is cofibrant, this is even a right semi-model category, as long as we use the definition of a semi-model category in~\cite{Fresse} and not that in~\cite{Spitzweck} (see~\cite{Hcwms}*{Section~3} for the explanation of differences between the two definitions). Our discussion of homotopy theory of semisimplicial objects can be phrased both in terms of this weak model structure or right semi-model structure. However, we prefer to provide more elementary arguments to make this section more self-contained.
\end{remark}

\begin{proposition}
  If $\cat E$ has finite coproducts, then the forgetful functor $\cats{E} \to \catss{E}$ has a left adjoint.
  It is given by
  \begin{equation*}
    (L X)_n = \bigcoprod_{[n] \sto [m]} X_m
  \end{equation*}
  where the coproduct is over all degeneracy operators $[n] \sto [m]$ in $\Simp$.
\end{proposition}

\begin{proof}
  The functor $L$ is the left Kan extension along $\sSimp \to \Simp$.
  If it can be computed pointwise, it is given by the formula
  \begin{equation*}
    (L X)_n = \colim_{[n] \to [m]} X_m
  \end{equation*}
  where the colimit is taken over the comma category $[n] \slice \sSimp^\op$.
  (Its objects are arbitrary simplicial operators $[n] \to [m]$, but its morphisms are just the face operators.)
  It follows from the existence of the degeneracy/face unique factorisation system in $\Simp$ that
  the discrete category of degeneracy operators $[n] \sto [m]$ is cofinal in this category.
  Hence the colimit above can be rewritten as the the coproduct in the statement of the proposition.
  Thus if $\cat E$ has finite coproducts, this colimit exists which concludes the proof.
\end{proof}

\begin{lemma}\label{free-preserves-cofibrations}
  The free functor $L \from \catss{E} \to \cats{E}$ preserves cofibrations and trivial cofibrations.
\end{lemma}

\begin{proof}
It can be checked easily that the natural transformation from the initial functor to $L$ satisfies the assumptions of \cref{pushout-application-instance2}, so it is enough to verify that $L$ sends the generating cofibrations and trivial cofibrations to cofibrations and trivial cofibrations, respectively. These generators are of the form $\fset{\shorn{n,k}} \cto \fset{\ssimp{n}}$ or $\fset{\bdssimp{n}} \cto \fset{\ssimp{n}}$ the image by $L$ is computed as in $\Set$, thus giving $\fset{\horn{n,k}} \cto \fset{\simp{n}}$ or $\fset{\bdsimp{n}} \cto \fset{\simp{n}}$, \ie, the generating cofibrations and trivial cofibrations in $\cats{E}$.
\end{proof}


\begin{lemma}\label{lem:Forget_pres_(triv)cof}
  The forgetful functor $U \from \cats{E} \to \catss{E}$ preserves cofibrations and trivial cofibrations.
\end{lemma}

\begin{proof}
  The forgetful functor preserves all colimits that exist so it is enough to show that
  the generating (trivial) cofibrations of $\cats{E}$ are sent to (trivial) cofibrations.
  The case of cofibrations follows from \cref{lem:cof_charac,ex:cofibration_are_coprod_inclusion}.
  For trivial cofibrations, note that if $X \in \sSet$, then $\fset{U X} = U \fset X$
  (the first $U$ is the forgetful functor $\sSet \to \ssSet$, the second one is $\cats{E} \to \catss{E}$).
  Thus it is enough to show that $\fset{U \horn{k,n}} \to \fset{U \simp{n}}$ is
  a trivial cofibration in $\catss{E}$ for all $0 \le k \le n$.
  For this it is sufficient to show that $\horn{k,n} \to U \simp{n}$ is a trivial cofibration in $\ssSet$ which
  was proven in~\cite{Hwms}*{Corollary~5.5.15~(ii)}.
\end{proof}

Note that the forgetful functor $U$ preserves trivial fibrations, but trivial fibrations in $\catss{E}$ are not necessarily weak equivalences.
Nonetheless, the following statement is valid.

\begin{lemma}\label{lem:forget_send_tcof_to_weq}
  The forgetful functor $U \from \cats{E} \to \catss{E}$ sends trivial fibrations to weak equivalences.
\end{lemma}

\begin{proof}
  This follows by the same argument as the second part of~\cite{H}*{Lemma~2.2.1}.
\end{proof}

\begin{lemma}\label{lem:unit_acyclic}
  For each $X \in \catss{E}$, the unit $X \to U L X$ is a trivial cofibration.
\end{lemma}

\begin{proof}
  The composite $U L$ preserves all the relevant colimits, so it is enough to check that
  for each generating cofibration $\fset{\bdssimp{n}} \to \fset{\ssimp{n}}$, the map
  \begin{equation*}
    U L(\fset{\bdssimp{n}}) \coprod_{\fset{\bdssimp{n}}} \fset{\ssimp{n}} \to UL \fset{\ssimp{n}}
  \end{equation*}
  is a trivial cofibration.
  It then follows from \cref{pushout-application-instance2} that the same holds for all cofibrations and
  the case of $\emptyset \to X$ concludes the proof.
  Thus it suffices to prove the statement in the case of semisimplicial sets which is~\cite{Hwms}*{Proposition~5.5.14}.
\end{proof}

\begin{proposition}\label{U-detects-we}
  The forgetful functor $U \from \cats{E} \to \catss{E}$ preserves and reflects weak equivalences.
\end{proposition}

\begin{proof}
  The conclusion is valid for $\cats{E} = \sSet$ by~\cite{H}*{Lemma~2.2.1} and thus it holds for morphisms between fibrant objects.
  Indeed, $\Hom_{\ssSet}(E, U X) = U \Hom_{\sSet}(E, X)$ and weak equivalences between fibrant objects in both $\cats{E}$ and $\catss{E}$ are detected
  by pointwise evaluation.

  For a general morphism $X \to Y$, we consider its fibrant replacement as constructed in the small object argument.
  Since $U$ preserves trivial cofibrations (by~\Cref{lem:Forget_pres_(triv)cof}) and fibrations, it follows that it preserves such fibrant replacements.
  Thus the conclusion follows from the special case of morphisms between fibrant objects.
\end{proof}

\begin{corollary}\label{counit-we}
  For each $X \in \cats{E}$, the counit $L U X \to X$ is a weak equivalence.
\end{corollary}

\begin{proof}
  This follows from the triangle identities using \cref{lem:unit_acyclic,U-detects-we}.
\end{proof}

\begin{theorem}\label{th:simp_ssimp_eq_extensive}
  When $\cat{E}$ is countably lextensive, the functor $U \from \cats{E}_\fib \to \catss{E}_\fib$ is an equivalence
  of fibration categories.
\end{theorem}

\begin{proof}
  Consider the functor $L' \from \catss{E}_\fib \to \cats{E}_\fib$ obtained by composing $L$ with
  a chosen fibrant replacement functor in $\cats{E}$.
  Such fibrant replacement along with the unit of the adjunction $L \adj U$ induce a natural transformation
  $\id_{\catss{E}_\fib} \to U L'$ which is a weak equivalence by \cref{lem:unit_acyclic,U-detects-we}.
  Similarly, using the counit we obtain two natural transformations $L' U X \leftarrow L U X \to X$ for $X \in \cats{E}$.
  They are weak equivalences by definition and by \cref{counit-we}, but $L U$ is not an endofunctor of $\cats{E}_\fib$,
  just of $\cats{E}$.
  However, we can apply a functorial factorisation to the morphism $L U X \to L' U X \times X$ to obtain
  a weak equivalence $L U X \weto T X$ and a fibration $T X \fto L' U X \times X$.
  Then $T$ is an endofunctor of $\cats{E}_\fib$ and we have
  two natural weak equivalences $L' U \leftarrow T \to \id_{\cats{E}_\fib}$ as required.
\end{proof}

\begin{corollary}\label{sE-left-proper}
Let $\cat{E}$ be a countably lextensive category.  Then the effective model structure on $\cats{E}$ is left proper.
\end{corollary}

\begin{proof}
  This follows by the combination of the following facts. First, the functor $L U \from \cats{E} \to \cats{E}$ preserves colimits;
  secondly, $L U$ preserves cofibrations by \cref{free-preserves-cofibrations,lem:forget_send_tcof_to_weq};
  thirdly, $L U$ takes values in cofibrant objects by \cref{ex:cofibration_are_coprod_inclusion,free-preserves-cofibrations};
  and, finally, the counit $L U X \to X$ is a weak equivalence by \cref{counit-we}.
\end{proof}

  \section{The \texorpdfstring{$\infty$}{infinity}-category \texorpdfstring{$\Hoi (\cats{E}_\fib)$}{Ho \textunderscore infty(sE \textunderscore fib)}.}
\label{sec:the-infty-category}

\Cref{sec:Elmendorf} provides a description of the $\infty$-category $\Hoi \cats{E}$ presented by the effective model structure on $\cats{E}$
when $\cat{E}$ is completely lextensive and locally connected.
The goal of this section is to give an alternative characterisation of this $\infty$-category under fewer assumptions on $\cat E$.
As shown in \cref{sec:fib_cat}, if $\cat{E}$ is only a category with finite limits, we already have a fibration category structure on $\cats{E}_\fib$,
which, in the case where $\cat{E}$ is countably lextensive corresponds to the category of fibrant objects of the effective model structure
hence models the same $\infty$-category.
We will consider the more general problem of describing the $\infty$-category $\Hoi \cats{E}_\fib$ in this case.

We do not know such description for a general category $\cat{E}$ with finite limits, but we will present an answer that applies when $\cat{E}$ is
either countably complete or countably lextensive.
More precisely, we will give a description of the $\infty$-category $\Hoi \catss{E}_\fib$,
which we showed in \cref{sec:Semi-simplicial} is equivalent to $\Hoi \cats{E}_\fib$ when $\cat E$ is either countably lextensive or countably complete.

\begin{theorem}\label{th:Joyal_Szumilo_conjecture}
Let $\cat E$ be a category that is either countably complete or countably lextensive.
Then, evaluations at all $E \in \cat{E}$ induce a fully faithful embedding of $\Hoi(\cats{E}_\fib)$ into the category of presheaves of spaces over $\cat E$.
More precisely, $\Hoi(\cats{E}_\fib)$ is equivalent to the full subcategory of presheaves of spaces over $\cat{E}$
that are homotopy colimits (geometric realisations) of Kan complexes in $\cat{E}$.
\end{theorem}

This is closely related to the exact completion (or ex/lex completion) of $\cat{E}$.
In general, the exact completion (see, \eg,~\cite{Carboni-Vitale}) of a category $\cat{E}$ with finite limits can be described as
the full subcategory of $\psh \cat{E}$ of objects that can be written as colimits of ``setoids objects'' in $\cat E$, \ie, as coequalisers of ``proof-relevant equivalence relations'', that is diagrams $R \rightrightarrows X$ in $\cat{E}$, such that the image of the map
\begin{equation*}
  \Hom_\Set(E,R) \to \Hom_\Set(E,X) \times \Hom_\Set(E,X)
\end{equation*}
is an equivalence relation on $\Hom_\Set(E,X)$ for each $E \in \cat{E}$.
The term ``proof-relevant'' refers to the fact that we do not assume that $R \to X \times X$ is a monomorphism, or equivalently that $\Hom_\Set(E,R)$ is a subset of $\Hom_\Set(E,X) \times \Hom_\Set(E,X)$.
The fact that $R \to X \times X$ is a proof-relevant equivalence relation can be encoded as a structure consisting of morphisms in $\cat{E}$ witnessing transitivity ($R \times_X R \to R$), symmetry ($R \to R$) and reflexivity ($X \to R$). \Cref{prop:fibration_pointwise} can be seen as a higher categorical version of this observation, i.e.,
Kan simplicial objects are a higher categorical generalisation of proof-relevant equivalence relations.
In fact, it is easy to deduce from the theorem above that the full subcategory of set-truncated objects in $\Hoi(\cats{E}_\fib)$ is equivalent to
the ex/lex completion of $\cat E$.

However, it does not seem accurate to think of $\Hoi(\cats{E}_\fib)$ as the $\infty$-categorical version of the ex/lex completion.
Let us say that an $\infty$-category is exact if it has finite limits and quotients of groupoid objects exist and are van Kampen colimits.
Lurie has shown that this condition together with complete lextensivity and local presentability characterises $\infty$-toposes~\cite{Lurie}.
We can then define the ex/lex completion of an $\infty$-category $\cat{C}$ with finite limits in the usual way:
it is an exact $\infty$-category $\cat C^{\text{ex/lex}}$ with a functor $\cat{C} \to \cat{C}^{\mathrm{ex/lex}}$ such that
any finite limit preserving functor to an exact $\infty$-category $\cat{C} \to \cat{D}$ extends essentially uniquely to
an exact functor $\cat{C}^{\text{ex/lex}} \to \cat{D}$.
We conjecture that the effective model structure is related to this ex/lex completion operation in the following way:

\begin{conjecture} \label{thm:conj-exlex} Let $\cat E$ be a countably lextensive category or countably complete category. The ex/lex completion of the $\infty$-category associated to $\cat E$ is equivalent to the full subcategory of $\Hoi(\cats{E}_\fib)$ on objects that are $n$-truncated for some $n$.
\end{conjecture}

More generally, we believe that this holds for any finitely complete category $\cat E$ when $\Hoi(\cats{E}_\fib)$ is replaced with $\Hoi(\catss{E}_\fib)$.

The general idea of the proof of \cref{th:Joyal_Szumilo_conjecture} is that the category $\Fam \cat{E}$ of families of objects of $\cat{E}$ is always
a completely lextensive locally connected category, such that $\cat{E}$ can be identified with its category of connected objects.
Hence, we can apply \cref{th:Elmendorf} to it and show that
\begin{equation*}
  \Hoi (\sFam \cat{E}) \simeq \Hoi \spsh \cat{E}
\end{equation*}
The right hand side is a model for the $\infty$-category of small presheaves of spaces on $\cat{E}$ (in the $\infty$-categorical sense).
We always have a fully faithful embedding $\cats{E} \to \sFam \cat{E}$ which identifies $\cats{E}$ with the full subcategory of levelwise connected simplicial objects.
Moreover, a map is a fibration or a weak equivalence (between fibrant objects) in $\cats{E}$ if and only if its image in~$\sFam \cat E$ is one,
so this embedding also restricts to a morphism of fibration categories.

Our goal is to show that (under the assumptions of \cref{th:Joyal_Szumilo_conjecture}) this also induces a fully faithful embedding on the level of the $\infty$-categories.
Unfortunately, we are able to give a proof of this only when we consider instead the semisimplicial version of this embedding $\catss{E} \to \ssFam \cat{E}$.
But as $\Fam \cat E$ is always countably lextensive we have an equivalence of fibration categories $\sFam \cat{E} \simeq \ssFam \cat{E}$
by \cref{th:simp_ssimp_eq_extensive}, and as soon as $\cat{E}$ is countably complete or countably lextensive
we have an equivalence $\Hoi(\catss{E}_\fib) \simeq \Hoi(\cats{E}_\fib)$ by \cref{th:simp_ssimp_equivalence}.
So we need to show that $\catss{E} \to \ssFam \cat{E}$ induces a fully faithful functor between the corresponding $\infty$-categories.
Because of the following lemma, it is enough to prove that it is fully faithful at the level of the homotopy categories.

\begin{lemma}\label{lem:hom_cat_fully_faithful} A finite limit preserving functor between two $\infty$-categories which is an equivalence (\resp fully faithful) on the homotopy categories is an equivalence (\resp fully faithful).
\end{lemma}

\begin{proof}
  This is shown for the case of equivalences in~\cite{cisinski2020higher}*{Theorem~7.6.10}.
  The case of fully faithful functors can be deduced from the case of equivalences.
  Let $f \from \cat{X} \to \cat{Y}$ be a finite limit preserving functor which is fully faithful on the homotopy category and let $\cat{Y}'$ denote its essential image.
  Then $\cat{Y}'$ contains the terminal object since $f$ preserves finite limits.
  Similarly, $\cat{Y}'$ is closed under pullbacks.
  Indeed, since $f$ is fully faithful on the homotopy categories, any cospan in $\cat{Y}'$ can be lifted to a cospan in $\cat{X}$.
  Its pullback exists in $\cat{X}$ and is preserved by $f$.
  It follows that $f$ induces a finite limit preserving functor $\cat X \to \cat Y'$ which is fully faithful and essentially surjective on the homotopy categories,
  so it is an equivalence, and hence by the result above, $f$ induces an equivalence between $\cat{X}$ and $\cat{Y}'$, i.e., it is fully faithful.
\end{proof}



\begin{theorem}\label{th:semisimplicial_JS}
For any category $\cat E$ with finite limits, the functor $\catss{E}_\fib \to (\ssFam \cat{E}) _\fib$ is fully faithful on the homotopy categories.
\end{theorem}

\begin{proof}
  The homotopy category of $\catss{E}_\fib$ is the quotient by the homotopy relation defined via maps $X \to \ssimp{1} \cotensor Y$.
  This follows since all semisimplicial objects are cofibrant and $\catss{E}_\fib$ is a path category in the sense of~\cite{berg2018exact}.
  The functor $\catss{E}_\fib \to (\ssFam \cat{E}) _\fib$ preserves finite limits and hence it preserves cotensors by $\ssimp{1}$.
  Thus morphisms in $\catss{E}_\fib$ are homotopic in $\catss{E}_\fib$ if and only if they are homotopic in $\ssFam \cat{E}_\fib$.
\end{proof}

\begin{remark}
  The crucial difference between semisimplicial and simplicial settings is that every semisimplicial object in $\catss{E}$ is cofibrant in $\ssFam \cat{E}$.
  However, a non-constant simplicial object in $\cats{E}$ is levelwise connected in $\sFam \cat{E}$ and thus not cofibrant by \Cref{lem:cof_charac}.
\end{remark}

We are now ready to prove \cref{th:Joyal_Szumilo_conjecture}.

\begin{proof}[Proof of \cref{th:Joyal_Szumilo_conjecture}]
We always have a diagram of functors:
\begin{tikzeq*}
  \matrix[diagram,column sep={between origins,10em}]
  {
     |(A)| \cats{E}_\fib & |(B)| \catss{E}_\fib \\
     |(C)| (\sFam \cat E)_\fib & |(D)| (\ssFam \cat E)_\fib \\
  };

  \draw[->] (A) to (B);
  \draw[->] (A) to (C);
  \draw[->] (B) to (D);
  \draw[->] (C) to (D);
  \end{tikzeq*}

\Cref{th:simp_ssimp_eq_extensive} shows that the bottom horizontal functor is always an equivalence of the homotopy categories
as $\Fam \cat{C}$is always a completely lextensive category.
The top horizontal map is also an equivalence on the homotopy categories by \cref{th:simp_ssimp_equivalence} since $\cat{E}$ is countably lextensive or countably complete.
Finally, we have shown in \cref{th:semisimplicial_JS} that the right vertical functor is fully faithful on the homotopy categories.
It follows that the left vertical functor is also fully faithful on the homotopy categories, and hence by \cref{lem:hom_cat_fully_faithful} induces
a fully faithful embedding of $\infty$-categories $\cats{E}_\fib \to (\sFam \cat{E})_\fib$.

Now, $\Fam \cat E$ is a locally connected completely lextensive category, and $\cat{E}$ is its category of connected objects.
Hence, by \cref{th:Elmendorf}, the $\infty$-category $\Hoi(\sFam \cat{E})_\fib$ is equivalent to the category of presheaves of spaces over $\cat{E}$,
which proves the first half of the theorem.

For the description of the essential image we simply investigate the precise nature of the embedding constructed above. If $X \in \cats{E}_\fib$ then its image in $(\sFam \cat E)_\fib$ is also fibrant, and the objects corresponding to $E \in \cat E$ are cofibrant, so, as this is a simplicial model category, the Hom space in the corresponding $\infty$-category between them is simply $\Hom_\sSet(E,X)$. Hence $X$ is sent to the presheaf of spaces $E \mapsto \Hom_\sSet(E,X)$. Note that as colimits in presheaf categories are computed levelwise and the colimit of a simplicial set in the $\infty$-category of spaces is the the spaces represented by this simplicial sets, this can equivalently be expressed as the fact that $X$ is sent to its geometric realisation in the presheaf category.
\end{proof}

  \appendix

\renewcommand\thesection{Appendix~\Alph{section}}

  \section{Remarks on constructivity}
\label{app:constructivity}

While the present paper has been written within ZFC for simplicity, many of our results
and proofs are constructive, \ie, do not rely on the law of excluded middle or the axiom of choice,
subject to some clarifications, which we will discuss briefly here.

First of all, in the constructive reading of the paper, a \emph{finite set} means a finite cardinal, or a finite decidable set, \ie, a set equipped with a bijection with $\{1,\dots,n\}$, for some $n \in \mathbb{N}$.  A \emph{countable set} is a set which is equipped with  a bijection with either $\{1,\dots,n\}$ or~$\mathbb{N}$. With this definition, a countable coproduct of countable sets is countable.

Secondly, we restrict ourselves to consider finitely lextensive, countably lextensive and completely lextensive categories. Here,
by a \emph{finitely lextensive} category we mean a category with a strict initial object and van Kampen binary coproducts and by a \emph{countably lextensive} category we mean a finitely lextensive category that in addition has $\mathbb{N}$-indexed van Kampen coproducts. With this definition, the category of countable sets is countably lextensive. Without these changes, we would run into problems as $\omega_1$ is not a regular cardinal in ZF and the axiom of countable choice is needed to show that a countable union of countable set is countable, and therefore \cref{thm:lextensive} would be problematic, as we could not show that the category of countable
sets is countably lextensive.

Finally, one should assume the convention that every time we discuss existence of an object, this involves explicit structure, rather than a mere property. For example, when we say that a map $f$ has the left lifting property against $g$, we mean that $f$ comes equipped with a function that associates a solution to each lifting problem.

We  make no claim on whether it is possible to make the results in \cref{sec:Elmendorf} and \cref{sec:the-infty-category}  constructive. Indeed, both of these sections involve $\infty$-categories, for which a constructive theory has not been developed yet. Also, \cref{lem:connected=preserves_coprod} is non-constructive: in a constructive setting, its conclusion should be taken as the definition of a connected object. Finally, \cref{sec:Elmendorf} relies on the existence of the projective model structure on the category of small presheaves on a large category, which is not known to exist constructively.

\smallskip

  {
  \small
  \noindent
  N.~Gambino, \textsc{University of Leeds}, \texttt{N.Gambino@leeds.ac.uk} \\[1ex]
  S.~Henry, \textsc{University of Ottawa}, \texttt{shenry2@uottawa.ca} \\[1ex]
  C.~Sattler, \textsc{Chalmers University of Technology}, \texttt{sattler.christian@gmail.com} \\[1ex]
  K.~Szumi\l{}o, \textsc{University of Leeds}, \texttt{K.Szumilo@leeds.ac.uk}
 }

\end{document}